\documentclass[reqno,11pt,letterpaper]{amsart}
\UseRawInputEncoding
\usepackage{mathrsfs}
\usepackage{amssymb}
\usepackage[usenames,dvipsnames]{xcolor}
\usepackage{amsthm}
\usepackage{amsmath}
\usepackage{amsfonts}
\usepackage{enumerate}
\usepackage{txfonts}
\usepackage{bm}
\usepackage{graphicx}
\numberwithin{equation}{section}

\newtheorem{theorem}{Theorem}[section]
\newtheorem{lemma}[theorem]{Lemma}
\newtheorem{corollary}[theorem]{Corollary}
\newtheorem{proposition}[theorem]{Proposition}

\theoremstyle{definition}
\newtheorem{definition}[theorem]{Definition}
\newtheorem{assumption}[theorem]{Assumption}
\newtheorem{example}[theorem]{Example}

\theoremstyle{remark}
\newtheorem{remark}[theorem]{Remark}

\begin{document}

\def\be{\begin{eqnarray}}
\def\ee{\end{eqnarray}}
\def\p{\partial}
\def\no{\nonumber}
\def\eps{\epsilon}
\def\de{\delta}
\def\De{\Delta}
\def\om{\omega}
\def\Om{\Omega}
\def\f{\frac}
\def\th{\theta}
\def\la{\lambda}
\def\lab{\label}
\def\b{\bigg}
\def\var{\varphi}
\def\na{\nabla}
\def\ka{\kappa}
\def\al{\alpha}
\def\La{\Lambda}
\def\ga{\gamma}
\def\Ga{\Gamma}
\def\ti{\tilde}
\def\wti{\widetilde}
\def\wh{\widehat}
\def\ol{\overline}
\def\ul{\underline}
\def\Th{\Theta}
\def\si{\sigma}
\def\Si{\Sigma}
\def\oo{\infty}
\def\q{\quad}
\def\z{\zeta}
\def\co{\coloneqq}
\def\eqq{\eqqcolon}
\def\bt{\begin{theorem}}
\def\et{\end{theorem}}
\def\bc{\begin{corollary}}
\def\ec{\end{corollary}}
\def\bl{\begin{lemma}}
\def\el{\end{lemma}}
\def\bp{\begin{proposition}}
\def\ep{\end{proposition}}
\def\br{\begin{remark}}
\def\er{\end{remark}}
\def\bd{\begin{definition}}
\def\ed{\end{definition}}
\def\bpf{\begin{proof}}
\def\epf{\end{proof}}
\def\bex{\begin{example}}
\def\eex{\end{example}}
\def\bq{\begin{question}}
\def\eq{\end{question}}
\def\bas{\begin{assumption}}
\def\eas{\end{assumption}}
\def\ber{\begin{exercise}}
\def\eer{\end{exercise}}
\def\mb{\mathbb}
\def\mbR{\mb{R}}
\def\mbZ{\mb{Z}}
\def\mc{\mathcal}
\def\mcS{\mc{S}}
\def\ms{\mathscr}
\def\lan{\langle}
\def\ran{\rangle}
\def\lb{\llbracket}
\def\rb{\rrbracket}

\title[Three dimensional smooth transonic flows]{Some three dimensional smooth transonic flows for the steady Euler equations with an external force}

\author{Shangkun WENG}
\address{School of Mathematics and Statistics and Hubei Key Laboratory of Computational Science, Wuhan University, Wuhan, Hubei Province, China, 430072.}
\email{skweng@whu.edu.cn}

\author{Zhouping XIN}
\address{The Institute of Mathematical Sciences and Department of Mathematics, The Chinese University of Hong Kong, Shatin, NT, Hong Kong, China.}
\email{zpxin@ims.cuhk.edu.hk}
\thanks{Updated on \today}
\maketitle

\begin{abstract}
  We establish the existence and uniqueness of some smooth accelerating transonic flows governed by the three dimensional steady compressible Euler equations with an external force in cylinders with arbitrary cross sections, which include both irrotational flows and Beltrami flows with nonuniform proportionality factors. One of the key ingredients in the analysis of smooth transonic irrotational flows is the well-posedness theory of classical solutions in $H^4$ to a linear elliptic-hyperbolic mixed second order differential equation of Keldysh type in cylinders with mixed boundary conditions. This is achieved by extending the problem to an auxiliary linear elliptic-hyberbolic-elliptic mixed problem in a longer cylinder where the governing equation becomes elliptic at the exit of the new cylinder, so that one can use the multiplier method and the cut-off techniques to derive the $H^2$ and higher order estimates in transonic regions. It is further shown that the energy estimate can be closed in the $H^4$ framework. For smooth transonic Beltrami flows, we solve a transport equation for the proportionality factor and a type-changing enlarged deformation-curl system with mixed boundary conditions. The compatibility conditions for the $H^4$ estimate to the enlarged deformation-curl system near the intersection between the entrance and the cylinder wall play a crucial role in the analysis.
\end{abstract}

\section{Introduction and main results}\label{1df}

We investigate smooth accelerating transonic irrotational and Beltrami flows to the steady isentropic compressible Euler equations with an external force:
\begin{eqnarray}\label{3deuler-f}
\begin{cases}
\text{div }(\rho {\bf u}) =0,\\
\text{div }(\rho {\bf u} \otimes {\bf u} + P  I_3) = \rho \na \Phi,
\end{cases}
\end{eqnarray}
where ${\bf u} = (u_1, u_2, u_3)$, $\rho $ represent the velocity and density respectively, $P = \rho^{\gamma}$ with $\gamma >1$ is the pressure, $\Phi $ is a potential for an external force. Denote the Bernoulli's quantity $B = \f 12 |{\bf u}|^2 + h (\rho) - \Phi $ with the enthalpy $h (\rho) =\frac{\gamma}{\gamma- 1 } \rho^{\gamma -1} $.

Since 1940s, there have been many literatures on continuous subsonic-sonic and smooth transonic steady flows. Gilbarg and Shiffman \cite{gs54} showed that for smooth  irrotational subsonic-sonic flows past a two dimensional profile, the sonic points must locate on the profile. Bers \cite{bers58} studied the possible continuation of a two dimensional irrotational subsonic-sonic flow across a sonic curve when such a flow was assumed to exist and found that sonic points can be classified into two classes: exceptional and nonexceptional. A sonic point in a $C^2$ transonic flow is exceptional if and only if the velocity is orthogonal to the sonic curve at this point, otherwise it is called nonexceptional. Bers further proved that the subsonic-sonic flow can be continued locally as a supersonic flow without discontinuities in a unique way across the sonic curve provided that there are no exceptional points on a sonic curve. The structure of the sonic curves and properties of the set of exceptional points for smooth transonic flows in 2D general nozzles have been studied further by Wang-Xin in \cite{wx16,wx21}. Friedrichs \cite{fri58} developed a general and powerful theory for positive symmetric systems of first order partial differential equations. Morawetz \cite{mo56} proved that inviscid steady irrotational transonic gas flows past a two-dimensional symmetric profile cannot remain smooth but must contain shock waves if the profile is perturbed in the supersonic region while fixing the speed at infinity. A well-posedness theory for the $H^1$ weak solutions to some linear partial differential equations of mixed type in multidimensional Euclidean spaces had been developed in the late 1970s and in the 1980s (see \cite{eg87,ka77,vr77}). Kuzmin \cite{ku02} further simplified some techniques therein and applied the theory to a perturbation problem of the von Karman equation in a rectangle.

In 2004, Morawetz \cite{mora04} surveyed the existence and uniqueness results for mixed equations and their applications to the problems of transonic flows. The authors in \cite{lmp07,lmp15} proved the existence, uniqueness and the variational characterization of $H^1$ weak solutions to the Dirichlet boundary value problem for mixed equations of Tricomi type, which had been generalized in \cite{o12} for various mixed equations of Keldysh type. Note also that by the compensated compactness, the existence of bounded subsonic-sonic flows past a profile and in nozzles were established in \cite{cdsw07,chw16,hww11,xx07}, yet there is no information available about the locations of the sonic points and the regularity nearby. There are also many studies on transonic flows involving shocks and contact discontinuities \cite{bf11,bp19,cf03,cxz22,fx21,hkwx21,lxy13,lxy16,wx23a}.

Wang and Xin \cite{wx13} proved the structural stability of a radially symmetric continuous irrotational subsonic-sonic flow in a convergent nozzle with straight solid walls under suitable perturbations of the incoming flow at the entrance. For the continuous subsonic-sonic flows constructed in \cite{wx13}, the acceleration blows up at the sonic curve and the sonic curve was treated as a free boundary. Wang and Xin \cite{wx16,wx19,wx21} established the existence and uniqueness of smooth irrotational transonic flows of Meyer type for the potential equation on two dimensional De Laval nozzles, whose sonic points are all exceptional and characteristically degenerate, and must locate at the throat. There are some further developments in \cite{wa19,wa22}. Another type of transonic flows are smooth transonic spiral flows in concentric cylinders (see \cite[Section 104]{cf48} and \cite{wxy21a,wxy21b}), whose sonic points are all nonexceptional and noncharacteristically degenerate. Different from \cite{wx19,wx21}, the degeneracy type of the linear second order elliptic-hyperbolic mixed equation obtained by linearizing around radially symmetric transonic spiral flows is similar to that of the Tricomi equation. Recently, Weng and Xin \cite{wx23b} proposed a quasi two dimensional steady Euler flow model in De Laval nozzles which generalizes the classical quasi one dimensional model. They first proved the existence and uniqueness of smooth transonic flows of Meyer type to the quasi one-dimensional model. The relation between the degeneracy rate of the velocity near the sonic point and the degree of flatness of the nozzle wall at the throat are classified precisely. The existence and uniqueness of smooth transonic flows with nonzero vorticity and positive acceleration to the quasi two dimensional model was established in \cite{wx23b} based on the analysis of a linear second order elliptic-hyperbolic mixed equation of Keldysh type. The authors in \cite{bdx} had proved the existence and uniqueness of a class of smooth transonic solutions to the 2-D steady Euler-Poisson equations. However, it seems difficult to generalize the previous studies to the truly 3-dimensional setting, which is the main focus of this paper.


Motivated by \cite{wx23b}, the authors in \cite{wz24a} proved the existence and uniqueness of smooth accelerating transonic flows to one dimensional steady Euler equations with a given external force $\bar{f} (x_1)$
\begin{eqnarray}\label{1df}\begin{cases}
(\bar{\rho} \bar{u})'(x_1)=0, \ \ \forall x_1\in [L_0,L_1],\\
\bar{\rho} \bar{u} \bar{u}'+ \frac{d}{dx_1} P(\bar{\rho})= \bar{\rho} \bar{f}(x_1), \ \ \forall x_1\in [L_0,L_1],\\
\bar{\rho}(L_0)=\rho_0>0,\ \ \bar{u}(L_0)= u_0>0,
\end{cases}\end{eqnarray}
where $\bar{f}(x_1)$ is a given infinitely differentiable function on $[L_0,L_1]$ $(L_0 < 0 < L_1)$ satisfying
\be\label{1df04}\begin{cases}
\bar{f}(x_1)<0,\ \forall x\in [L_0,0),\\
\bar{f}(0)=0,\\
\bar{f}(x_1)>0, \ \ \forall x_1\in (0,L_1]
\end{cases}\ee
and the initial state at $x_1=L_0$ is subsonic, i.e. $u_0^2<c^2(\rho_0)=\gamma \rho_0^{\gamma-1}$. Assume that
\be\label{1df90}\begin{cases}
\int_{L_0}^{0} \bar{f}(\tau) d\tau= \frac{\gamma+1}{2(\gamma-1)} \gamma^{\frac{2}{\gamma+1}} (\rho_0 u_0)^{\frac{2(\gamma-1)}{\gamma+1}}- B_0,\ \ \ B_0 \co \frac{1}{2} u_0^2 + \frac{\gamma}{\gamma-1} \bar{\rho}_0^{\gamma-1},\\
\bar{f}'(0)>0.
\end{cases}\ee

The following Lemma was proved in \cite{wz24a} .
\begin{lemma}\label{1dfpo}({\bf 1-d transonic flows with positive acceleration.}) Assume that $(u_0, \rho_0)$ is subsonic and the external force satisfies \eqref{1df04} and \eqref{1df90}. The problem \eqref{1df} has a unique solution $(\bar{\rho}(x_1), \bar{u}(x_1))\in
C^{\infty}([L_0,L_1])$ which is subsonic in $[L_0,0)$, supersonic in $(0, L_1]$ with a sonic state at $x_1=0$.
\end{lemma}

The smooth transonic flow with positive acceleration given in Lemma \ref{1dfpo} is called the background transonic flow in this paper. The authors in \cite{wz24a} showed the structural stability of the background transonic flows under irrotational axisymmetric perturbations of suitable boundary conditions. In this paper, we will establish the structural stability under more general three dimensional perturbations of suitable boundary conditions and external forces in a cylinder $\Omega \co (L_0, L_1) \times \mathbb{E}$, where $\mb{E}$ is a bounded simply-connected domain in $\mbR^2$ with $C^3$ smooth boundary.

By the vector identity ${\bf u}\cdot\nabla {\bf u}= \nabla \frac12 |{\bf u}|^2- {\bf u}\times \text{curl }{\bf u}$, the momentum equations in \eqref{3deuler-f} imply
\be\label{momentum}
\nabla (\frac{1}2|{\bf u}|^2 + h(\rho)-\Phi)={\bf u}\times \text{curl }{\bf u}\ \ \ \text{in }\ \Om.
\ee
where $\Phi(x_1,x')=\bar{\Phi}(x_1) + \epsilon \Phi_0(x_1,x')$ is a small perturbation of the potential of the background external force, i.e. $\bar{\Phi}(x_1)=\int_{L_0}^{x_1} \bar{f}(s) ds$ and $\Phi_0\in H^4(\Omega)$ satisfies the condition
\be\label{Phi1}
(n_2\p_2+n_3\p_3) \Phi_0(x_1,x')=0\ \ \ \text{on }\Ga_w:= [L_0,L_1]\times \p \mathbb{E}.
\ee

We start with isentropic irrotational flows so that $\frac{1}2|{\bf u}|^2 + h(\rho)-\Phi\equiv B_0$ and
\be\label{irrotation}
\text{curl }{\bf u}\equiv 0 \ \ \ \text{in }\Omega.
\ee
Thus there exists a potential function $\varphi$ such that ${\bf u}=\nabla \varphi$, and
\be\label{den1}
\rho=H(|\nabla\var|^2)=\bigg(\frac{\gamma-1}{\ga}\bigg)^{\frac{1}{\gamma-1}}\bigg(B_0+\Phi-\frac12 |\nabla \varphi|^2\bigg)^{\frac{1}{\gamma-1}}.
\ee
Substituting \eqref{den1} into the continuity equation yields
\be\label{po1}
c^2(\rho)\Delta \varphi- \sum_{i,j=1}^3 \p_i \varphi\p_j \varphi \p_{ij}^2 \varphi +\sum_{i=1}^3 \p_i\varphi \p_i \Phi=0, 
\ee
where $c^2(\rho)=(\gamma-1)(B_0+\Phi-\frac{1}{2}|\nabla \varphi|^2)$.

A potential function for the background transonic flow in Lemma \ref{1dfpo} can be taken as $\bar{\varphi} = \bar{\varphi } (x_1) = \int_{L_0}^{x_1} \bar{u}(\tau) d \tau$, which satisfies
\be\no
(c^2(\bar{\rho})- |\bar{\varphi}'|^2)\bar{\varphi}'' + \bar{f} \bar{\varphi}' =0
\ee
with $c^2(\bar{\rho})=(\gamma-1)(B_0+\bar{\Phi}-\frac{1}{2}\bar{u}^2)$.

The boundary conditions for \eqref{po1} are prescribed as follows
\be\label{fbcs}\begin{cases}
\p_j \varphi(L_0,x')= \epsilon \p_j h_0(x'), \ \ &\forall x'\in \mb{E}, j=2,3\\
(n_2 \p_2+n_3 \p_3)\varphi(x_1, x')=0, \ \ &\forall (x_1,x')\in \Ga_w:=[L_0,L_1]\times \p \mathbb{E},
\end{cases}\ee
where $h_0\in H^{4}(\mb{E})$ satisfies the following compatibility conditions
\be\label{fcp1}
(\p_2 h_0, \p_3 h_0)(x')=(n_2\p_2 +n_3\p_3)(\p_2 h_0,\p_3 h_0)(x')=0,\ \ \ \ \ \text{on }\p\mb{E}
\ee
and $h_0(\bar{x}')=0$ for a fixed point $\bar{x}'\in \mathbb{E}$.

Note that the first conditions in \eqref{fbcs} indicate some restrictions on the flow angles at the entrance and the second one in \eqref{fbcs} is the slip boundary condition on the wall, both of them are physically acceptable and experimentally controllable. Mathematically, it will be shown that these boundary conditions are also admissible for the linearized mixed type equation (see Lemma \ref{H1e}) and enable us to derive the $H^1$ and more importantly, the higher-order energy estimates. Note that no boundary conditions are imposed at the exit of the cylinder where the flows are expected to be hyperbolic.

Then the first main result in this paper states the existence and uniqueness of smooth transonic irrotational flows to \eqref{3deuler-f} which are close to the background transonic flow.
\begin{theorem}\label{3dirro}
{\it Let $(\bar{\rho},\bar{u})$ be a background transonic flow given in Lemma \ref{1dfpo}. Assume that $\gamma>1$, $\Phi_0\in H^4(\Om)$ and $h_0\in H^{4}(\mathbb{E})$ satisfy \eqref{Phi1} and \eqref{fcp1}, respectively. Then there exist positive
constants $C$ and $\epsilon_0$ depending on the background flow, the external force and the boundary datum, such that for any $0<\epsilon<\epsilon_0$, the problem \eqref{3deuler-f},\eqref{irrotation} with \eqref{fbcs} has a unique smooth transonic irrotational solution ${\bf u}=\nabla \varphi$, where $\varphi\in H^4(\Omega)$ satisfies
\be\label{3d1}
\|\varphi-\bar{\varphi}\|_{H^4(\Omega)}\leq C\epsilon.
\ee

Moreover, all the sonic points form a $C^1$ smooth surface given by $x_1=\xi(x')\in C^{1}(\ol{\mathbb{E}})$, which is close to the background sonic front $x_1=0$ in the sense that
\be\label{3dsonic}
\|\xi(x')\|_{C^1(\ol{\mb{E}})}\leq C\epsilon.
\ee
}\end{theorem}

Next we turn to the transonic flows with non-trivial vorticities, and consider general Beltrami flows such that $\frac{1}2|{\bf u}|^2 + h(\rho)-\Phi\equiv B_0$. Then it follows from \eqref{momentum} that
\be\no
{\bf u}\times \text{curl }{\bf u}\equiv 0 \ \ \ \text{in }\Om.
\ee
Thus the vorticity $\text{curl }{\bf u}$ is always parallel to the velocity and can be represented as
\be\no
\text{curl }{\bf u} (x)= \ka(x) \rho(x){\bf u}(x)
\ee
for some scalar function $\ka(x)$. It follows from the continuity equation in \eqref{3deuler-f} that
\be\no
{\bf u}\cdot\nabla \ka =0\ \ \ \text{in }\Omega.
\ee
By the definition of the Bernoulli's quantity, there holds
\be\label{bel4}
\rho=H(|{\bf u}|^2,\Phi)=\bigg(\frac{\gamma-1}{\ga}\bigg)^{\frac{1}{\gamma-1}}\bigg(B_0+\Phi-\frac12 |{\bf u}|^2\bigg)^{\frac{1}{\gamma-1}}.
\ee

Therefore for any Beltrami flows, the steady Euler system \eqref{3deuler-f} is reduced to
\be\label{belt}\begin{cases}
\text{div }(H(|{\bf u}|^2,\Phi){\bf u})=0,\\
\text{curl }{\bf u} (x)= \ka(x) H(|{\bf u}|^2,\Phi){\bf u}(x),\\
{\bf u}\cdot\nabla \ka =0\ \ \ \text{in }\Omega.
\end{cases}\ee

As above, the boundary conditions are prescribed as
\be\label{bbc}\begin{cases}
u_i(L_0,x')= \epsilon h_i(x'), \ \ &\forall x'\in \mb{E},\, \ i=2,3,\\
(n_2 u_2+n_3 u_3)(x_1,x')=0, \ \ &\forall (x_1,x')\in \Ga_w:=[L_0,L_1]\times \p \mathbb{E},
\end{cases}\ee
where $h_i\in H^{\frac{7}{2}}(\mb{E}), i=2,3$ satisfy $\p_2 h_3 -\p_3 h_2\in H^3(\mb{E})$ and the following compatibility conditions
\be\label{bcp1}\begin{cases}
h_i(x')=(n_2\p_2+n_3\p_3) h_i(x')=0,\ \ \ \ \forall x'\in \p \mb{E}, i=2,3,\\
(\p_2 h_3-\p_3 h_2)(x')=0,\ \forall x'\in \p \mb{E},\\
(n_2\p_2+n_3\p_3)\p_2 h_2(x')=(n_2\p_2+n_3\p_3)\p_3 h_3(x')=0,\ \forall x'\in \p \mb{E}.
\end{cases}\ee

Then the second main result in this paper is the existence and uniqueness of a smooth transonic Beltrami flow to \eqref{belt}-\eqref{bcp1}.
\begin{theorem}\label{beltrami}
{\it Let $(\bar{\rho},\bar{u})$ be a background transonic flow given in Lemma \ref{1dfpo}. Assume that $\gamma>1$, $\Phi_0\in H^4(\Omega)$ and $h_i\in H^{\frac72}(\mathbb{E}),i=2,3$ satisfy the compatibility conditions \eqref{Phi1} and \eqref{bcp1}, respectively. Then there exist positive
constants $C, \epsilon_0$ depending on the background flow, the potential force and the boundary datum, such that for any $0<\epsilon<\epsilon_0$, the problem \eqref{belt} with \eqref{bbc} has a unique smooth transonic Beltrami flow $(u_1,u_2,u_3)\in H^3(\Omega)$ satisfying the estimate
\be\no
\|u_1-\bar{u}\|_{H^3(\Omega)}+\|u_2\|_{H^3(\Omega)}+\|u_3\|_{H^3(\Omega)}\leq C\epsilon.
\ee
Moreover, all the sonic points form a $C^1$ smooth surface $x_1=\xi(x')\in C^{1}(\ol{\mathbb{E}})$, which is close to the background sonic front $x_1=0$:
\be\no
\|\xi(x')\|_{C^1(\ol{\mb{E}})}\leq C\epsilon.
\ee
}\end{theorem}

\begin{remark}
{\it To the best of our knowledge, Theorem \ref{3dirro} and \ref{beltrami} are the first results on three dimensional smooth transonic flows for the compressible Euler equations in cylinders with general cross sections. The condition $(\p_2 h_3-\p_3 h_2)(x')=0$ on $\p\mb{E}$ implies that the vorticity of the transonic Beltrami flows on the cylinder wall must be zero. This is required for the $H^4$ regularity of the deformation-curl system at the intersection of the entrance and the cylinder wall. In general, $(\p_2 h_3-\p_3 h_2)(x')$ is not nonzero in $\mb{E}$, thus the transonic Beltrami flows constructed in Theorem \ref{beltrami} have nontrivial vorticity inside the cylinder.
}\end{remark}


\begin{remark}
{\it There are many literatures on the incompressible Beltrami flows \cite{cc15,cw16,cm88,ds13,ep23,na14}. Especially, the incompressible Beltrami flows are used in \cite{ds13} as building blocks to construct continuous solutions to the incompressible Euler equations which dissipate the total kinetic energy. The authors in \cite{myz95} constructed several examples with different symmetries for compressible Beltrami flows. Theorem \ref{beltrami} seems to be the first general existence and uniqueness result for the compressible Beltrami flows.
}\end{remark}

\begin{remark}\label{three}
{\it By the deformation-curl decomposition developed in \cite{wex19,we19}, the system \eqref{3deuler-f} is equivalent to the system
\be\no\begin{cases}
\sum_{i,j=1}^3( c^2(H)\delta_{ij}- u_i u_j) \p_i u_j + \sum_{i=1}^3 u_i \p_i \Phi=0,\\
\p_2 u_3-\p_3 u_2= \omega_1,\\
\p_3 u_1-\p_1 u_3=\omega_2=\frac{u_2\omega_1 + \p_3 B}{u_1},\\
\p_1 u_2-\p_2 u_1=\omega_3(x_1,x')=\frac{u_3\omega_1 - \p_2 B}{u_1},\\
(u_1\p_1+u_2\p_2+u_3\p_3)B=0
\end{cases}\ee
where $\omega_1$ is solved by the transport equation
\be\no
\p_1 \om_1 + \sum_{i=2}^3 \frac{u_i}{u_1}\p_i \om_1 + \om_1 \sum_{i=2}^3 \p_i\bigg(\frac{u_i}{u_1}\bigg)+ \p_2\bigg(\frac{1}{u_1}\bigg)\p_3 B-\p_3\bigg(\frac{1}{u_1}\bigg)\p_2 B=0.
\ee
To look for a solution $({\bf u}, B)$ in $H^3(\Omega)$, as in the two dimensional case \cite{wxy21b,wx23b}, the source terms in the curl system belong to $H^2(\Omega)$, one thus can derive only the $H^2(\Om)$ norm estimate for the velocity due to the type changing of the deformation-curl system in the transonic region. Thus there is a possible loss of derivatives. This is indeed an essential difficulty, which was overcome in the two dimensional case by utilizing the one order higher regularity of the stream function and the fact that the Bernoulli's quantity can be represented as a function of the stream function. However, in the three dimensional case, such a stream function formulation is unavailable in general. It maybe plausible to develop a tame estimate for the linearized mixed type equation and an appropriate version of the Nash-Moser iteration to overcome this difficult issue. This will be investigated in future.

}\end{remark}

As we will show later, the existence and uniqueness of smooth transonic irrotational flows rely on the well-posedness theory of strong solutions to a boundary value problem for a linear second order elliptic-hyperbolic mixed differential equation of Keldysh type with mixed boundary conditions in a cylinder, which takes the form
\be\label{3li}\begin{cases}
k_{11}\p_1^2\psi + (k_{13}+k_{31})\p_{13}^2\psi + \sum_{i,j=2}^3 k_{ij}\p_{ij}^2\psi + \bar{k}_1(x_1)\p_1\psi= F(x_1,x'),\ \ \text{in }\Om,\\
\psi(L_0,x')=0, \ \ \ \text{for all}\ x'\in \mb{E},\\
(n_2\p_2+n_3\p_3)\psi(x_1,x')=0,\ \ \ \text{on }\ \ \Ga_w.
\end{cases}\ee
Here $k_{11}(x_1,x')$ changes its signs when the fluid past the sonic front and the coefficient matrix $(k_{ij})_{i,j=2}^3$ is positive definite uniformly in $\Om$. Compared with the two dimensional and the three dimensional axisymmetric cases in \cite{wx23b,wz24a}, there are some crucial differences in estimating the $L^2$ norm of $\p_1^2\psi$ in the transonic region due to the appearance of an additional term $\sum_{i,j=2}^3 \p_1 k_{ij}\p_{ij}^2\psi$. It seems quite difficult to estimate $\p_1^2\psi$ and $\nabla_{x'}^2\psi$ to \eqref{3li} directly as in \cite{wx23b,wz24a}. Our main strategy here is to extend the problem \eqref{3li} along the positive $x_1$ direction to an auxiliary linear second order elliptic-hyperbolic-elliptic mixed problem in a longer cylinder where the governing equation becomes elliptic at the exit of the new cylinder. For the latter case, it is plausible to utilize the multiplier method and a cut-off technique to deal with the trouble term $\sum_{i,j=2}^3 \p_1 k_{ij}\p_{ij}^2\psi$ at the expense of one order more regularity requirement on the coefficients in comparison with the 2-D case (see more explanations in \eqref{W105}). Such an extension technique had been outlined without a detailed construction in \cite{ku02} for the stability analysis of the von Karman equation in a two dimensional duct. The $H^3$ and $H^4$ estimates are also obtained with similar costs on the coefficients. The separation of variables and the Banach contraction principle are then used to prove the $H^4$ estimate near the intersection between the entrance and the wall of the cylinder. We then show that the energy estimates can be closed in the $H^4$ framework in the 3-D case.

For a smooth transonic Beltrami flow, the vorticity field is parallel to the velocity field whose proportionality factor $\kappa$ is conserved along stream lines. Then one has to solve a problem consisting of a transport equation for $\kappa$ and an enlarged deformation-curl system, where we introduce a new unknown function to overcome the difficulty that the source term in the curl system is not divergence free in general after the linearization. Though there is no loss of derivatives in the form for the linear stability analysis for the Beltrami flows, compared with general three dimensional transonic rotational flows discussed in Remark \ref{three}, however, since the boundary data for $\kappa$ at the entrance should be obtained by the tangential velocity in \eqref{bbc}, there is a loss of $\frac12$ derivatives on the boundary conditions due to the trace theorem in Sobolev spaces (see the equation \eqref{ka21}). We overcome this difficulty by using the higher regularity of the velocity field in the subsonic region. It should be noted that the compatibility conditions and the higher order regularity of the solutions to the Poisson equation and the deformation-curl system with mixed boundary conditions in cylinders play crucial roles in our analysis here, which need to be checked carefully.

The rest of this paper is organized as follows. We linearize the transonic irrotational flow problem \eqref{po1} with \eqref{fbcs} in \S\ref{formulation} and prove the existence and uniqueness of the $H^2$ strong solution to the linearized problem in \S\ref{strongH2}. The $H^4$ estimate and the proof of Theorem \ref{3dirro} for the transonic irrotational flows are given in \S\ref{high}. In \S\ref{belflow}, we reformulate the transonic Beltrami flow problem and give the proof of Theorem \ref{beltrami}. Finally, the $H^4$ regularity of the solution to the divergence-curl system with homogeneous normal boundary conditions is shown in the Appendix \S\ref{div-curl}.

\section{The reformulation of the transonic irrotational flow problem}\label{formulation}

Let $\bar{\varphi}$ be a potential of the background flow and $\varphi$ be a solution to \eqref{po1} subject to the boundary condition \eqref{fbcs}. Set $\psi_1 = \varphi - \bar{\varphi}$. Then
\be\no\begin{cases}
\sum_{i,j=1}^3 k_{ij} (\na \psi_1 ) \p_{ij}^2 \psi_1 + \bar{k}_{1}(x_1)\p_{1}\psi_1 = F_0(\na \psi_1), \\
  \p_j\psi_1 (L_0, x') = \eps \p_j h_0 (x'), \q \forall x' \in \mathbb{E}, \q j=2,3,\q \psi_1(L_0, 0) = 0, \\
  (n_2\p_2+n_3\p_3)\psi_1 (x_1,x')=0, \q \forall (x_1,x') \in \Ga_w,
\end{cases}
\ee
where
\be\label{coe}
\begin{cases}
  k_{11} (\na \psi_1 ) = 1-\bar{M}^2+\frac1 {c^2 (\bar{\rho})} \bigg((\gamma-1)\epsilon \Phi_0- (\gamma+1)\bar{u}\p_{1}\psi_1\\
  \q\q\q-(\p_1\psi_1)^2-\frac{\gamma-1}{2}|\nabla\psi_1|^2\bigg), \ \ \bar{M}^2(x_1)= \frac {c^2(\bar{\rho }) - \bar{u}^2}{c^2 (\bar{\rho })}, \\
  k_{ii} (\na \psi_1 ) = 1+ \f{\gamma-1}{c^2 (\bar{\rho})} \bigg(\epsilon \Phi_0- \bar{u}\p_1\psi_1-\frac{1}{2}|\nabla \psi_1|^2\bigg) - \frac{(\p_i \psi_1 )^2}{c^2(\bar{\rho})},\ i=2,3,\\
  k_{1i} (\na \psi_1 ) = k_{i1}(\nabla\psi_1)=-\frac{(\bar{u}+\p_1\psi_1) \p_i \psi_1}{c^2(\bar{\rho})},\ \ \ i=2,3,,\\
  k_{23} (\na \psi_1 ) = k_{32}(\nabla \psi_1)=-\frac {\p_2 \psi_1 \p_3\psi_1}{c^2 (\bar{\rho})},\ \ \bar{k}_1(x_1)= \frac{\bar{f} - (\ga +1) \bar{u} \bar{u}'}{c^2 (\bar{\rho})}= \f {\bar{f} (c^2 (\bar{\rho }) + \ga \bar{u}^2)}{c^2 (\bar{\rho }) (c^2 (\bar{\rho }) - \bar{u}^2)},  \\
  F_0(\na \psi_1)=- \frac{\bar{u}'}{c^2 (\bar{\rho})}\bigg(
(\gamma -1) \epsilon\Phi_0- \frac{\gamma +1}2 (\p_{1} \psi_1)^2 - \frac {\gamma-1}{2} |\nabla'\psi_1|^2\bigg)\\
\q\q\q\q\q\q-\frac{\epsilon}{c^2 (\bar{\rho})}\bigg(\bar{u}\p_1\Phi_0+\sum_{j=1}^3 \p_j\Phi_0\p_j \psi_1\bigg),\ \ \ \nabla'=(\p_2,\p_3),\\
c^2(\rho) =c^2(\bar{\rho})+ (\gamma -1) (\epsilon \Phi_0- \bar{u}\p_1\psi_1 -\frac12 |\nabla \psi_1|^2).
\end{cases}
\ee

Choose a monotonic decreasing function $\eta_0 (x_1) \in C^{\oo} ([L_0, L_1])$ satisfying
\be\label{eta}
\eta_0 (x_1) = \begin{cases}
    1, & L_0 \le x_1 \le \f {19}{20} L_0 \\
    0, & \f9{10} L_0 \le x_1 \le L_1.
    \end{cases}
\ee
Set $\psi (x_1,x') = \psi_1 (x_1 ,x') - \epsilon\psi_0(x_1,x')$, where $ \psi_0(x_1,x')=\eta_0 (x_1)h_0 (x')$. Then $\psi$ satisfies
\be\label{li2}
\begin{cases}
  \sum_{i,j=1}^3 k_{ij} (\nabla \psi+\epsilon\nabla \psi_0) \p_{ij}^2\psi + \bar{k}_1(x_1) \p_1\psi= F(\na \psi), & (x_1, x') \in \Omega, \\
  \psi(L_0, x') =0, & \forall x' \in \mathbb{E}, \\
  (n_2\p_2 +n_3\p_3)\psi(x_1,x')= 0, & \forall (x_1,x') \in \Ga_w,
\end{cases}
\ee
where
\be\label{f0}
F(\nabla\psi)= F_0(\nabla\psi+\epsilon\nabla \psi_0)-\displaystyle \epsilon\sum_{i,j=1}^3 k_{ij}(\nabla\psi+\epsilon\nabla \psi_0)\p_{ij}^2 \psi_0- \epsilon\bar{k}_1(x_1) \p_1 \psi_0.
\ee

Set the function space $\Xi_{\de_0}$ to be the set of the functions $\psi \in H^4(\Omega)$ satisfying the estimate $\|\psi\|_{H^4(\Omega)} \leq \delta_0$ with $\de_0 > 0$ to be specified later and the compatibility conditions
\be\label{cp00}\begin{cases}
\psi (L_0, x') = 0, \ \ \forall x' \in \mb{E}, \\
(n_2\p_2 + n_3 \p_3) \psi(x_1,x')=0,\ \forall (x_1,x')\in\Ga_w.
\end{cases}\ee

For any given $\hat{\psi} \in \Xi_{\de_0}$, one can define an operator $\mc{T}$, mapping $\hat{\psi} \in \Xi_{\de_0}$ to $ \psi \in \Xi_{\de_0} $, where $\psi$ is obtained by solving the following boundary value problem for a linearized mixed type equations
\be\label{li0}
\begin{cases}
\sum_{i,j=1}^3 k_{ij} (\nabla \hat{\psi}+\epsilon\nabla \psi_0) \p_{ij}^2\psi + \bar{k}_1(x_1) \p_1\psi= F(\na \hat{\psi }), & (x_1, x') \in \Omega,\\
\psi(L_0, x') =0, & \forall x' \in \mathbb{E}, \\
(n_2\p_2 +n_3\p_3)\psi(x_1,x')= 0, & \forall (x_1,x') \in \Ga_w.
\end{cases}
\ee
It will be shown in the next section that $\mc{T}$ is well-defined. As a preparation, we list some properties of the background solution and the coefficients $k_{ij}$ to be used later. Define $\bar{k}_{11}(x_1)=1-\bar{M}^2(x_1)$ and $\bar{k}_1(x_1)$ as in \eqref{coe}.

\begin{lemma}[Lemma 1.3 in \cite{wz24a}]\label{bkg-coe}
  For the background transonic flows in Lemma \ref{1dfpo}, there exists a constant $\kappa_* > 0$, such that
  \be\label{6}
  2 \bar{k}_1 (x_1) + (2j -1) \bar{k}_{11}^{\prime} (x_1) \le - \kappa_*, \, j = 0, 1, 2,3,\,\, \forall x_1 \in [L_0, L_1].
  \ee
  Thus there also exists another positive number $d_0 >0$, such that $d(x_1) = 6 (x_1 - d_0) < 0$ and
  \be\label{7}
  (\bar{k}_1 + j \bar{k}_{11}^{\prime}) d - \f 12 (\bar{k}_{11}d)^{\prime} \ge 4, \, j = 0, 1, 2, 3,\,\, \forall x_1 \in [L_0, L_1].
  \ee
\end{lemma}

For the coefficients $k_{ij}(\nabla \hat{\psi}+\epsilon\nabla \psi_0)$ defined in \eqref{coe}, the following properties hold.
\begin{lemma}\label{coe1}
For any $\psi\in \Xi_{\delta_0}$, $k_{ij}(\nabla \psi+\epsilon\nabla \psi_0),i,j=1,2,3$ satisfy the properties
\be\label{coe100}
&&\sum_{j=1}^3\|k_{1j}(\nabla(\psi+\epsilon\psi_0))-\bar{k}_{11}\delta_{1j}\|_{H^3(\Om)}\leq C_*(\epsilon+\delta_0),\\\label{coe101}
&&\sum_{i,j=2}^3 \|k_{ij}(\nabla(\psi+\epsilon\psi_0))-\delta_{ij}\|_{H^3(\Om)}\leq C_*(\epsilon+\delta_0),\\\label{coe102}
&&\sum_{i,j=2}^3 k_{ij}(\nabla(\psi+\epsilon\psi_0))\xi_i \xi_j \geq \frac12 (\xi_2^2+\xi_3^2),\  \forall (\xi_2,\xi_3)\in \mbR^2,\\\label{coe103}
&&\displaystyle\sum_{i=2}^3 n_i k_{1i}(\nabla(\psi+\epsilon\psi_0))(x_1,x')=0,\ \ \forall (x_1,x')\in \Ga_w.
\ee
For any $\psi_{i}\in C^4(\ol{\Om})$ satisfying the boundary condition $(n_2\p_2+n_3\p_3)\psi_i=0$ on $\Ga_w$ for $i=1,2$, then it holds that
\be\label{coe104}
\displaystyle\sum_{i,j=2}^3 \p_1^m(k_{ij}(\nabla\psi_1))\p_j\psi_2 n_i=0,\ \  &\forall (x_1,x')\in \Ga_w,\ \ m=0,1,2,3.
\ee
\end{lemma}

\begin{proof}
The estimates \eqref{coe100} and \eqref{coe101} follow simply from the definition in \eqref{coe} and the fact that the Sobolev space $H^3(\Om)$ is a Bananch algebra. The positivity of the matrix $(k_{ij})_{i,j=2}^3$ is a consequence of \eqref{coe101}. The conditions \eqref{fcp1} and \eqref{cp00} imply \eqref{coe103}. It remains to prove \eqref{coe104}. For brevity, denote
\be\no
b(\nabla \psi_1)= \f{\gamma-1}{c^2 (\bar{\rho})} \bigg(\epsilon \Phi_0- \bar{u}\p_1\psi_1-\frac{1}{2}|\nabla \psi_1|^2\bigg), \ \ q(x_1)=\frac{1}{c^2(\bar{\rho}(x_1))}.
\ee
Then
\be\no
&&k_{ii}(\nabla\psi_1)=1+ b(\nabla\psi_1)-q(x_1)(\p_i \psi_1)^2,\ \  i=2,3,\\\no
&&k_{23}(\nabla\psi_1)=k_{32}(\nabla\psi_1)=-q(x_1)\p_2 \psi_1 \p_3 \psi_1.
\ee
Denote $\p_{{\bf n}}= n_2\p_2 + n_3\p_3$, then simply calculations show that
\be\no
&&\sum_{i,j=2}^3 k_{ij}(\nabla \psi_1)\p_j\psi_2 n_i=(1+b(\nabla\psi_1))\p_{{\bf n}}\psi_2 - q(x_1)\p_{{\bf n}}\psi_1 \sum_{i=2}^3 \p_i\psi_1\p_i\psi_2=0,\\\no
&&\sum_{i,j=2}^3 \p_1 k_{ij}(\nabla \psi_1)\p_j\psi_2 n_i= \p_1(b(\nabla\psi_1))\p_{{\bf n}}\psi_2 - q'(x_1)\p_{{\bf n}}\psi_1 \sum_{i=2}^3\p_i\psi_1\p_i\psi_2 \\\no
&&\q\q\q- q(x_1) \bigg(\p_{{\bf n}}(\p_{1}\psi_1)\sum_{i=2}^3\p_i\psi_1\p_i\psi_2+ \p_{{\bf n}}\psi_1 \sum_{i=2}^3\p_{1i}^2\psi_1\p_i\psi_2\bigg)=0,
\ee
which verify \eqref{coe104} for $m=0,1$. The other two cases can be proved similarly by direct computations.

\end{proof}

\section{The $H^2$ strong solution to the linearized problem \eqref{li0}}\label{strongH2} \noindent

For simplicity, we assume that $h_0\in C^5(\ol{\mb{E}})$. The function $\hat{\psi}\in \Xi_{\de_0}$ can be approximated by a sequence of smooth functions $\{\psi^m\}_{m=1}^{\infty}\subset C^5(\ol{\Omega})\cap \Xi_{\de_0}$ in the $H^4(\Omega)$ norm, so that the coefficients $k_{ij}(\nabla \psi^m+\epsilon\nabla \psi_0)\in C^4(\ol{\Omega})$ for $i,j=1,2,3$ and $m\geq 1$ satisfy the estimates and the compatibility conditions listed in Lemma \ref{coe1}. In the following sections, we consider the following boundary value problem for a linear second order mixed type differential equation
\be\label{li1}
\begin{cases}
  \mc{L} \psi \co \displaystyle\sum_{i,j=1}^3 k_{ij}(x_1,x') \p_{ij}^2\psi + \bar{k}_1(x_1) \p_1\psi= F(x_1,x'), & (x_1, x') \in \Omega, \\
  \psi(L_0, x') =0, & \forall x' \in \mathbb{E}, \\
  (n_2\p_2 +n_3\p_3)\psi(x_1,x')= 0, & \forall (x_1,x') \in \Ga_w,
\end{cases}
\ee
where the coefficients $k_{ij}, i,j=1,2,3$ are assumed to belong to $C^4(\ol{\Omega})$ and satisfy the estimates and the compatibility conditions listed in Lemma \ref{coe1}. The function $F$ is assumed to belong to $C^3(\ol{\Omega})$.

In this section, we first derive the basic $H^1(\Om)$ estimate for \eqref{li1} using the famous ``ABC" method developed by Friedrichs\cite{fri58}, thus the uniqueness of the $H^2(\Om)$ strong solution to \eqref{li1} follows immediately. However, when estimating the $L^2$ norm of $\p_1^2\psi$, there is a trouble term $\sum_{i,j=2}^3\p_1 k_{ij}\p_{ij}^2 \psi$ which is difficult to deal with. Moreover, it seems impossible to derive the $L^2(\Om)$ norm of $\nabla_{x'}^2\psi$ to \eqref{li1} directly as in the 2-D and 3-D axisymmetric cases \cite{wx23b,wz24a}. To obtain the existence of the $H^2(\Om)$ strong solution to \eqref{li1}, our key strategy here is to extend the problem \eqref{li1} along the positive $x_1$ direction to an auxiliary elliptic-hyperbolic-elliptic mixed problem in a longer cylinder where the governing equation becomes elliptic near the exit of the extended cylinder. For the latter case, it is plausible to utilize the multiplier method, the cut-off technique and the integration by parts to handle the difficult term $\sum_{i,j=2}^3 \p_1 k_{ij}\p_{ij}^2\psi$ at the expense of one order more regularity requirement on the coefficients in comparison with the 2D case (see more explanations in \eqref{W105}).

\subsection{The uniqueness of the $H^2$ strong solution to \eqref{li1}}\label{se31}

The following Lemma gives the $H^1(\Omega)$ energy estimate for any classical solution to \eqref{li1}. The properties of the background flow given in Lemma \ref{bkg-coe} play a crucial role here.

\begin{lemma}\label{H1e}
There exists positive constants $\epsilon_*$ and $\de_*$ depending only on the background flow, the potential force and the boundary data, such that if $0<\epsilon<\epsilon_*$ and $0 < \de_0 \le \de_*$ in Lemma \ref{coe1}, any classical solution to \eqref{li1} satisfies
  \be\label{H10}
\iint_{\mb{E}}|\p_1\psi(L_0,x')|^2+|\nabla \psi(L_1,x')|^2dx'+\|\psi\|_{H^1(\Om)}^2\leq C_* \|F\|_{L^2(\Om)}^2,
\ee
where $C_*$ depends only on the $H^3(\Omega)$ norms of $k_{ij}, i,j=1,2,3$ and $\bar{k}_1$.
\end{lemma}
\begin{proof}
Let $d (x_1) = 6 (x_1 - d_0) < 0$ for $x_1 \in [L_0, L_1]$ as in Lemma \ref{bkg-coe}. Multiplying the first equation in \eqref{li1} by $d (x_1) \p_{x_1} \psi $ and integration by parts in $ D $, one can get from \eqref{coe103} that
\be\no
&& \iiint_\Omega d(x_1) \p_{1} \psi F  dx' dx_1 = \iint_{\mb{E}}\bigg(\frac12 d k_{11} (\p_{1} \psi)^2 - \frac{d}{2} \sum_{i,j=2}^3 k_{ij}\p_i\psi \p_j\psi \bigg)\bigg|_{x_1 = L_0}^{L_1} dx'\\\label{4}
&&\q\q+\iiint_{\Omega} \bigg(d \bar{k}_1 - \f 12  \p_{1} (d k_{11})-d\sum_{i=2}^3\p_i k_{1i}\bigg)(\p_{1} \psi)^2 dx' dx_1\\\no
&&\q\q +\iiint_{\Om}\sum_{i,j=2}^3\bigg(\frac{1}{2} d' k_{ij}\p_i \psi\p_j\psi- d\p_1\psi\p_i k_{ij}\p_j\psi  +  \frac{d}{2}\p_1 k_{ij}\p_i\psi \p_j \psi\bigg) dx' dx_1,
\ee
where one has used \eqref{coe104} for $m=0$ so that
\be\no
&&\iiint_{\Omega} d\p_1 \psi\sum_{i,j=2}^3 k_{ij}\p_{ij}^2 \psi =-\iiint_{\Omega} \sum_{i,j=2}^3 d k_{ij}\p_{1i}^2 \psi \p_j\psi + d\p_1 \psi  \p_i k_{ij} \p_j\psi\\\no
&&=-\iiint_{\Omega}\sum_{i,j=2}^3 \frac{d}{2}\bigg(\p_1(k_{ij}\p_i\psi \p_j\psi)- \p_1 k_{ij}\p_i\psi \p_j\psi\bigg) + d\p_1\psi \p_i k_{ij}\p_j\psi dx' dx_1\\\no
&&=-\iint_{\mb{E}} \frac{d}{2} \sum_{i,j=2}^3 k_{ij}\p_i\psi\p_j\psi\b|_{x_1=L_0}^{L_1} dx'+\iiint_{\Omega} \frac12 d'\sum_{i,j=2}^3 k_{ij}\p_{i} \psi \p_j\psi dx' dx_1\\\no
&&\q\q\q + \iiint_{\Omega}\sum_{i,j=2}^3 \left(\frac{1}{2} d \p_1 k_{ij}\p_{i}\psi \p_j\psi - d\p_1 \psi \p_i k_{ij}\p_j\psi \right) dx' dx_1.
\ee

According to \eqref{6}-\eqref{7}, there exist positive constants $\epsilon_*$ and $\de_*$ such that if $0<\epsilon\leq\epsilon_*$ and $0 < \de_0 \leq\de_*$ in Lemma \ref{coe1}, we can conclude that
\be\no
&& \bar{k}_1 d - \f 12 \p_{1} (d k_{11}) - d \sum_{i=2}^3\p_i k_{1i}= \bar{k}_1 d - \f 12 (d \bar{k}_{11})' - \f 12 \p_{1} ((k_{11} - \bar{k}_{11})d) - d \sum_{i=2}^3\p_i k_{1i} \\\no
&& \geq 4 - \f 12 \| \p_{1} ((k_{11} - \bar{k}_{11})d) \|_{L^{\oo}(\Om)} - \sum_{j=2}^3\| d\p_j k_{1j}\|_{L^{\oo}(\Om)}\geq 3 , \q \forall (x_1 ,x') \in \Om,
\ee
due to the Sobolev embedding $H^3(\Omega) \subset C^{1, \f 12 }(\ol{\Omega})$ and $H^2(\Omega) \subset L^{\oo } (\Omega) $. Note also $d (L_0) < 0$ and $d(L_1) < 0$, while $k_{11} (L_0, x') > 0$ and $k_{11} (L_1, x') < 0$, it can be inferred from \eqref{4} that
\be\no
\iint_{\mb{E}}|\p_1\psi(L_0,x')|^2+|\nabla \psi(L_1,x')|^2dx'+ \|\na \psi\|_{L^2(\Omega)}^2\leq C_*\iiint_{\Omega} d (x_1) \p_{1} \psi F dx' d x_1,
\ee
due to $\psi(L_0,x') = 0$. Then \eqref{H10} follows from this and the Poincare's inequality.

\end{proof}

This implies the following corollary.
\begin{corollary}\label{uniqueness}
{\it Under the assumptions in Lemma \ref{H1e}, the $H^2(\Om)$ strong solution to \eqref{li1} is unique.
}\end{corollary}

\begin{proof}
Suppose that $\psi_m, m=1,2$ in $H^2(\Om)$ are two strong solutions to \eqref{li1} where the equation in \eqref{li1} is satisfied almost everywhere. Let $v=\psi_1-\psi_2$. Then $v$ satisfies the boundary conditions in \eqref{li1} and $\mc{L} v=0$ holds almost everywhere. Note that $v\in H^2(\Omega)$, there exists a sequence of $C^2(\ol{\Om})$ functions $\{v_m\}_{m=1}^{\infty}$ satisfying the boundary conditions in \eqref{li1} and converging to $v$ in $H^2(\Om)$. Since $\mc{L}v=0$, the $L^2(\Om)$ norms of $\mc{L} v_m$ will converge to $0$ as $m\to \infty$. It follows from the $H^1(\Om)$ estimate in \eqref{H10} that
\be\no
\|v_m\|_{H^1(\Omega)}\leq C_*\|\mc{L} v_m\|_{L^2(\Omega)}\to 0\ \ \text{as }\ \ m\to \infty.
\ee
Thus $\|v\|_{H^1(\Omega)}=0$ and $v\equiv 0$.

\end{proof}

\subsection{The existence of the $H^2(\Om)$ strong solution to \eqref{li1}}\label{se32}

To show the existence of a $H^2(\Om)$ strong solution to \eqref{li1}, we extend the problem \eqref{li1} along the positive $x_1$ direction to an auxiliary elliptic-hyperbolic-elliptic mixed problem in a longer cylinder. First, we extend the background solution to the interval $[L_0, L_2]$ where $L_2 = 2 L_1$ so that the functions $\bar{k}_1$, $\bar{k}_{11}$ defined as in \S\ref{formulation} also satisfy the properties in \eqref{6}-\eqref{7} on $[L_0, L_2]$ if $d_0$ is chosen to be large enough. This can be achieved by simply extending the function $\bar{f}$ to $[L_0, L_2]$ so that $\bar{f} (x_1) $ is a $C^4$ differentiable function on $[L_0, L_2]$ and $\bar{f} (x_1)$ is positive on $(0, L_2]$. According to the theory of ordinary differential equations, the corresponding solution $(\bar{u}, \bar{\rho})$ as for \eqref{1df} and the associated functions $\bar{k}_1$, $\bar{k}_{11}$ can be extended to the interval $[L_0, L_2]$ and satisfy the properties in \eqref{6}-\eqref{7} on $[L_0, L_2]$ if $d_0$ is sufficiently large.

 Let $\ell = \f {L_1}{20}$. Define two non-increasing cut-off functions on $[L_0, L_2]$ as follows
 \be\no
 \zeta_1 (x_1) = \begin{cases}
               1, & \mbox{if } L_0 \le x_1 \le L_1 + 2 \ell,  \\
               0, & \mbox{if } L_1 + 4 \ell \le x_1 \le L_2,
             \end{cases},
             \q
 \zeta_2 (x_1) = \begin{cases}
               1, & \mbox{if } L_0 \le x_1 \le L_1 + \ell, \\
               0, & \mbox{if } L_1 + 2 \ell \le x_1 \le L_2.
             \end{cases}
 \ee
Set
\be\no
\bar{a}_{11} (x_1) = \bar{k}_{11} (x_1) \zeta_1 (x_1) + (1 - \zeta_1 (x_1)), \\\no
\bar{a}_1 (x_1) = \bar{k}_1 (x_1) \zeta_2 (x_1) - k_0 (1 - \zeta_2 (x_1)),
 \ee
 where $k_0$ is a positive constant to be specified later. Then
\be\label{a11}
&&\bar{a}_{11} (x_1 ) = \begin{cases}
                         \bar{k}_{11} (x_1), & \mbox{if } L_0 \le x_1 \le L_1 + 2 \ell, \\
                         1 , & \mbox{if } L_1 + 4 \ell \le x_1 \le L_2,
                       \end{cases}
\\\no
&&\bar{a}_1 (x_1) = \begin{cases}
                     \bar{k}_1 (x_1), & \mbox{if } L_0 \le x_1 \le L_1 + \ell, \\
                     - k_0 , & \mbox{if }L_1 + 2 \ell \le x_1 \le L_2,
                   \end{cases}
 \ee
 and for $j = 0, 1 ,2 ,3$
 \be\no
 && 2 \bar{a}_1 + (2 j -1) \bar{a}_{11}' = 2 \bar{k}_1 \zeta_2 + (2 j -1) \bar{k}_{11}' \zeta_1 + (2j -1) (\bar{k}_{11} - 1) \zeta_1' - 2 k_0 (1 - \zeta_2)\\\no
 && = \begin{cases}
        2 \bar{k}_1 + (2 j -1) \bar{k}_{11}' \le -\ka_* < 0, & \mbox{if } L_0 \le x_1 \le L_1 + \ell, \\
        2 \bar{k}_1 \zeta_2 + (2 j -1) \bar{k}_{11}' - 2 k_0 (1 - \zeta_2), & \mbox{if } L_1 + \ell \le x_1 \le L_1 + 2 \ell, \\
        (2 j -1) \bar{k}_{11}' \zeta_1 + (2 j -1) (\bar{k}_{11} - 1) \zeta_1' - 2 k_0, & \mbox{if } L_1 + 2 \ell \le x_1 \le L_1 + 4 \ell, \\
        - 2 k_0 , & \mbox{if } L_1 + 4 \ell \le x_1 \le L_2.
      \end{cases}
 \ee
 Therefore, for sufficiently large $k_0$, $d_0 > 0$, the following inequalities hold for any $x_1 \in [L_0, L_2]$,
\be\label{8}
&&2 \bar{a}_1 + (2 j -1) \bar{a}_{11}' \le - \ka_* < 0, \q j = 0, 1, 2, 3, 4,\\\label{9}
&&(\bar{a}_1 + j \bar{a}_{11}') d - \f 12 (\bar{a}_{11} d)' \ge 4, \q  j = 0, 1, 2, 3,
\ee
where $d (x_1) = 6 (x_1 - d_0) < 0$ for any $x_1 \in [L_0, L_2]$.

Introduce an extension operator $\mc{E}$ that extends a function $f (x_1,x')$ from $\Om$ to $\Om_2 \co\{(x_1,x'): L_0 < x_1 < L_2, x'\in \mb{E}\} $ as
\be\no
\mc{E} (f) (x_1,x') = \begin{cases}
                   f (x_1, x'), & \mbox{if } (x_1, x') \in \Om, \\
                   \sum_{j =1}^4 c_j f \big(L_1 + \frac 1j (L_1 - x_1), x'\big) , & \mbox{if } (x_1, x') \in (L_1, L_2) \times \mb{E},
                 \end{cases}
\ee
where the constants $c_j, j=1,2,3,4$ are uniquely determined by the following algebraic equations
\be\no
\sum_{j=1}^4 (-1)^k j^{-k} c_j = 1, \q k = 0, 1, 2,3.
\ee
The extension operator $\mc{E} $ is a bounded operator from $H^j(\Om)$ to $H^j(\Om_2)$ for any $j = 1, 2, 3, 4$. Now one can define the extension of the operator $\mc{L}$ in \eqref{li1} to the domain $\Om_2$ as follows
\be\label{aij}\begin{cases}
a_{11} = \bar{a}_{11} + \mc{E} (k_{11} - \bar{k}_{11}),
\q  a_{1i} = a_{i1} = \mc{E} (k_{1i}), \ \ \ i=2,3, \\
a_{ii}= 1+ \mc{E}(k_{ii}-1),\ \ \ i=2,3,\\
a_{23}= a_{32}= \mc{E}(k_{23}),\ \ \ \ G(x_1,x')= \mc{E}(F(x_1,x')).
\end{cases}\ee
Then, similar to Lemma \ref{coe1}, the following Lemma holds.

\begin{lemma}\label{coe2}
The functions $a_{ij},i,j=1,2,3$ defined in \eqref{aij} satisfy the following properties
\be\no
&&\sum_{j=1}^3 \|a_{1j}-\bar{a}_{11}\delta_{1j}\|_{H^3(\Om_2)}+\sum_{i,j=2}^3 \|a_{ij}-\delta_{ij}\|_{H^3(\Om_2)}\leq C_*(\epsilon+\delta_0),\\\no
&&\sum_{i,j=2}^3 a_{ij}(x_1,x')\xi_i \xi_j \geq \frac12 (\xi_2^2+\xi_3^2),\ \ \forall (x_1,x')\in \Om_2, \forall (\xi_2,\xi_3)\in \mbR^2,\\\no
&&\displaystyle\sum_{i=2}^3 n_i a_{1i}(\nabla \psi(x_1,x'))=0,\ \ \forall (x_1,x')\in [L_0,L_2]\times \p\mb{E}.
\ee
For any $\Psi\in C^4(\ol{\Om_2})$ satisfying the boundary condition $(n_2\p_2+n_3\p_3)\Psi=0$ on $[L_0,L_2]\times \p\mb{E}$, it holds that
\be\label{coe203}
\displaystyle\sum_{i,j=2}^3 \p_1^m a_{ij}\p_j\Psi n_i=0,\ \  &\forall (x_1,x')\in [L_0,L_2]\times \p\mb{E},\ \ m=0,1,2,3.
\ee
\end{lemma}

Consider now the following auxiliary problem in $\Om_2$:
\be\label{au1}
\begin{cases}
  \mc{M} \Psi =\sum_{i,j=1}^3 a_{ij} \p_{ij}^2 \Psi + \bar{a}_1 \p_{1} \Psi= G(x_1,x'), &  (x_1,x') \in \Om_2, \\
  \Psi (L_0, x') = 0 , & \forall x' \in \mb{E},\\
  (n_2\p_2 \Psi+n_3\p_3 \Psi)(x_1,x')= 0, & \forall (x_1,x') \in [L_0,L_2]\times \p\mb{E},\\
  \p_{1} \Psi(L_2,x') = 0, & \forall x' \in \mb{E}.
\end{cases}
\ee
It follows from \eqref{a11} and Lemma \ref{coe2} that $a_{11}(L_2,x')>0$ for any $x'\in\mb{E}$, so the equation in \eqref{au1} is elliptic near the exit of the cylinder $\{(L_2,x'):x'\in\mb{E}\}$. Thus we could prescribe a Neumann boundary condition at $x_1=L_2$.

We will prove that the existence and uniqueness of $H^2(\Om_2)$ strong solution $\Psi $ to \eqref{au1} and derive some a priori estimates on $\na \p_{1}^2 \Psi $ and $\na \p_{1}^3 \Psi$ in the subregion $(\f 35 L_0, L_1 + 12 \ell) \times \mb{E}$. Since $a_{ij}\equiv k_{ij}$ hold on $\Om$ for $i,j=1,2,3$ and $\bar{a}_1(x_1)=\bar{k}_1(x_1)$ for all $x_1\in [L_0,L_1]$, we set $\psi(x_1,x')=\Psi(x_1,x')$ for any $(x_1,x')\in \Omega$, then $\psi$ would be the unique $H^2(\Omega)$ strong solution to \eqref{li1} due to Corollary \ref{uniqueness}.

To find a solution to \eqref{au1}, we consider a singular perturbation problem to \eqref{au1} by adding an additional third order dissipation term and one more boundary condition at the entrance
\be\label{au2}
\begin{cases}
  \mc{M}^{\si} \Psi^{\si}  = \si \p_{1}^3 \Psi^{\si}  + \sum_{i,j=1}^3 a_{ij} \p_{ij}^2 \Psi^{\si} + \bar{a}_1 \p_{1} \Psi^{\si} = G(x_1,x'), \q \text{in } \Om_2, \\
  \Psi^{\si}  (L_0,x') = \p_{1}^2 \Psi^{\si} (L_0,x') = 0,\q\q \forall x' \in \mb{E},\\
  (n_2\p_2 \Psi^{\si}+n_3\p_3 \Psi^{\si})(x_1,x')= 0, \q\q \forall (x_1,x') \in [L_0,L_2]\times \p\mb{E},\\
  \p_{1}\Psi^{\si}(L_2,x') = 0,\q\q \forall x' \in \mb{E},
\end{cases}
\ee
with $\sigma$ being a small positive constant. The method of singular perturbations combined with the Galerkin approximation for mixed type problems was suggested by Egorov \cite{eg87} and modified by Kuzmin \cite{ku02} using simpler cut-off functions. Here we supplement the singular perturbation problem with $\p_1^2\Psi^{\si}(L_0,x')=0$, different from the boundary condition $\p_1\Psi^{\si}(L_0,x')=0$ used by Kuzmin \cite{ku02}, which enables us to handle the three dimensional case and also simplify the a priori estimates by using some simpler but more effective multipliers.

Then the following $H^2(\Om_2)$ estimate for the solution $\Psi^{\si}$ to \eqref{au2} can be proved.

\begin{lemma}\label{aH2}
There exist positive constants $\sigma_*,\epsilon_*$ and $\de_* >0$ depending only on the background flow, the potential force and the boundary data, such that if $0<\sigma\leq\sigma_*,0<\epsilon\leq\epsilon_*$ and $0 < \de_0 \leq\de_*$ in Lemma \ref{coe2}, the classical solution $\Psi^{\si}$ to \eqref{au2} satisfies
\be\label{aH1}
  && \si\|\p_{1}^2 \Psi^{\si}\|_{L^2(\Om_2)}^2 +\|\Psi^{\si}\|_{H^1(\Om_2)}^2 + \int_{L_0}^{\frac{L_0}{10}}\iint_{\mb{E}} |\nabla\p_1\Psi^{\si}|^2\le C_* \|G\|_{L^2(\Om_2)}^2, \\\label{aH21}
  &&  \int_{\f 45 L_0}^{L_1 + 16 \ell} \iint_{\mb{E}}\si  |\p_{1}^3 \Psi^{\si}|^2 dx' dx_1 +\iiint_{\Om_2} |\na\p_1\Psi^{\si }|^2 \leq C_{\sharp} \|G\|_{H^1(\Om_2)}^2,
  \ee
where $C_{*}$ depends only on the $H^3(\Om_2)$ norms of $a_{ij}, i,j=1,2,3$, $\bar{a}_1$, and $C_{\sharp}$ depends only on the $C^2(\ol{\Om_2})$ norms of $a_{ij}, i,j=1,2,3$, $\bar{a}_1$.
\end{lemma}

\begin{proof}
For the sake of simplicity, the superscript $\si$ is omitted in the following argument. By choosing the same multiplier as in the proof of Lemma \ref{H1e}, one can derive by integrating by parts that
\be\no
&& \iiint_{\Omega_2} d(x_1) \p_{1} \Psi G  dx' dx_1  = \iiint_\Omega d(x_1) \p_{1} \Psi \mc{M}^{\sigma} \Psi dx' dx_1 \\\no
&& = \iiint_{\Omega_2}\bigg\{ -\si d (\p_{1 }^2 \Psi )^2 +\bigg(d \bar{a}_1 - \f 12  \p_{1} (d a_{11})-d\sum_{i=2}^3\p_ia_{1i}\bigg)(\p_{1} \Psi)^2 \\\no
&&+\sum_{i,j=2}^3\bigg(\frac{d'}{2} a_{ij}\p_i \Psi\p_j\Psi- d\p_1\Psi\p_i a_{ij}\p_j\Psi  +  \frac{d}{2} \p_1 a_{ij}\p_i\Psi \p_j \Psi\bigg)-6\si \p_{1} \Psi \p_{1}^2 \Psi\bigg\}\\\no
&&\q-\frac12\iint_{\mb{E}}d(L_2)\sum_{i,j=2}^3(a_{ij}\p_i\Psi\p_j\Psi)(L_2,x')+ d(L_0)(a_{11}(\p_1\Psi)^2)(L_0,x')dx'.
\ee
Similar arguments as in Lemma \ref{H1e} lead to the first part of \eqref{aH1} if $0<\sigma\leq \sigma_*$ with $\sigma_*$ as stated in the Lemma.

Define a monotonic decreasing cut-off function $\eta_1 \in C^{\oo} ([L_0, L_2])$ such that
\be\no
\eta_1 (x_1) =
\begin{cases}
  1, & \mbox{if } L_0 \le x_1 \le \f {L_0}{10}, \\
  0, & \mbox{if } \f {L_0}{20} \le x_1 \le L_2.
\end{cases}
\ee

Multiplying the first equation in $\eqref{au2}$ by $\eta_1^2 \p_{1}^2 \Psi$ and integrating over $\Om_2$ yield
\be\no
&&  \iiint_{\Om_2} \bigg((- \si \eta_1 \eta_1' + \eta_1^2 a_{11}) (\p_{1}^2 \Psi)^2 + \eta_1^2 \sum_{i,j=2}^3 a_{ij} \p_{1i}^2 \Psi\p_{1j}^2 \Psi \bigg) dx' dx_1
\\\no
&&=-\iiint_{\Om_2}\sum_{i,j=2}^3 2 \eta_1 \eta_1' a_{ij}\p_j\Psi \p_{1i}^2\Psi +\eta_1^2 \big(\p_{1}a_{ij}\p_j\Psi\p_{1i}^2\Psi-\p_i a_{ij} \p_j\Psi \p_1^2\Psi\big) dx' dx_1\\\no
&& \quad
+ \iiint_{\Om_2}\eta^2_1 \p_{1}^2 \Psi (G -\bar{a}_1 \p_{1} \Psi - 2\sum_{j=2}^3 a_{1j}\p_{1j}^2 \Psi) dx' dx_1.
\ee
Since $\eta_1 $ is monotonically decreasing, then $- \si \iiint_{\Om_2} \eta_1 \eta_1' (\p_{1}^2 \Psi)^2 dx' dx_1\geq 0$. Given that $\bar{a}_{11}(x_1)=\bar{k}_{11}(x_1)\ge 2\ka_1 > 0$ for all $x_1 \in [L_0, \f {L_0}{10}]$ with some positive constant $\ka_1$, there exists constants $\epsilon_*>0, \de_*$ such that if $0<\epsilon<\epsilon_*$ and $0 < \de_0 \le \de_*$ in Lemma \ref{coe2}, then
\be\no
a_{11} (x_1,x') \ge \ka_1 > 0, \q \forall (x_1, x') \in [L_0, \f {L_0}{10}] \times \mb{E}.
\ee
This, together with Lemma \ref{coe2}, implies
\be\no
&&\ka_1 \int_{L_0}^{\f {L_0}{10}} \iint_{\mb{E}} |\nabla \p_{1}\Psi |^2 dx' dx_1\leq \iiint_{\Omega_2} \eta_1^2 a_{11}(\p_{1}^2 \Psi)^2 + \eta_1^2 \sum_{i,j=2}^3 a_{ij} \p_{1i}^2 \Psi\p_{1j}^2 \Psi dx' dx_1\\\no
&&\leq \iiint_{\Omega_2} (\frac{\ka_1}4+\sum_{j=2}^3|a_{1j}|^2)\eta_1^2|\p_{1}\nabla \Psi|^2 + \frac{C}{\ka_1}\{|G|^2 +(\bar{a}_1^2 + \sum_{i,j=2}^3(|a_{ij}|+|\nabla a_{ij}|)^2)|\na \Psi|^2\}\\\label{aH22}
&&\leq C_* \iiint_{\Om_2} |G|^2 dx' dx_1.
\ee

Choose a monotonic increasing cut-off function $\eta_3 \in C^{\oo} ([L_0, L_2])$ such that
  \be\no
  \eta_3 (x_1) = \begin{cases}
                 0, & L_0 \le x_1 \le L_1 + 2 \ell, \\
                 1, & L_1 + 4 \ell \le x_1 \le L_2.
               \end{cases}
  \ee
Multiplying the equation \eqref{au2} by $\eta_3^2 \p_{1}^2 \Psi$ and integrating by parts over $\Om_2$ yield
  \be\no
&& \iiint_{\Om_2} \eta_3^2 \bigg(a_{11} (\p_{1}^2 \Psi)^2 +\sum_{i,j=2}^3 a_{ij}\p_{1i}^2\Psi \p_{1j}^2 \Psi\bigg)
 + \f {\si }2 \iint_{\mb{E}} (\p_{1}^2 \Psi (L_2, x'))^2 dx'
\\\label{exit}
&&=
\iiint_{\Om_2} \si \eta_3 \eta_3' (\p_{1}^2 \Psi)^2 + \eta_3^2 \p_{1 }^2 \Psi (G - \bar{a}_1 \p_{1} \Psi -2 \sum_{j=2}^3 a_{1j}\p_{1j}^2\Psi)
 \\\no
&& \q- \iiint_{\Om_2}\sum_{i,j=2}^3 2\eta_3 \eta_3' a_{ij}\p_j \Psi\p_{1i}^2 \Psi+\eta_3^2 \big(\p_1 a_{ij}\p_j\Psi \p_{1i}^2 \Psi-\p_i a_{ij}\p_j\Psi \p_{1}^2 \Psi\big),
\ee
where $\p_1^2\Psi(L_0,x')=\p_i\Psi(L_0,x')=\p_{1i}^2 \Psi (L_2,x') = 0$ for any $x'\in \mb{E}$ and $i=2,3$ have been used. The first term on the right hand side can be controlled by \eqref{aH1}. By \eqref{a11} and Lemma \ref{coe2}, it holds that
\be\no
a_{11}(x_1,x')\geq \frac 12,\ \ \forall (x_1,x')\in [L_1 + 4 \ell, L_2]\times \mb{E}.
\ee
Thus it follows from \eqref{exit} that
\be\label{aH23}
\si\iint_{\mb{E}} (\p_{1}^2 \Psi (L_2, x'))^2 dx'+\int_{L_1 + 4\ell}^{L_2} \iint_{\mb{E}} |\na \p_{x_1} \Psi|^2 dx' dx_1 \le C_* \|G_2\|_{L^2(\Om_2)}^2.
\ee

Next we derive the $L^2$ norm of $\nabla \p_1\Psi$ on the subregion $(\f45 L_0, L_1 + 16 \ell)\times \mb{E}$. Set $W_1 = \p_{1} \Psi$. Then $W_1$ solves
\be\label{W1}
\begin{cases}
 \si \p_{1}^3 W_1 + \sum_{i,j=1}^3 a_{ij} \p_{ij}^2 W_1 + (\bar{a}_1+\p_{1} a_{11}) \p_{1} W_1+ \sum_{j=2}^3 2 \p_1 a_{1j} \p_j W_1\\
 \q\q\q+ \sum_{i,j=2}^3\p_{1} a_{ij} \p_{ij}^2 \Psi= \p_{1} G - \bar{a}_1' \p_{1} \Psi, \\
  \p_{1} W_1 (L_0, x') =W_1 (L_2 , x') = 0  , \q \forall x' \in \mb{E},\\
  (n_2\p_2+n_3\p_3) W_1 (x_1, x') = 0, \q \forall (x_1,x') \in \Ga_w.
\end{cases}
\ee

Choose a cut-off function $\eta_4 \in C^{\oo} ([L_0, L_2])$ such that
\be\no
\eta_4 (x_1) = \begin{cases}
               0, & L_0 \le x_1 \le \f {17}{20} L_0, \\
               1, & \frac45 L_0 \le x_1 \le L_1 + 16 \ell, \\
               0, & L_1 + 18 \ell \le x_1 \le L_2.
             \end{cases}
\ee
Multiplying the first equation in $\eqref{W1}$ by $\eta_4^2 d(x_1) \p_{1} W_1$ and integrating over $\Om_2$ lead to
\be\no
&&\iiint_{\Om_2}- \si  \eta_4^2 d(\p_{1}^2 W_1)^2+ \frac{1}{2}(\eta_4^2d)' \sum_{i,j=2}^3 a_{ij}\p_i W_1 \p_j W_1 dx' d{x_1}\\\no
&&\q + \iiint_{\Om_2} \bigg(\eta_4^2 d(\bar{a}_1+\p_1 a_{11})-\frac12 \p_1(\eta_4^2 d a_{11})-\eta_4^2 d \sum_{i=2}^3 \p_i a_{1i}\bigg)(\p_1 W_1)^2 dx' dx_1\\\no
&&=\iiint_{\Om_2}\bigg\{ \eta_4^2 d\p_{1} W_1 (\p_{1} G - \bar{a}_1' \p_{1} \Psi-2\sum_{i=2}^3 \p_1 a_{1i} \p_i W_1)+ \si (\eta_4^2 d)' \p_{1} W_1  \p_{1 }^2 W_1 \\\no
&&\q\q-(\eta_4^2 d)' \sum_{i,j=2}^3\p_1 a_{ij}\p_j\Psi \p_i W_1-\eta_4^2 d \sum_{i,j=2}^3\bigg((\frac32\p_1 a_{ij}\p_j W_1+\p_1^2 a_{ij}\p_j\Psi) \p_i W_1\\\label{W10}
&&\q\q\q\q-(\p_i a_{ij} \p_j W_1+\p_{1i}^2 a_{ij}\p_j\Psi)\p_1 W_1\bigg)\bigg\}.
\ee
Here it should be emphasized that
\be\no
&&\iiint_{\Om_2}\eta_4^2 d \p_1 W_1 \sum_{i,j=2}^3\p_1 a_{ij}\p_{ij}^2\Psi dx' dx_1\\\label{W105}
&&= - \iiint_{\Omega_2}\eta_4^2 d \sum_{i,j=2}^3 (\p_1 a_{ij}\p_j \Psi \p_{1i}^2 W_1+ \p_{1i}^2 a_{ij}\p_j\Psi \p_1 W_1) dx' dx_1\\\no
&&=-\underline{\iint_{\mb{E}}\eta_4^2 d\sum_{i,j=2}^3\p_1 a_{ij}\p_j\Psi\p_i W_1\b|_{x_1=L_0}^{L_2}dx'}+ \sum_{i,j=2}^3\iiint_{\Om_2}(\eta_4^2 d)' \p_1 a_{ij} \p_j\Psi \p_i W_1 \\\no
&&\q+ \sum_{i,j=2}^3\iiint_{\Om_2} \eta_4^2 d (\p_1 a_{ij}\p_j W_1 \p_i W_1 + \p_1^2 a_{ij}\p_j\Psi \p_i W_1-\p_{1i}^2 a_{ij}\p_j\Psi \p_1 W_1),
\ee
where in the first equality, the boundary term vanishes due to \eqref{coe203} for $m=1$. The underlined boundary term vanishes since $\eta_4(L_0)=\eta_4(L_2)=0$. To estimate the term in the last line, we need one order more regularity on the coefficients $(a_{ij})$ in comparison with the 2D and 3D axisymmetric cases \cite{wx23a,wz24a}. Note that if one estimates $\|\p_1^2\psi\|_{L^2}$ for \eqref{li1} directly, the corresponding boundary terms do not vanish and would be out of control. It follows from \eqref{8}-\eqref{9} and the estimates in Lemma \ref{coe2} that
\be\no
 && \eta_4^2 d (\bar{a}_1+\p_{1} a_{11}) - \f 12 \p_{1} (\eta^2_4 d a_{11})- \eta_4^2 d (\p_2 a_{12}+\p_3 a_{13})\\\no
 &&
\geq  \eta_4^2 [4 - \|d(\p_{1} a_{11}-\bar{a}_{11}') \|_{L^{\oo}} - \f 12 \| \p_{1} (d a_{11} - d \bar{a}_{11}) \|_{L^{\oo}}- \| d \sum_{j=2}^3 \p_j a_{1j}\|_{L^{\oo}} ]
\\\no
&& \q\q- \eta_4 \eta_4' d a_{11} \geq 3 \eta_4^2 - \eta_4 \eta_4' a_{11} d, \q \forall (x_1, x') \in \Om_2.
\ee
Then it follows from \eqref{W10} that
\be\no
&&\iiint_{\Om_2}\si \eta_4^2 |\p_1^2 W_1|^2 + 3\eta_4^2 |\nabla W_1|^2 dx' dx_1 \\\no
&&\leq \iiint_{\Om_2} \bigg\{\eta_4^2(|\p_1 G|^2+|\p_1\Psi|^2+\sigma |\p_1 W_1|^2)+ |\eta_4'|^2(\si+\sum_{i,j=1}^3|a_{ij}|)|\nabla W_1|^2\\\no
&&\q + |\eta_4'|^2\sum_{i,j=2}^3|\p_1 a_{ij}|^2|\p_j\Psi|^2+\eta_4^2 \bigg(\sum_{i=2}^3 |\p_1 a_{1i}|+ \sum_{i,j=2}^3|\nabla a_{ij}|\bigg)|\nabla W_1|^2 \\\no
&&\q + \eta_4^2\sum_{i,j=2}^3 (|\p_1 a_{ij}|^2 + |\nabla\p_1 a_{ij}|^2) |\p_j\Psi|^2 +\frac12 \eta_4^2 |\nabla W_1|^2\bigg\} dx' dx_1.
\ee
Using \eqref{aH22} and \eqref{aH23} to control the second term on the right hand side in the above inequality, one can derive
\be\label{aH231}
\int_{\frac{4}{5} L_0}^{L_1+16\ell}\iint_{\mb{E}}\si (\p_1^2 W_1)^2+ |\nabla W_1|^2 dx' dx_1\leq C_{\sharp} \|G\|_{H^1(\Om_2)}^2.
\ee
The estimate \eqref{aH21} follows from \eqref{aH22}-\eqref{aH231} immediately.

\end{proof}

With the estimates \eqref{aH1} and \eqref{aH21} at hand, the existence of $H^2(\Om_2)$ strong solution to \eqref{au1} can be proved as follows. We first prove the existence of the weak $H^2(\Om_2)$ solution to the singular perturbation problem \eqref{au2} by a Galerkin's method and a limit argument. Let $\{b_m(x')\}_{m=0}^{\infty}$ be a family of all eigenfunctions associated to the eigenvalue problem
\be\label{ei}\begin{cases}
-(\p_2^2+\p_3^2) u(x')= \lambda u(x'),\ \ &\forall x' \in \mb{E},\\
(n_2\p_2+n_3\p_3) u=0,\ \ &\text{on }\p\mb{E}.
\end{cases}\ee
It is well-known that the eigenvalue problem \eqref{ei} has a family of eigenvalues $ 0= \lambda_0 < \lambda_1 \leq \lambda_2 \leq ... $  and $ \lambda_m \to +\infty$  as $m \to +\infty$ with corresponding eigenvectors $(b_m(x'))_{m=0}^{\infty}$, which constitute an orthonormal basis in $L^2(\mb{E})$ and an orthogonal basis in $H^1(\mb{E})$. 

Define the approximate solutions as
\be\label{app}
\Psi^{N, \si} (x_1,x') = \sum_{j =0}^N A_j^{N, \si} (x_1) b_j (x'),
\ee
where $\{A_j^{N,\sigma}\}_{j=0}^N$ are chosen such that the following $N+1$ linear equations on $[L_0, L_2]$ and boundary conditions are satisfied
\be\no
\begin{cases}
\iint_{\mb{E}} \mc{M}^{\si} \Psi^{N, \si } (x_1, x') b_m (x') dx' = \iint_{\mb{E}} G(x_1,x') b_m (x') dx' , \q\q m = 0, \cdots, N, \\
\Psi^{N, \si} (L_0, x')=\p_{1}^2 \Psi^{N, \si } (L_0, x') =  0,\\
\p_{1}\Psi^{N, \si } (L_2,x') =0,
\end{cases}
\ee

Thus $\{A_m^{N,\sigma}(x_1)\}_{m=0}^N$ solve the following boundary value problem for a system of ordinary differential equations
\be\label{app2}\begin{cases}
 \si \f {d^3}{ d x_1^3} A_m^{N,\si}
   + \sum_{j= 0}^N \bigg(a_{m j} \f {d^2}{d x_1^2} A_j^{N, \si} +  b_{m j}\frac d {d x_1} A_j^{N, \si} + c_{m j} A_j^{N,\si}\bigg)
   = G_{m} (x_1),\\
   A_m^{N, \si} (L_0) = \frac{d^2}{d x_1^2} A_m^{N,\si} (L_0) = 0,\ \ \ m=0,1,\cdots, N,\\
   \frac {d}{d x_1} A_m^{N, \si } (L_2) = 0
 \end{cases}\ee
where
\be\no\begin{cases}
a_{mj}(x_1) = \iint_{\mb{E}} a_{11} (x_1,x') b_j (x') b_m (x') dx',\\
 b_{mj}(x_1) = \iint_{\mb{E}} 2\sum_{q=2}^3 a_{1q}(x_1,x') \p_{\ell}b_j(x') b_m (x') dx'+\bar{a}_1(x_1)\delta_{mj},\\
 c_{mj} (x_1)= \iint_{\mb{E}} \sum_{i,q=2}^3 a_{i q}(x_1,x')\p_{i q}^2 b_j(x') b_m (x') dx', \\
G_{m} (x_1) = \iint_{\mb{E}} G(x_1, x') b_m (x') dx'.
\end{cases} \ee

\begin{lemma}\label{exist1}
There exists a unique smooth solution $\{A_j^{N, \si }\}_{j=0}^N$ to \eqref{app2} such that $\Psi^{N, \si }$ defined in \eqref{app} satisfies the following estimates
   \be\label{Ha1}
   &&\si\|\p_1^2\Psi^{N,\sigma}\|_{L^2(\Om_2)}^2+\|\Psi^{N,\sigma}\|_{H^1(\Omega_2)}^2+\int_{L_0}^{\frac1{10} L_0}\iint_{\mb{E}} |\nabla\p_1\Psi^{N,\sigma}|^2 \leq C_* \|G\|_{L^2(\Om_2)}^2,\\\label{Ha2}
   &&\int_{\f 9{10} L_0}^{L_1 + 16 \ell} \iint_{\mb{E}}\si  |\p_{1}^3 \Psi^{N,\si}|^2 dx' dx_1+\|\na\p_1\Psi^{N,\si}\|_{L^2(\Om_2)}^2\leq C_{\sharp} \|G\|_{H^1(\Om_2)}^2,
   \ee
 where $C_*, C_{\sharp}$ are positive constants with $C_*$ depending only on the $H^3(\Om_2)$ norms of $a_{ij}, i,j=1,2,3$, $\bar{a}_1$, while $C_{\sharp}$ depending on the $C^2(\overline{\Om_2})$ norms of $a_{ij}$, $\bar{a}_1$.
 \end{lemma}
 \begin{proof}

Recall that $d(x_1)=6(x_1-d_0)$. Multiplying the $m^{th}$ equation in \eqref{app2} by $d(x_1)\frac{d}{dx_1} A_m^{N,\sigma}$, summing from
$0$ to $N$, and integrating over $[L_0,L_2]$, one can get that
\begin{eqnarray}\no
\iiint_{\Omega_2}(\mathcal{M}^{\sigma} \Psi^{N,\sigma}-G)d(x_1)\partial_{1} \Psi^{N,\sigma}dx' dx_1=0.
\end{eqnarray}
Integrating by parts and arguing as in Lemma \ref{aH2} yield
\be\no
\iint_{\Omega_2}\sigma  |\p_{1}^2 \Psi^{N,\sigma}|^2  + |\Psi^{N,\sigma}|^2 + |\nabla\Psi^{N,\sigma}|^2 dx' dx_1\leq C_* \|G\|_{L^2(\Om)}^2.
\ee
This implies the uniqueness of the solution to the Problem \eqref{app2}. For the system of $N+1$ third-order equations supplemented with $3(N+1)$ boundary conditions, the
uniqueness ensures the existence of the solution to \eqref{app2}, one may refer to \cite[Lemma 2.7]{wx23b} for a proof. Thus the existence and uniquness of the approximate solution to the system \eqref{app2} are established. Since the coefficients in \eqref{app2} are in $C^4([L_0,L_2])$, so $A_m^{N,\sigma}$ are in $C^5([L_0,L_2])$.

Let $\eta_i, i=1,3$ be the cut-off functions defined in the proof of Lemma \ref{aH2}. Multiplying the $m^{th}$ equation in \eqref{app2} by $\eta_i^2(x_1)\frac{d^2}{dx_1^2} A_m^{N,\sigma}$ for $i=1,3$ respectively, summing from $0$ to $N$, and integrating over $[L_0,L_2]$, one can argue as for \eqref{aH22} and \eqref{aH23} to get
\be\no
&&\int_{L_0}^{\frac{L_0}{10}}\iint_{\mb{E}} |\nabla\p_{1}\Psi^{N,\sigma}|^2 dx' d x_1\leq  C_*\iiint_{\Omega_2} G^2 dx' dx_1,\\\no
&&\int_{L_1+4\ell}^{L_2}\iint_{\mb{E}} |\nabla\p_{1}\Psi^{N,\sigma}|^2 dx' d x_1\leq  C_{*}\iiint_{\Omega_2} G^2 dx' dx_1.
\ee
Similarly, letting $\eta_4$ be defined in the proof of Lemma \ref{aH2}, multiplying the $m^{th}$ equation in \eqref{app2} by $\frac{d}{dx_1}(\eta_4^2(x_1)d(x_1)\frac{d^2}{dx_1^2} A_m^{N,\sigma})$, summing from $0$ to $N$, and integrating over $[L_0,L_2]$, one can derive an estimate as \eqref{aH231} and thus \eqref{Ha2} follows.

 \end{proof}

Then one can prove that
\begin{lemma}\label{exist}
There exists a unique strong solution $\Psi$ in $H^2(\Om_2)$ to \eqref{au1} satisfying
\be\label{au100}
&&\|\Psi\|_{H^1(\Om_2)}^2 +\int_{L_0}^{\frac1{10} L_0}\iint_{\mb{E}} |\nabla^2\Psi|^2 dx' dx_1\leq C_{*} \|G\|_{L^2(\Om_2)}^2,\\\label{au101}
&&\|\Psi\|_{H^2(\Om_2)} \le C_{\sharp} \|G\|_{H^1(\Om_2)},
\ee
where $C_*$ and $C_{\sharp}$ are positive constants, $C_*$ depends only on the $H^3(\Om_2)$ norms of $a_{ij}, i,j=1,2,3$, $\bar{a}_1$, and $C_{\sharp}$ depends only on the $C^2(\ol{\Om_2})$ norms of $a_{ij}$, $\bar{a}_1$.
\end{lemma}

\begin{proof}

Let $\Psi^{N,\sigma}$ be constructed in Lemma \ref{exist1}. Thanks to \eqref{Ha1}, $\|\Psi^{N,\sigma}\|_{H^1(\Omega_2)}$ is uniformly
bounded in $N$. Therefore, due to the weak compactness of a bounded set in a Hilbert space, there exists a subsequence, which is still denoted  by $\Psi^{N,\sigma}$ for simplicity, converging strongly in $L^2(\Omega_2)$ and weakly in $H^1(\Omega_2)$ to a limit $\Psi^{\sigma}\in H^1(\Omega_2)$. Due to \eqref{Ha2}, the sequence $\p_1 \Psi^{N,\sigma}$ will converge strongly in $L^2(\Om_2)$ and weakly in $H^1(\Om_2)$ to $\p_1\Psi^{\si}$ as $N$ tends to infinity. The subsequence $\p_1^3\Psi^{N,\sigma}$ will converge weakly in $L^2((\frac{4}{5}L_0, L_1+16\ell)\times \mb{E})$ to $\p_1^3 \Psi^{\sigma}$. The limit $\Psi^{\sigma}$ retains the boundary condition $\Psi^{\si}(L_0,x')=0$ for almost every $x'\in \mb{E}$. Moreover, $\Psi^{\sigma}$ satisfies the following uniform estimates:
\be\label{Hs1}
&&\si\|\p_1^2\Psi^{\si}\|_{L^2(\Om_2)}^2+\|\Psi^{\sigma}\|_{H^1(\Omega_2)}^2+\int_{L_0}^{\frac1{10} L_0}\iint_{\mb{E}} |\nabla\p_1\Psi^{\sigma}|^2\leq C_{*} \|G\|_{L^2(\Omega_2)}^2,\\\label{Hs2}
&& \int_{\f 45 L_0}^{L_1 + 16 \ell} \iint_{\mb{E}}\si  |\p_{1}^3 \Psi^{\si}|^2 dx' dx_1+\|\na\p_1\Psi^{\si}\|_{L^2(\Om_2)}^2\leq C_{\sharp} \|G\|_{H^1(\Om_2)}^2,
\ee
with some positive constants $C_*$ and $C_{\sharp}$ as stated in the Lemma.

Now we show that $\Psi^{\sigma}$ is a weak solution to the system \eqref{au2}. Given any test function $\xi(x_1,x')= \sum_{m=0}^{N_0} \xi_m(x_1) b_m(x')$, where $\xi_m(x_1)\in
C_c^{\infty}((L_0,L_2))$ and $N_0$ is an arbitrary positive integer. Let $N\geq N_0$. Multiplying each of the equations in \eqref{app2} by $\xi_m$ ($\xi_m\equiv 0$ for any $N_0+1\leq m\leq N$), then sum up from
$m=0$ to $m=N$, and integrate with respect to $x_1$ from $L_0$ to $L_2$, one gets that
\be\no
\iiint_{\Omega_2}(\sigma \p_{1}^3 \Psi^{N,\sigma}+\sum_{i,j=1}^3 a_{ij}\p_{ij}^2 \Psi^{N,\sigma}+\bar{a}_1 \p_{1}\Psi^{N,\sigma}) \xi dx' dx_1=0.
\ee
Integrating by parts and passing to the limit for the above weak convergent subsequence of $\psi^{N,\sigma}$ yield
\be\label{au202}
&&\iiint_{\Omega_2}\sigma \p_{1}^2\Psi^{\sigma}\p_{1}\xi+\p_{1}\Psi^{\sigma}\p_{1}(a_{11}\xi)+2\sum_{i=2}^3\p_{1} \Psi^{\sigma}\p_{i}(a_{1i}\xi)\\\no
&&\quad\quad\q\q\q\q +\sum_{i,j=2}^3 \p_{j}\Psi^{\sigma} \p_{i}(a_{ij}\xi)+\xi (G-\bar{a}_1 \p_{1}\Psi^{\sigma}) dx'dx_1=0.
\ee

By a density argument, \eqref{au202} holds for any test function $\xi\in H^1(\Omega)$ vanishing at $x_1=L_0$ and $x_1=L_2$. Next we consider a subsequence of the solutions $\{\Psi^{\sigma}\}$ as $\sigma\to 0$. Thanks to \eqref{Hs1}, $\|\Psi^{\sigma}\|_{H^1(\Omega_2)}$ is
uniformly bounded independent of $\sigma$. This further implies the existence of a weakly convergent subsequence labeled as $\{\Psi^{\sigma_j}\}_{j=1}^{\infty}$ with $\sigma_j\to 0$ as $j\to \infty$ converging weakly in $H^1(\Om_2)$ to a limit $\Psi\in H^1(\Omega_2)$. The subsequence $\{\p_1\Psi^{\si}\}$ converges strongly in $L^2(\Om_2)$ and weakly in $H^1(\Om_2)$ to $\p_1\Psi$. Thus $\Psi$ retains the boundary condition $\Psi(L_0,x')=0$ for almost every $x'\in \mb{E}$ and satisfies the following estimates
\be\label{Hp1}
&&\|\Psi\|_{H^1(\Omega_2)}^2+\int_{L_0}^{\frac1{10} L_0}\iint_{\mb{E}} |\nabla\p_1\Psi|^2 dx' dx_1 \leq C_{*} \|G\|_{L^2(\Omega_2)}^2,\\\label{Hp2}
&&\|\na\p_1\Psi\|_{L^2(\Om_2)}\leq C_{\sharp} \|G\|_{H^1(\Om_2)}^2.
\ee


It follows from \eqref{au202} that
\be\label{p1}
&&\iiint_{\Omega_2}\p_{1}\Psi\p_{1}(a_{11}\xi)+ 2\sum_{i=2}^3\p_{1} \Psi \p_{i}(a_{1i}\xi)\\\no
&&\q\q\q\q +\sum_{i,j=2}^3\p_{j}\Psi \p_{i}(a_{ij}\xi)+(G-\bar{a}_1\p_1\Psi)\xi dx' dx_1=0,
\ee
holds for any $\xi\in H^1(\Omega_2)$ vanishing at $x_1=L_0$ and $x_1=L_2$. Rewrite \eqref{p1} as
\be\label{p2}
&&\iiint_{\Om_2}\p_1 \Psi\p_1\xi + \sum_{i,j=2}^3 a_{ij}\p_j\Psi \p_i\xi dx' dx_1\\\no
&&=\iiint_{\Om_2}\xi\bigg((a_{11}-1)\p_1^2\Psi+ 2\sum_{i=2}^3 a_{1i}\p_{1i}^2 \Psi- \sum_{i,j=2}^3 \p_i a_{ij}\p_j\Psi+ \bar{a}_1 \p_1\Psi-G\bigg) dx' dx_1
\ee
for any $\xi\in H^1(\Omega_2)$ vanishing at $x_1=L_0$ and $x_1=L_2$. Therefore $\Psi\in H^1(\Om_2)$ can be regarded as a weak solution to the following second order elliptic equation
\be\no
-\p_1^2 \Psi -\sum_{i,j=2}^3 \p_i(a_{ij}\p_j\Psi)= \tilde{G}\in L^2(\Om_2),
\ee
where
\be\no
\tilde{G}=(a_{11}-1)\p_1^2\Psi+ 2\sum_{i=2}^3 a_{1i}\p_{1i}^2 \Psi- \sum_{i,j=2}^3 \p_i a_{ij}\p_j\Psi+ \bar{a}_1 \p_1\Psi-G.
\ee
By the interior and boundary $H^2$ estimate near the cylinder wall for elliptic equations in \cite[Chapter 8]{gt}, one has
\be\label{p4}
\int_{\frac{4}{5}L_0}^{L_1+16\ell}\iint_{\mb{E}}\sum_{i,j=2}^3 |\p_{ij}^2\Psi|^2 dx' dx_1\leq C\|\tilde{G}\|_{L^2(\Om_2)}^2\leq C_{\sharp}\|G\|_{H^1(\Om_2)}^2.
\ee
Since $a_{11}\geq \ka_1$ on $[L_0,\frac{1}{10}L_0]\times \ol{\mb{E}}$, one can improve easily the regularity of $\Psi$, the second order derivatives $\p_{ij}^2\Psi, i,j=1,2,3$ exist in the classical sense on $(L_0,\frac{1}{10}L_0]\times \ol{\mb{E}}$, and the following equation for $\Psi$ holds:
\be\label{p5}
\sum_{i,j=1}^3 a_{ij}\p_{ij}^2 \Psi + \bar{a}_1\p_1 \Psi=G,\ \ \forall (x_1,x')\in (L_0,\frac{1}{10}L_0]\times \ol{\mb{E}}.
\ee
For each $x_1\in (L_0,\frac{L_0}{10}]$, rewrite \eqref{p5} as
\be\no\begin{cases}
\sum_{i=2}^3 \p_i^2\Psi(x_1,x')=\{G- a_{11}\p_1^2\Psi- 2\sum_{i=2}^3 a_{1i}\p_{1i}^2\Psi\\
\q\q\q\q-\bar{a}_1\p_1\Psi- \sum_{i,j=2}^3(a_{ij}-\delta_{ij})\p_{ij}^2\Psi\}(x_1,x'),\\
(n_2\p_2+n_3\p_3)\Psi(x_1,x')=0,\ \ \ \forall x'\in \p\mb{E}.
\end{cases}\ee
Then the classical elliptic theory in \cite{gt} implies that for each $x_1\in (L_0,\frac{L_0}{10}]$, there holds
\be\no
&&\|\nabla_{x'}^2 \Psi(x_1,x')\|_{L^2(\mb{E})}^2\leq C(\mb{E})\bigg(\|G-\bar{a}_1\p_1\Psi\|_{L^2(\mb{E})}^2+ \|a_{11}\p_1^2\Psi\|_{L^2(\mb{E})}^2\\\no
&&\q\q+2\sum_{i=2}^3 \|a_{1i}\p_{1i}^2\Psi\|_{L^2(\mb{E})}^2+\sum_{i,j=2}^3\|(a_{ij}-\delta_{ij})\p_{ij}^2\Psi\|_{L^2(\mb{E})}^2\bigg).
\ee
Integrating the above inequality with respect to $x_1$ over $[L_0,\frac{1}{10} L_0]$ gives
\be\label{p8}
&&\|\nabla_{x'}^2 \Psi\|_{L^2(\Om_{1/10})}^2\leq C(\mb{E})\big(\|G\|_{L^2(\Om)}^2+\|\nabla\Psi)\|_{L^2(\Om)}^2\\\no
&&\q\q\q\q+\|\nabla \p_1\Psi\|_{L^2(\Om_{1/10})}^2+ (\epsilon+\delta_0)\|\nabla_{x'}^2 \Psi\|_{L^2(\Om_{1/10})}^2\big)
\ee
where $\Om_{1/10}=(L_0,\frac{1}{10} L_0)\times \mb{E}$. Let $\epsilon+\delta_0$ be small enough such that $C(\mb{E})(\epsilon+\delta_0)\leq \f12$. Then one gets from \eqref{Hp1} and \eqref{p8} that
\be\label{p9}
\|\nabla_{x'}^2 \Psi\|_{L^2(\Om_{1/10})}\leq C_{\sharp}\|G\|_{H^1(\Om_2)}.
\ee
Thus the estimate \eqref{au100} follows from \eqref{Hp1} and \eqref{p9}. Since $a_{11}\geq \f12$ on $[L_1+4\ell,L_2]\times \ol{\mb{E}}$, a similar argument as for \eqref{p9} shows that
 \be\label{p10}
\|\nabla_{x'}^2 \Psi\|_{L^2([L_1+4\ell,L_2]\times \mb{E})}\leq C_{\sharp}\|G\|_{H^1(\Om_2)}.
\ee
Thus the estimate \eqref{au101} follows from \eqref{Hp2},\eqref{p4},\eqref{p9} and \eqref{p10}.

\end{proof}

\begin{lemma}\label{exist2}
   The problem \eqref{li1} has a unique $H^2(\Om)$ strong solution $\psi (x_1,x')$ such that
   \be\label{H1}
   &&\iiint_{\Om} ( |\psi |^2 + |\na \psi|^2) dx' dx_1+ \int_{L_0}^{\frac1{10} L_0}\iint_{\mb{E}} |\nabla^2\psi|^2\leq C_* \|F\|_{L^2(\Om)}^2,\\\label{H2}
   &&\iiint_{\Om} |\nabla^2\psi|^2 dx' dx_1\leq C_{\sharp}\|F\|_{H^1(\Om)}^2,
   \ee
  where $C_*, C_{\sharp}$ are positive constants with $C_*$ depending only on the $H^3(\Om)$ norms of $k_{ij}$, $\bar{k}_1$, and $C_{\sharp}$ depending only on the $C^2(\ol{\Om})$ norms of $k_{ij}, i,j=1,2,3$, $\bar{k}_1$.
\end{lemma}

 \begin{proof}
Let $\psi(x_1,x')$ be the restriction on the domain $\Omega$ of the function $\Psi$ obtained in Lemma \ref{exist}. Then $\psi\in H^2(\Omega)$ solves the first equation in \eqref{li2} almost everywhere and also satisfies the boundary conditions in \eqref{li2}. The uniqueness of $H^2(\Om)$ strong solution to the problem \eqref{li1} just follows from Corollary \ref{uniqueness}. Finally, the estimates on $\psi$ in\eqref{H1} and \eqref{H2} follow from \eqref{au100} and \eqref{au101} respectively, since $\|G\|_{L^2(\Omega_2)}\leq C\|F\|_{L^2(\Om)}$ and $\|G\|_{H^1(\Omega_2)}\leq C\|F\|_{H^1(\Om)}$.

\end{proof}

\section{The $H^4$ estimate and the proof of Theorem \ref{3dirro}}\label{high}\noindent

Since $k_{11}(x_1,x')$ changes sign as the fluid passes through the sonic curve, the equation \eqref{li1} is elliptic on subsonic regions and then becomes hyperbolic thereafter, we will analyze the regularity of the solution in the subsonic region and the transonic region separately.

\subsection{The $H^4$ estimate in subsonic regions}
In this subsection, we will derive the $H^4$ estimates for the solution of \eqref{li1} in subsonic regions. Note that the principal part of the equation in \eqref{li1} can be regarded as a small perturbation of $\p_1(\bar{k}_{11}\p_1\psi)+ \p_2^2\psi +\p_3^2\psi$ in $\Om_{1/10}$. To improve the regularity near $\{(L_0,x'): x'\in \p\mb{E}\}$, we utilize a method of separation of variables and the Banach's contraction mapping theorem to derive the $H^4$ estimate of $\psi $ on $\Om_{1/5}\co \{ (x_1,x'): L_0< x_1 < \f {L_0}{5}, x'\in \mb{E}\}$.

Consider the following boundary value problem in the cylinder $\Om_{1/10}$:
\be\label{li5}\begin{cases}
\p_1(\bar{k}_{11}\p_1 v) + \p_2^2 v + \p_3^2 v= g(x_1,x'),\ \ &\text{in }\ \ \Om_{1/10},\\
v(L_0,x')= v(\f1{10} L_0,x')=0, \ \ &\forall x'\in \mb{E},\\
(n_2\p_2+n_3\p_3) v(x_1,x')=0,\ \ &\forall (x_1,x')\in [L_0,\f1{10} L_0]\times \mb{E}.
\end{cases}\ee

Let $\mc{V}$ be the set of $g\in C^{\infty}(\ol{\Om_{1/10}})$ satisfying the boundary condition
\be\label{gcp}
(n_2\p_2 + n_3\p_3)g(L_0,x')=(n_2\p_2 + n_3\p_3)g(\frac1{10}L_0,x')=0\ \ \ \ \text{in }L^{2}(\p\mb{E}).
\ee
Denote by $H^{\varsigma}_{cp}(\Om_{1/10})$ the closure of $\mc{V}$ in the Sobolev spaces $\|\cdot\|_{H^\varsigma(\Om_{1/10})}$ for $\varsigma=1,2$.

\begin{lemma}\label{li50}
Suppose that $g\in H_{cp}^\varsigma(\Om_{1/10})$ for $\varsigma=1,2$. Then there exists a unique solution $v\in H^{2+\varsigma}(\Om_{1/10})$ to the problem \eqref{li5} satisfying
\be\label{li51}
\|v\|_{H^{2+\varsigma}(\Om_{1/10})}\leq C_*\|g\|_{H^\varsigma(\Om_{1/10})}, \ \ \varsigma=1,2.
\ee
where the positive constant $C_*$ depends only on $\bar{k}_{11}, L_0$ and $\mb{E}$.
\end{lemma}
\begin{proof}
Since $g\in H^\varsigma_{cp}(\Om_{10}) (\varsigma=1,2)$, one can express $g(x_1,x')$ as
\begin{equation}\label{four1}
    g(x_1,x') = \sum_{m=0}^{\infty} g_m(x_1)b_m(x'),
\end{equation}
where $\{b_m(x')\}_{m=0}^{\infty}$ are eigenfunctions to the eigenvalue problem \eqref{ei} and
\be\no
g_m(x_1)= \iint_{\mb{E}} g(x_1,x') b_m(x') dx',\ \ \ m\geq 0.
\ee
The series in \eqref{four1} converges in the $H^\varsigma(\Om_{10})$ norm. Note that $g_m\in H^\varsigma([L_0,\frac1{10} L_0])$ for $\varsigma=1,2$. In the following, we assume that $g_m\in C^2([L_0,\frac1{10} L_0])$ and the general case will follow by a density argument.

The unique solution to \eqref{li5} with $g$ replaced by $g_N(x_1,x')= \sum_{m=0}^{N} g_m(x_1)b_m(x')$ can be represented as $v_N(x_1,x')=\sum_{m=0}^{N} a_m(x_1) b_m(x')$, where
\begin{eqnarray}\label{an}
\begin{cases}
(\bar{k}_{11} a_m')'(x_1) - \lambda_m a_m(x_1) = g_m(x_1),\\
a_m(L_0)=a_m(\frac1{10} L_0)=0.
\end{cases}
\end{eqnarray}
The equation \eqref{an} for $m\geq 0$ has a unique classical solution $a_m$ in $C^3([L_0,\frac1{10} L_0])$. Then $v_N\in C^3(\ol{\Om_{1/10}})$ is a classical solution to \eqref{li5} and it is easy to check that $g_N$ must satisfy the compatibility condition \eqref{gcp}. This is automatically true, since $(n_2\p_2+n_3\p_3)b_m(x')=0$ hold for any $x'\in \p\mb{E}$ and all $m\geq 0$.

In the following, we prove that the sequence $\{v_N\}_{N=1}^{\infty}$ converges in $H^{2+\varsigma}(\Om_{1/10})$ to $v(x_1,x')=\sum_{m=0}^{\infty} a_m(x_1) b_m(x')$ for $\varsigma=1,2$, which solves \eqref{li5} with \eqref{li51}.

Squaring the equation in \eqref{an} and integrating over $[L_0,\frac1{10} L_0]$ yield
\be\no
\int_{L_0}^{\f 1{10} L_0} g_m^2(x_1)dx_1=\int_{L_0}^{\f 1{10} L_0}  |(\bar{k}_{11} a_m')'|^2+2\lambda_m\bar{k}_{11}|a_m'|^2+\lambda_m^2|a_m|^2 dx_1,
\ee
which implies that for $m>0$:
\be\label{an1}
\int_{L_0}^{\f 1{10} L_0} |a_m''|^2+ \la_m |a_m'|^2 + \la_m^2|a_m|^2 dx_1\leq C_*\int_{L_0}^{\f 1{10} L_0} g_m^2 dx_1.
\ee
For $m=0$, $\la_0=0$, it follows from \eqref{an} and the Poincare's inequality that
\be\no
\int_{L_0}^{\f 1{10} L_0} \bar{k}_{11} (a_0')^2 dx_1=- \int_{L_0}^{\f 1{10} L_0} g_0 a_0 dx_1\leq C_*\|g_0\|_{L^2(L_0,\f 1{10}L_0]} \|a_0'\|_{L^2(L_0,\f 1{10}L_0]}.
\ee
Thus using the equation \eqref{an} and the Poincare's inequality again, one obtains
\be\label{an2}
&&\int_{L_0}^{\f 1{10} L_0} (|a_0''|^2 + |a_0'|^2 + a_0^2 )dx_1\leq C_*\int_{L_0}^{\f 1{10} L_0} |g_0|^2 dx_1,\\\no
&&\int_{L_0}^{\f 1{10} L_0}|a_0^{(4)}(x_1)|^2+|a_0^{(3)}(x_1)|^2 dx_1\leq C_*\int_{L_0}^{\f 1{10} L_0}(|g_0''|^2+|g_0'|^2+|g_0|^2) dx_1.
\ee
Note that for each $x_1\in [L_0,\frac{1}{10}L_0]$, $v_N(x_1,x')$ regarded as a function defined on $\mb{E}$ satisfies
\be\no\begin{cases}
\displaystyle\sum_{i=2}^3 \p_i^2 v_N(x_1,x')= \sum_{m=1}^N \lambda_m a_m(x_1) b_m(x'),\ \ \ &\forall x'\in \mb{E},\\
(n_2\p_2+n_3\p_3)v_N(x_1,x')=0,\ \ \ &\forall x'\in\p\mb{E}.
\end{cases}\ee
Then the classical elliptic regularity theory implies that
\be\no
\|\nabla_{x'}^i v_N(x_1,x')\|_{L^2(\mb{E})}^2\leq C_*\|\sum_{m=1}^N \lambda_m a_m(x_1) b_m(x')\|_{H^{i-2}(\mb{E})}^2, \ \ \ i=2,3,4.
\ee
Integrating the above inequalities on $[L_0,\frac{1}{10} L_0]$ yields that
\be\label{an26}
&&\|\nabla_{x'}^i v_N\|_{L^{2}(\Om_{1/10})}\leq C_*\|\sum_{m=1}^N \lambda_m a_m(x_1) b_m(x')\|_{H^{i-2}(\Om_{1/10})},\ \ \ \ i=2,3,4.
\ee
The estimate \eqref{an26} for $i=2$, together with \eqref{an1} and \eqref{an2}, implies that
\be\no
&&\|v_N\|_{H^2(\Om_{1/10})}^2\leq C_*\bigg(\int_{L_0}^{\f 1{10} L_0}\sum_{i=0}^2|a_0^{(i)}|^2 dx_1+\sum_{m=1}^{N}\int_{L_0}^{\f 1{10} L_0}\sum_{i=0}^2\la_m^{2-i}|a_m^{(i)}|^2dx_1\bigg)\\\no
&&\leq C_*\sum_{m=0}^{N}\int_{L_0}^{\f 1{10} L_0} g_m^2 dx_1\leq C_*\|g\|_{L^2(\Om_{1/10})}^2.
\ee
Furthermore, there holds
\be\no
&&\int_{L_0}^{\f 1{10} L_0} \lambda_m |a_m^{(3)}|^2 d x_1=\int_{L_0}^{\f 1{10} L_0} \frac{\la_m}{\bar{k}_{11}^2}(g_m'-2\bar{k}_{11}' a_m''-\bar{k}_{11}'' a_m'+\lambda_m a_m')^2dx_1\\\no
&&\leq C_*\int_{L_0}^{\f 1{10} L_0} (\la_m |g_m'|^2+ \lambda_m|a_m''|^2+ \lambda_m |a_m'|^2+ \lambda_m^3 |a_m'|^2) dx_1 \\\no
&&\leq C_*\int_{L_0}^{\f 1{10} L_0} \lambda_m |g_m'|^2 + (1+\lambda_m^2) |g_m|^2 dx_1,\\\no
&&\int_{L_0}^{\f 1{10} L_0} |a_m^{(4)}|^2 d x_1=\int_{L_0}^{\f 1{10} L_0} \frac{1}{\bar{k}_{11}^2}(g_m''-3\bar{k}_{11}' a_m^{(3)}-3\bar{k}_{11}'' a_m''-\bar{k}_{11}^{(3)} a_m'+\lambda_m a_m'')^2dx_1\\\no
&&\leq C_*\int_{L_0}^{\f 1{10} L_0} |g_m''|^2+|a_m^{(3)}|^2+|a_m''|^2+ |a_m'|^2+ \lambda_m^2|a_m''|^2) dx_1 \\\no
&&\leq C_*\int_{L_0}^{\f 1{10} L_0}\sum_{i=0}^2 \lambda_m^{2-i} |g_m^{(i)}|^2 dx_1
\ee
and thus
\be\label{an4}
&&\int_{L_0}^{\f 1{10} L_0} \sum_{i=0}^3 \lambda_m^{3-i} |a_m^{(i)}|^2 d x_1\leq C_*\int_{L_0}^{\f 1{10} L_0}(|g_m'|^2+\lambda_m |g_m|^2) dx_1,\\\label{an5}
&&\int_{L_0}^{\f 1{10} L_0} \sum_{i=0}^4 \lambda_m^{4-i} |a_m^{(i)}|^2 d x_1\leq C_*\int_{L_0}^{\f 1{10} L_0}\sum_{i=0}^2 \lambda_m^{2-i} |g_m^{(i)}|^2 dx_1.
\ee
Similar to \eqref{an26}, one has
\be\label{an41}
&&\|\nabla_{x'}^i\p_1 v_N\|_{L^2(\Om_{1/10})}\leq C_*\|\sum_{m=1}^N \lambda_m a_m'(x_1) b_m(x')\|_{H^{i-2}(\Om_{1/10})},\ \ \ i=2,3,\\\label{an43}
&&\|\nabla_{x'}^2\p_1^2v_N\|_{L^2(\Om_{1/10})}\leq C_*\|\sum_{m=1}^N \lambda_m a_m''(x_1) b_m(x')\|_{L^2(\Om_{1/10})}.
\ee
It follows from the estimates \eqref{an26} (for $i=3$), \eqref{an4} and \eqref{an41} (for $i=2$) that
\be\label{an419}
\|v_N\|_{H^3(\Omega_{1/10})}\leq C_*\|\sum_{m=0}^N g_m b_m\|_{H^1(\Om_{1/10})}.
\ee
Due to the linearity of the problem \eqref{li5}, the arguments for \eqref{an419} imply that
\be\no
\|v_{N_1}- v_{N_2}\|_{H^3(\Omega_{1/10})}\leq C_*\|\sum_{m=N_1+1}^{N_2} g_m b_m\|_{H^1(\Om_{1/10})},\ \ \forall N_2>N_1,
\ee
which implies that the sequence $\{v_N\}_{N\geq 1}$ converges in $H^3(\Om_{1/10})$ to $v$ with the estimate \eqref{li51} for $\varsigma=1$. The estimate \eqref{li51} with $\varsigma=2$ follows from \eqref{an26} (for $i=3$), \eqref{an5}, \eqref{an41} (for $i=3$) and \eqref{an43}.

\end{proof}

Now we use Lemma \ref{li50} and the Banach contraction mapping theorem to improve the regularity of the $H^2(\Om)$ strong solution to \eqref{li1}.
\begin{lemma}\label{H3}
Under the assumption in Lemma \ref{H1e}, the $H^2(\Om)$ strong solution to \eqref{li1} satisfies
  \be\label{H31}
  \int_{L_0}^{\f {L_0}5} \iint_{\mb{E}} (|\na^3 \psi |^2 + |\na^4 \psi |^2) dx' dx_1 \le C_* \|F\|_{H^2(\Om)}^2,
  \ee
 where the positive constant $C_*$ depends only on the $H^3(\Om)$ norms of $k_{ij}, i,j=1,2,3$, $\bar{k}_1$.
\end{lemma}
\begin{proof}

Let $\eta$ be a smooth cut-off function such that $0 \le \eta (x) \le 1$ on $[L_0, L_1]$ and
\be\no
\eta(x_1) = \begin{cases}
                 1, & L_0 \le x_1 \le \f {3}{20} L_0, \\
                 0, & \f 1{10} L_0 \le x_1 \le L_1.
               \end{cases}
\ee
Set $\psi_c= \eta \psi$. Then $\psi_c$ solves
\be\label{li6}\begin{cases}
\displaystyle\sum_{i,j=1}^3 k_{ij} \p_{ij}^2 \psi_c=\eta (F-\bar{k}_1 \p_1\psi)+ 2\eta' \sum_{i=1}^3 k_{1i}\p_i\psi + k_{11}\eta''\psi,\ \text{in }\ \Om_{1/10},\\
\psi_c(L_0,x')=\psi_c(\f1{10} L_0,x')=0, \ \q \forall x'\in \mb{E},\\
(n_2\p_2 + n_3 \p_3) \psi_c(x_1,x')=0,\ \q \text{on }[L_0,\frac{1}{10}L_0]\times \p\mb{E}.
\end{cases}\ee

The problem \eqref{li6} can be rewritten as
\be\label{li61}\begin{cases}
\p_1(\bar{k}_{11}\p_1\psi_c) + \p_2^2 \psi_c + \p_3^2\psi_c =J_1(\psi)+ J_2(\psi_c),\ \ \text{in }\ \Om_{1/10},\\
\psi_c(L_0,x')=\psi_c(\f1{10} L_0,x')=0, \ \ \ \forall x'\in \mb{E},\\
(n_2\p_2 + n_3 \p_3) \psi_c(x_1,x')=0,\ \ \text{on }[L_0,\frac{1}{10}L_0]\times \p\mb{E},
\end{cases}\ee
where
\be\no
&&J_1(\psi)= \eta(F+ (\bar{k}_{11}'-\bar{k}_{1})\p_1\psi)+\eta'\bigg(2\sum_{i=1}^3 k_{1i}\p_i \psi + \bar{k}_{11}'\psi\bigg)+\eta'' k_{11}\psi,\\\no
&&J_2(\psi_c)=-(k_{11}-\bar{k}_{11})\p_1^2 \psi_c- 2\sum_{i=2}^3 k_{1i}\p_{1i}^2 \psi_c- \sum_{i,j=2}^3 (k_{ij}-\delta_{ij}) \p_{ij}^2 \psi_c.
\ee
Thanks to \eqref{H1}, $J_1(\psi)\in H^1(\Omega_{1/10})$ and
\be\no
\|J_1(\psi)\|_{H^1(\Omega_{1/10})}\leq C_*(\|F\|_{H^1(\Om)}+\|\psi\|_{H^2(\Omega)})\leq C_*\|F\|_{H^1(\Omega)}.
\ee

Define
\be\no
\mc{X}_3=\bigg\{v\in H^3(\Om_{1/10}): v(L_0,x')=v(\frac{1}{10}L_0,x')=0 \ \text{on }\mb{E}; \\\no
\q\q\q\q\q \q (n_2\p_2+n_3\p_3) v(x_1,x')=0 \ \text{on }[L_0,\frac1{10} L_0]\times \p\mb{E}\bigg\}.
\ee
For any given $\tilde{v}\in\mc{X}_3$, consider the boundary value problem
\be\label{li62}\begin{cases}
\p_1(\bar{k}_{11}\p_1 v) + \p_2^2 v + \p_3^2 v = J_1(\psi)+ J_2(\tilde{v}),\ \ &\text{in }\ \Om_{1/10},\\
v(L_0,x')=v(\f1{10} L_0,x')=0, \ \ \ &\forall x'\in \mb{E},\\
(n_2\p_2 + n_3 \p_3) v(x_1,x')=0,\ \ &\text{on }[L_0,\frac{1}{10}L_0]\times \p\mb{E}.
\end{cases}\ee

To obtain $H^3$ estimate for $v$, one needs to show that both $J_1(\psi)$ and $J_2(\ti{v})$ belong to $H_{cp}^1(\Om_{1/10})$. Note that one needs only to consider the case that $\psi,\ti{v}\in C^{\infty}(\ol{\Om_{1/10}})$, since the general case follows by a density argument. Thus it suffices to verify \eqref{gcp} at $x_1=L_0,\frac{1}{10} L_0$ for both $J_1(\psi)$ and $J_2(\ti{v})$. Here we only deal with the case $x_1=L_0$, the other case $x_1=\frac{1}{10} L_0$ can be verified similarly. Notes that
\be\no
&&J_1(\psi)(L_0,x')= F(\nabla \psi)(L_0,x')+ (\bar{k}_{11}'-\bar{k}_1)\p_1\psi(L_0,x'),\\\no
&&J_2(\ti{v})(L_0,x')=-\bigg((k_{11}(\nabla(\psi+\epsilon\psi_0))-\bar{k}_{11})\p_1^2\ti{v}- 2\sum_{i=2}^3 k_{1i}(\nabla(\psi+\epsilon\psi_0))\p_{1i}^2 \ti{v}\bigg)(L_0,x').
\ee
Recall that $\p_{{\bf n}}=\sum_{i=2}^3 n_i \p_i$, then on $\p\mb{E}$ there holds
\be\no
&&\p_{{\bf n}}\{J_1(\psi)(L_0,x')\}=\p_{{\bf n}}\{F(\nabla\psi)\}(L_0,x')\\\no
&&=-\frac{\bar{u}'}{c^2(\bar{\rho})}\bigg((\gamma-1)\epsilon \p_{{\bf n}} \Phi_0-\f{\gamma+1}{2}\p_{{\bf n}}((\p_1\psi)^2)-\frac{(\gamma-)\epsilon^2}{2}\p_{{\bf n}}(|\nabla'h_0|^2)\bigg)(L_0,x')\\\no
&&\q-\frac{\epsilon}{c^2(\bar{\rho})}\bigg(\p_{{\bf n}}((\bar{u}+\p_1\psi)\p_1\Phi_0)+ \epsilon\sum_{i=2}^3 \p_{{\bf n}}(\p_i\Phi_0 \p_i h_0)\bigg)(L_0,x')=0,
\ee
and
\be\no
&&\p_{{\bf n}}\{J_2(\ti{v})(L_0,x')\}\\\no
&&=-\bigg\{(k_{11}(\nabla(\psi+\epsilon\psi_0))-\bar{k}_{11})\p_1^2\p_{{\bf n}}\ti{v}+\p_{{\bf n}}\{k_{11}(\nabla(\psi+\epsilon\psi_0))\}\p_1^2 \ti{v}\bigg\}(L_0,x')\\\no
&&\q\q-\frac{2\epsilon}{c^2(\bar{\rho})} \sum_{j=2}^3\bigg(\p_1\p_{{\bf n}}(\psi+\epsilon \psi_0) \p_j h_0 \p_{1j}^2 \ti{v}+(\bar{u}+\p_1\psi)(\p_{{\bf n}}\p_j h_0\p_{1j}^2 \ti{v}\\\no
&&\q\q\q\q+\p_j h_0 \p_{{\bf n}} \p_{1j}^2 \ti{v})\bigg)(L_0,x')=0.
\ee

It follows from Lemma \ref{li50} for $\varsigma=1$ that there exists a unique solution $v\in \mc{X}_3$ to \eqref{li62} such that
\be\label{li621}
\|v\|_{H^3(\Om_{1/10})}&\leq& C(L_0,\mb{E})(\|J_1(\psi)\|_{H^1(\Om_{1/10})}+\|J_2(\tilde{v})\|_{H^1(\Om_{1/10})})\\\no
&\leq& C_*(\|F\|_{H^1(\Om)}+(\epsilon+\delta_0)\|\tilde{v}\|_{H^3(\Om_{1/10})}).
\ee
Thus one can define an operator $\mc{T}$ mapping $\mc{X}_3$ to itself as $\mc{T}(\tilde{v})= v$. Note that the mapping $\mc{T}$ is a contraction. Indeed, for any $\tilde{v}_m\in \mc{X}_3, m=1,2$, denote $v_m=\mc{T}(\tilde{v}_m)$ and $V=v_1-v_2, \tilde{V}=\tilde{v}_1-\tilde{v}_2$. Then
\be\no\begin{cases}
\p_1(\bar{k}_{11}\p_1 V) + \p_2^2 V + \p_3^2 V = J_2(\tilde{V}),\ \ &\text{in }\ \Om_{1/10},\\
V(L_0,x')=V(\f1{10} L_0,x')=0, \ \ \ &\forall x'\in \mb{E},\\
(n_2\p_2 + n_3 \p_3) V(x_1,x')=0,\ \ &\text{on }[L_0,\frac{1}{10}L_0]\times \p\mb{E}.
\end{cases}\ee
Since $J_2(\ti{v})$ belongs to $H^1_{cp}(\Om_{1/10})$ as for \eqref{li621}, so Lemma \ref{li50} yields
\be\no
&&\|V\|_{H^3(\Om_{1/10})}\\\no
&&\leq C(L_0,\mb{E})\bigg(\|k_{11}-\bar{k}_{11}\|_{H^2(\Om)}+2\sum_{i=2}^3\|k_{1i}\|_{H^2(\Om)}+\sum_{i,j=2}^3 \|k_{ij}-\delta_{ij}\|_{H^2(\Om)}\bigg)\|\tilde{V}\|_{H^3(\Omega_{1/10})}\\\no
&&\leq C(L_0,\mb{E})(\epsilon + \delta_0)\|\tilde{V}\|_{H^3(\Omega_{1/10})}\leq \f12 \|\tilde{V}\|_{H^3(\Omega_{1/10})}
\ee
provided $\epsilon+\delta_0$ is small enough. By the Banach contraction mapping theorem, there exists a unique fixed point $\bar{v}\in \mc{X}_3$ such that $\mc{T}(\bar{v})=\bar{v}$. It follows from \eqref{li621} that
\be\no
\|\bar{v}\|_{H^3(\Om_{1/10})}\leq C_*\|F\|_{H^1(\Om)}
\ee
provided $\epsilon+\delta_0$ is small enough.

The definition of $\mc{T}$ implies that $\bar{v}$ solves \eqref{li61}. By the uniqueness of the solution to \eqref{li61}, one has $\psi_c=\bar{v}\in H^3(\Om_{1/10})$. Since $\psi_c=\eta \psi$ and $\Om_{3/20}\subset \Om_{1/10} (L_0<0)$, then
\be\label{li64}
\|\psi\|_{H^3(\Om_{3/20})}\leq \|\psi_c\|_{H^3(\Om_{1/10})}=\|v\|_{H^3(\Om_{1/10})}\leq C_*\|F\|_{H^1(\Om)},
\ee
where $\Om_{3/20}:=(L_0,\frac{3}{20}L_0)\times \mb{E}$. Thanks to \eqref{li64}, $J_1(\psi)\in H^2((L_0,\frac{3}{20}L_0)\times \mb{E})$ and
\be\no
\|J_1(\psi)\|_{H^2((L_0,\frac{3}{20}L_0)\times \mb{E})}\leq C_*(\|F\|_{H^2(\Om)}+\|\psi\|_{H^3((L_0,\frac{3}{20}L_0)\times \mb{E})})\leq C_*\|F\|_{H^2(\Om)}.
\ee

Now, define
\be\no
\mc{X}_4=\bigg\{v\in H^4(\Om_{3/20}): v(L_0,x')=v(\frac{3}{20}L_0,x')=0 \ \text{on }\mb{E}; \\\no
\q\q\q\q\q\q (n_2\p_2+n_3\p_3) v(x_1,x')=0 \ \text{on }[L_0,\frac3{20} L_0]\times \p\mb{E}\bigg\}.
\ee
Similar to \eqref{li62}, for any given $\ti{v}\in \mc{X}_4$, consider the following problem in $\Om_{3/20}$:
\be\label{li66}\begin{cases}
\p_1(\bar{k}_{11}\p_1 v) + \p_2^2 v + \p_3^2 v = J_1(\psi)+ J_2(\tilde{v}),\ \ &\text{in }\ \Om_{3/20},\\
v(L_0,x')=v(\f3{20} L_0,x')=0, \ \ \ &\forall x'\in \mb{E},\\
(n_2\p_2 + n_3 \p_3) v(x_1,x')=0,\ \ &\text{on }[L_0,\frac{1}{10}L_0]\times \p\mb{E}.
\end{cases}\ee

As for \eqref{li62}, one can verify that $J_1(\psi)+ J_2(\tilde{v})$ belongs to $H_{cp}^2(\Om_{3/20})$. By Lemma \ref{li50}, there exists a unique solution $v\in H^4(\Om_{3/20})$ to the problem \eqref{li66}, which defines a mapping $\mc{T}$ from $\mc{X}_4$ to itself. The fixed point $\bar{v}\in \mc{X}_4$ of $\mc{T}$ coincides with $\psi_c$ on $\Om_{3/20}$. Finally, it is clear that $\psi\in H^4((L_0,\frac{1}{5}L_0)\times \mb{E})$ with the estimate
\be\no
\|\psi\|_{H^4((L_0,\frac{1}{5}L_0)\times \mb{E})}\leq C_*\|F\|_{H^2(\Omega)}.
\ee
\end{proof}

\subsection{The $H^4$ estimate in transonic regions}
Now we are in a position to improve the regularity of $\psi$ on the transonic region $[\frac15 L_0, L_1]\times \mb{E}$. To this end, one can investigate further the singular perturbation problem \eqref{au2} for $\Psi^{\sigma}$ on $\Om_2$. The following lemma gives an $L^2$ estimate for $\na \p_{1}^2 \Psi^{\sigma}$ on the subregion $(\f7{10} L_0, L_1 + 14\ell) \times \mb{E}$, uniformly in $\sigma$.
\begin{lemma}\label{Ha3}
Under the assumptions of Lemma \ref{aH2}, the classical solution to \eqref{au2} satisfies
  \be\label{Ha31}
  \int_{\f 7{10} L_0}^{L_1 + 14\ell} \iint_{\mb{E}} \si |\p_{1}^4 \Psi^{\si}|^2 + |\na \p_{1}^2 \Psi^{\si}|^2 dx' dx_1  \le C_{\sharp} \| G\|_{H^2(\Om_2)}^2,
  \ee
  where the positive constant $C_\sharp$ depends only on the $C^3(\overline{\Om_2})$ norms of $a_{ij}$, $\bar{a}_1$.
\end{lemma}
\begin{proof}
  Define smooth cut-off functions $0 \le \eta_m (x) \le 1$ on $[L_0, L_2]$ for $m= 5, 6$ satisfying
  \be\no
  \eta_5 (x_1) = \begin{cases}
                 0, & L_0 \le x_1 \le \f 45 L_0, \\
                 1, & \f 3{4} L_0 \le x_1 \le \f 7{10} L_0, \\
                 0, & \f {13}{20} L_0 \le x_1 \le L_2,
               \end{cases}
  \quad
  \eta_6 (x_1) = \begin{cases}
                 0, & L_0 \le x_1 \le L_1 + 13 \ell, \\
                 1, & L_1 + 14 \ell \le x_1 \le L_1 + 15 \ell, \\
                 0, & L_1 + 16 \ell \le x_1 \le L_2.
               \end{cases}
  \ee

Multiplying \eqref{W1} by $\eta_m^2 \p_1^2 W_1$ for $m=5,6$ respectively and integrating over $\Om_2$ yield
\be\no
&&\iiint_{\Om_2}\eta_m^2 a_{11}(\p_1^2 W_1)^2 +\eta_m^2 \sum_{i,j=2}^3 a_{ij}\p_{1i}^2 W_1 \p_{1j}^2 W_1  dx' dx_1\\\no
&&=\iiint_{\Om_2} \bigg\{\eta_m^2 \p_1^2 W_1 (\p_1 G-\bar{a}_1'W_1-(\bar{a}_1+\p_1 a_{11})\p_1 W_1-2\sum_{i=2}^3 \p_1a_{1i}\p_i W_1)\\\no
&&\q\q-2\eta_m^2 \p_1^2 W_1 \sum_{i=2}^3 a_{1i}\p_{1i}^2 W_1- \eta_m^2 \sum_{i,j=2}^3 \bigg(2\p_1 a_{ij}\p_{1i}^2 W_1 -\p_i a_{ij}\p_1^2 W_1\bigg)\p_j W_1\\\no
&&\quad \q- \eta_m^2\sum_{i,j=2}^3 \bigg(\p_1^2 a_{ij}\p_j\Psi \p_{1i}^2 W_1-\p_{1i}^2 a_{ij}\p_j\Psi \p_1^2 W_1\bigg) \\\no
&&\quad\q- 2\eta_m\eta_m' \sum_{i,j=2}^3 (a_{ij}\p_j W_1 -\p_1 a_{ij}\p_j\Psi)\p_{1i}^2 W_1 +\si \eta_m\eta_m' (\p_1^2 W_1)^2\bigg\} dx' dx_1,
\ee
where integration by parts and Lemma \ref{coe2} have been used. Then there holds
\be\no
&&\iiint_{\Om_2}\eta_m^2 |\nabla \p_1 W_1|^2 dx' dx_1\\\no
&&\leq C_*\iiint_{\Om_2}\bigg\{\eta_m^2(|\p_1 G|^2+W_1^2+|\p_1 W_1|^2+\sum_{i=1}^3|\p_1 a_{1i}|^2 |\p_i W_1|^2)\\\no
&&\q+\si |\eta_m'|^2 |\p_1^3\Psi|^2 + \eta_m^2 \sum_{i=2}^3 |a_{1i}||\p_1^2 W_1||\p_{1i}^2 W_1|+ \eta_m^2\sum_{i,j=2}^3|\nabla a_{ij}|^2|\p_j W_1|^2\\\no
&&\q+ \sum_{i,j=2}^3 \eta_m^2|\p_1\nabla a_{ij}|^2|\p_j \Psi|^2 + |\eta_m'|^2 (|a_{ij}|^2|\p_j W_1|^2+ |\p_1 a_{ij}|^2 |\p_j\Psi|^2)\bigg\}\\\label{Ha32}
&&\leq C_{\sharp} \iiint_{\Om_2} (G^2 + |\na G|^2) dx' dx_1,
\ee
where one has used Lemmas \ref{coe2} and \ref{aH2}, and the facts that $a_{11} \ge \ka_1 >0$ on the supports of $\eta_m$ for $m = 5, 6$, while the supports of $\eta_m' (x_1)$ are contained in $[\f 45 L_0, L_1 + 16\ell]$ for $m = 5,6$.

Set $W_2=\p_1 W_1$. Then it follows from \eqref{W1} that $W_2$ solves
\be\label{W2}\begin{cases}
\sigma \p_1^3 W_2 + \sum_{i,j=1}^3 a_{ij} \p_{ij}^2 W_2 + (\bar{a}_1+2\p_1 a_{11})\p_1 W_2\\\q \q+\sum_{i,j=2}^3\bigg(2\p_1 a_{ij}\p_{ij}^2 W_1 + \p_1^2 a_{ij}\p_{ij}^2 \Psi\bigg)+ 4 \sum_{i=2}^3 \p_1 a_{1i}\p_i W_2=G_2,\\
W_2(L_0,x')=0,\ \ \forall x'\in\mb{E},\\
(n_2\p_2+n_3\p_3) W_2(x_1,x')=0,\ \ \ \forall (x_1,x')\in [L_0,L_2]\times \p\mb{E},
\end{cases}\ee
where
\be\no
G_2:=\p_1^2 G-(2\bar{a}_1'+\p_1^2 a_{11}) W_2 -2\sum_{i=2}^3 \p_1^2 a_{1i} \p_i W_1 -\bar{a}_1'' W_1.
\ee

Define a smooth cut-off function $0 \le \eta_7 (x_1) \le 1$ on $[L_0, L_2]$ satisfying
\be\no
\eta_7 (x_1) = \begin{cases}
               0, & L_0 \le x_1 \le \f 3{4} L_0, \\
               1, & \f {7}{10} L_0 \le x_1 \le L_1 + 14 \ell, \\
               0, & L_1 + 15 \ell \le x_1 \le L_2.
             \end{cases}
\ee
Multiplying the equation in $\eqref{W2}$ by $\eta_7^2 d(x_1) \p_{1} W_2$, integrating the resulting identity over $\Om_2$ and integration by parts yield
\be\no
&&\iiint_{\Om_2}\bigg\{ -\si \eta_7^2 d(\p_1^2 W_2)^2-\si (\eta_7^2 d)'\p_1 W_2\p_1^2 W_2 + \frac{1}{2}(\eta_7^2 d)'\sum_{i,j=2}^3 a_{ij}\p_i W_2\p_j W_2\\\no
&&\q+\bigg(\eta_7^2 d(\bar{a}_1+2\p_1 a_{11})-\frac12\p_1(\eta_7^2 d a_{11})-\eta_7^2 d \sum_{i=2}^3 \p_i a_{1i}\bigg)(\p_1 W_2)^2\\\label{h311}
&&\q+\eta_7^2 d\sum_{i,j=2}^3 \bigg(\frac{5}{2}\p_1 a_{ij}\p_j W_2+ 3 \p_1^2 a_{ij}\p_j W_1 + \p_1^3 a_{ij}\p_j\Psi\bigg)\p_i W_2\\\no
&&\q- \eta_7^2 d \p_1 W_2\bigg(\sum_{i,j=2}^3 \big(\p_i a_{ij}\p_j W_2+2\p_{1i}^2 a_{ij}\p_j W_1 +\p_{1}^2\p_i a_{ij}\p_j \Psi\big) +\sum_{i=2}^3 4\p_1 a_{1i} \p_i W_2\bigg)\\\no
&&\q+ (\eta_7^2 d)' \sum_{i,j=2}^3 (2\p_1 a_{ij}\p_j W_1  + \p_1^2 a_{ij}\p_j \Psi)\p_i W_2\bigg\}=\iiint_{\Om_2} \eta_7^2 d \p_1 W_2 G_2 dx' dx_1,
\ee
where \eqref{coe203} for $m=0,1,2$ has been used. Note that
\be\no
\eta_7^2 d [(\bar{a}_1+2\p_{1} a_{11}) -\sum_{i=2}^3\p_i a_{1i}]- \f 12 \p_{1} (\eta^2_7 d a_{11})\geq 3 \eta_7^2 - \eta_7 \eta_7' a_{11} d,\forall (x_1, x') \in \Om_2.
\ee

Then one may conclude from \eqref{h311} that
\be\no
&&\iiint_{\Om_2}\eta_7^2 (\si (\p_1^2 W_2)^2+ 3|\nabla W_2|^2)\leq\iiint_{\Om_2}|\eta_7'|^2(\si |\p_1 W_2|^2 + |\nabla W_2|^2)\\\no
&&+\iiint_{\Om_2}\bigg\{\eta_7^2(|\p_1^2 G|^2+ \sum_{i=2}^3 |\p_1^2 a_{1i}|^2|\p_i W_1|^2+ W_1^2+ W_2^2)+\eta_7^2 |\nabla W_2|^2\sum_{i=2}^3 |\p_1 a_{1i}|\\\no
&&\q+\eta_7^2 \sum_{i,j=2}^3|\nabla a_{ij}|^2|\p_jW_2|^2+|\p_1\nabla a_{ij}|^2|\p_j W_1|^2+|\p_1^2\nabla a_{ij}|^2|\p_j \Psi|^2 \\\label{Ha33}
&&\q+(\eta_7^2+|\eta_7'|^2)\sum_{i,j=2}^3 (|\p_1 a_{ij}|^2 |\p_j W_1|^2 + |\p_1^2 a_{ij}|^2 |\p_j\Psi|^2)+\eta_7^2|\na W_2|^2\bigg\}.
\ee
Since the support of $\eta_7' (x_1)$ is contained in $[\f 34 L_0, \f 7{10} L_0] \cup [L_1 + 14 \ell, L_1 + 15 \ell]$, we use \eqref{aH21} and \eqref{Ha32} to control the first term on the right hand side of \eqref{Ha33}. Then \eqref{Ha31} follows from \eqref{aH1},\eqref{aH21} and \eqref{Ha33}.

\end{proof}

The following lemma gives the $L^2$ estimate for $\na \p_{1}^3 \Psi^{\sigma}$ on the subregion $(\f 35 L_0, L_1 + 12\ell) \times \mb{E}$.
\begin{lemma}\label{Ha4}
Under the assumptions of Lemma \ref{aH2}, the classical solution to \eqref{au2} satisfies
\be\label{Ha41}
\int_{\f 35 L_0}^{L_1 + 12 \ell} \iint_{\mb{E}} \si|\p_{1}^5 \Psi^{\si}|^2 +|\na \p_{1}^3 \Psi^{\si}|^2 dx' dx_1\leq C_{\sharp} \|G\|_{H^3(\Om_2)}^2,
\ee
where $C_\sharp>0$ depends only on the $C^4(\overline{\Om_2})$ norms of $a_{ij}, i,j=1,2,3$ and $\bar{a}_1$.

\end{lemma}
\begin{proof}
  Define smooth cut-off functions $0 \le \eta_m (x_1) \le 1$ on $[L_0, L_2]$ for $m = 8, 9$ such that
  \be\no
  \eta_8 (x_1) = \begin{cases}
                 0, & L_0 \le x_1 \le \f {7}{10} L_0, \\
                 1, & \f {13}{20} L_0 \le x_1 \le \f {3}{5} L_0, \\
                 0, & \f {11}{20} L_0 \le x_1 \le L_2,
               \end{cases}
  \q
  \eta_9 (x_1) = \begin{cases}
                 0, & L_0 \le x_1 \le L_1 + 11\ell, \\
                 1, & L_1 + 12 \ell \le x_1 \le L_1 + 13 \ell, \\
                 0, & L_1 + 14 \ell \le x_1 \le L_2.
               \end{cases}
  \ee
Multiplying the equation in \eqref{W2} by $\eta_m^2 \p_1^2 W_2$ for $m=8,9$ respectively, integrating over $\Om_2$ and integration by parts yield
\be\no
&&\iiint_{\Om_2}\eta_m^2 a_{11}(\p_1^2 W_2)^2 + \eta_m^2 \sum_{i,j=2}^3 a_{ij}\p_{1i}^2 W_2 \p_{1j}^2 W_2 dx' dx_1\\\no
&&=\iiint_{\Om_2}\bigg\{\eta_m^2 \p_1^2 W_2 \bigg(G_2-\sum_{i=2}^3 (2a_{1i}\p_{1i}^2 W_2+4\p_1 a_{1i} \p_i W_2)- (\bar{a}_1+2\p_1a_{11})\p_1 W_2\bigg)\\\no
&&\q-\eta_m^2 \sum_{i,j=2}^3 (3\p_1 a_{ij}\p_j W_2+ 3\p_1^2 a_{ij}\p_j W_1+ \p_1^3 a_{ij}\p_j\Psi)\p_{1i}^2 W_2 \\\no
&&\q+\eta_m^2\sum_{i,j=2}^3(\p_i a_{ij}\p_j W_2+ 2\p_{1i}^2 a_{ij}\p_j W_1+ \p_1^2 \p_i a_{ij}\p_j\Psi) \p_1^2 W_2\\\no
&&\q-\eta_m\eta_m'\bigg(2\sum_{i,j=2}^3 (a_{ij}\p_j W_2+ 2\p_1 a_{ij}\p_j W_1 + \p_1^2 a_{ij}\p_j \Psi)\p_{1i}^2 W_2-\si(\p_1^2 W_2)^2\bigg)\bigg\}.
\ee

On the support of $\eta_m$ for $m = 8, 9$, one has $a_{11} \ge \ka_1 >0$. Note also that the supports of $\eta_m' (x_1)$ are contained in $[\frac7{10} L_0, L_1 + 14\ell]$ for $m = 8, 9$, so the last term on the right hand side above can be controlled by \eqref{Ha31}. Then there holds from Lemmas \ref{exist} and \ref{Ha3} that
\be\no
&&\iiint_{\Om_2}\eta_m^2 |\nabla \p_1 W_2|^2 dx' dx_1\\\no
&&\leq C_*\iiint_{\Om_2}\bigg\{\eta_m^2(|\p_1^2 G|^2+W_1^2+W_2^2+\sum_{i=1}^3|\p_1^2 a_{1i}|^2 |\p_i W_1|^2)\\\no
&&\q+\eta_m^2\sum_{i=1}^3 |\p_1 a_{1i}|^2|\nabla W_2|^2+ \si |\eta_m'|^2 |\p_1^4\Psi|^2 \\\no
&&\q+\eta_m^2\sum_{i,j=2}^3(|\nabla a_{ij}|^2|\p_j W_2|^2+|\p_1^2 a_{ij}|^2|\p_j W_1|^2+|\p_1^3 a_{ij}|^2|\p_j \Psi|^2)\\\no
&&\q +|\eta_m'|^2 \sum_{i,j=2}^3(|a_{ij}|^2|\p_j W_2|^2+ |\p_1 a_{ij}|^2 |\p_j W_1|^2+ |\p_1^2 a_{ij}|^2 |\p_j\Psi|^2)\bigg\}dx' dx_1\\\label{Ha42}
&&\leq C_{\sharp} \iiint_{\Om_2} (G^2 + |\na G|^2+|\nabla^2 G|) dx' dx_1,
\ee
where $C_{\sharp}$ depends only on the $C^3(\overline{\Om_2})$ norms of $a_{ij}$ and $\bar{a}_1$.

Set $W_3=\p_1 W_2$. Then $W_3$ solves
\be\label{W3}\begin{cases}
\si \p_1^3 W_3 + \sum_{i,j=1}^3 a_{ij}\p_{ij}^2 W_3 + (\bar{a}_1+3\p_1 a_{11})\p_1 W_3 + 6 \sum_{i=2}^3 \p_1a_{1i}\p_i W_3\\
\q\q \q+\sum_{i,j=2}^3 (3\p_1 a_{ij}\p_{ij}^2 W_2+3\p_1^2 a_{ij}\p_{ij}^2 W_1+\p_1^3 a_{ij}\p_{ij}^2 \Psi)=G_3, \text{in }\ \Om_2,\\
(n_2\p_2+n_3\p_3) W_3(x_1,x')=0,\ \ \ (x_1,x')\in [L_0,L_2]\times \p\mb{E},
\end{cases}\ee
where
\be\no
&&G_3:=\p_1^3 G-3(\bar{a}_1'+\p_1^2 a_{11})W_3-(3\bar{a}_{1}''+\p_1^3 a_{11})W_2\\\no
&&\q\q\q-\bar{a}_1^{(3)} W_1-\sum_{i=2}^3 (6\p_1^2 a_{1i}\p_i W_2+ 2\p_1^3 a_{1i}\p_i W_1).
\ee

Define a smooth cut-off function $0 \le \eta_{10} (x_1) \leq 1$ on $[L_0, L_2]$ such that
\be\no
\eta_{10} (x_1) = \begin{cases}
                  0, & L_0 \le x_1  \le \f{13}{20} L_0, \\
                  1, & \f 35 L_0 \le x_1 \le L_1 + 12 \ell, \\
                  0, & L_1 + 13 \ell \le x_1 \le L_2.
                \end{cases}
\ee
Multiplying the equation in \eqref{W3} by $\eta_{10}^2 d \p_1 W_3$, integrating the resulting identity over $\Om_2$ and integration by parts yield
\be\no
&&\iiint_{\Om_2}\bigg(\eta_{10}^2 d(\bar{a}_1+3\p_1a_{11})-\frac{1}{2}\p_1(\eta_{10}^2 d a_{11})-\eta_{10}^2 d \sum_{i=2}^3 \p_i a_{1i}\bigg)(\p_1 W_3)^2 dx'dx_1 \\\no&&\q+\iiint_{\Om_2}\frac{1}{2}(\eta_{10}^2 d)'\sum_{i,j=2}^3 a_{ij}\p_j W_3\p_i W_3 -\si \eta_{10}^2 d (\p_1^2 W_3)^2 dx'dx_1\\\no
&&=\iiint_{\Om_2}\bigg\{\si (\eta_{10}^2 d)'\p_1 W_3\p_1^2 W_3 - 6\eta_{10}^2 d\p_1 W_3\sum_{i=2}^3 \p_1 a_{1i}\p_i W_3-\eta_{10}^2 d \p_1 W_3 G_3\\\no
&&\q-\eta_{10}^2 d \sum_{i,j=2}^3\bigg(\frac{7}{2}\p_1 a_{ij}\p_j W_3+6\p_1^2a_{ij}\p_j W_2+4\p_1^3 a_{ij}\p_j W_1+ \p_1^4 a_{ij}\p_j \Psi\bigg)\p_i W_3 \\\no
&&\q+\eta_{10}^2 d \sum_{i,j=2}^3(\p_i a_{ij}\p_j W_3+ 3\p_{1i}^2 a_{ij}\p_j W_2+3\p_1^2 \p_i a_{ij}\p_j W_1+\p_1^3\p_i a_{ij}\p_j \Psi)\p_1 W_3\\\label{4109}
&&\q-(\eta_{10}^2 d)'\sum_{i,j=2}^3 \bigg(3\p_1 a_{ij}\p_j W_2 + 3\p_1^2 a_{ij}\p_j W_1 + \p_1^3 a_{ij}\p_j \Psi \bigg)\p_i W_3\bigg\},\ee
where \eqref{coe203} for $m=0,\cdots,3$ has been used. By \eqref{8}-\eqref{9} and Lemma \ref{coe2}, it holds that
\be\no
 && \eta_{10}^2 d (\bar{a}_1+3\p_{1} a_{11}) - \f 12 \p_{1} (\eta^2_{10} d a_{11})- \eta_{10}^2 d(\p_2 a_{12}+\p_3 a_{13})\\\no
 &&
\ge  \eta_{10}^2 [4 - 3\|d(\p_{1} a_{11}-\bar{a}_{11}') \|_{L^{\oo}} - \f 12 \| \p_{1} (d a_{11} - d \bar{a}_{11}) \|_{L^{\oo}}- \| d \sum_{j=2}^3 \p_j a_{1j}\|_{L^{\oo}} ]
\\\no
&&- \eta_{10} \eta_{10}' d a_{11}\geq 3 \eta_{10}^2 - \eta_{10} \eta_{10}' a_{11} d, \q \forall (x_1, x') \in \Om_2.
\ee
Thus it follows from \eqref{4109} that
\be\no
&&\iiint_{\Om_2}\eta_{10}^2(\si |\p_1^2 W_3|^2 + |\nabla W_3|^2) dx' dx_1\\\label{Ha43}
&&\leq  \iiint_{\Om_2}\bigg\{ |\eta_{10}'|^2 (\si|\p_1W_3|^2+|\nabla W_3|^2)+\eta_{10}^2 \sum_{i,j=2}^3 |\na a_{ij}||\nabla W_3|^2\\\no
&&\q+ \eta_{10}^2\sum_{i,j=2}^3 (|\p_1\nabla a_{ij}\p_j W_2|^2 + |\p_1^2\nabla a_{ij}\p_j W_1|^2 + |\p_1^3\nabla a_{ij}\p_j \Psi|^2) \\\no
&&\q+ (\eta_{10}^2+|\eta_{10}'|^2)\sum_{i,j=2}^3 (|\p_1 a_{ij}\p_j W_2|^2 + |\p_1^2 a_{ij}\p_j W_1|^2 + |\p_1^3 a_{ij}\p_j \Psi|^2) \\\no
&&\q+ \eta_{10}^2\bigg(|\p_1^3 G|^2 +\sum_{i=1}^3|\p_1^2 a_{1i}\p_i W_2|^2+|\p_1^3 a_{1i}\p_i W_1|^2+|W_i|^2\bigg)\bigg\} dx' dx_1.
\ee
Since the support of $\eta_{10}'$ is contained in $[\f{13}{20} L_0, \f 35 L_0] \cup [L_1 + 12 \ell, L_1 + 13 \ell]$, one can use \eqref{Ha31} and \eqref{Ha42} to control the first integral on the right hand side of \eqref{Ha43}. Finally, the estimate \eqref{Ha41} follows from \eqref{aH1}, \eqref{aH21}, \eqref{Ha31} and \eqref{Ha43}.

\end{proof}

Now we can improve the regularity of the solution $\psi$ to \eqref{li1} to be $H^4(\Om)$.
\begin{lemma}\label{H4first}
The $H^2(\Om) $ strong solution $\psi$ to \eqref{li1} admits $H^4(\Om)$ regularity such that
\be\label{aH400}
\|\psi\|_{H^4(\Om)}  \le C_{\sharp} \| F\|_{H^3(\Om)},
\ee
where $C_\sharp>0$ depends only on the $C^4(\overline{D_2})$ norms of $a_{ij}, i,j=1,2,3$ and $\bar{a}_1$.

\end{lemma}
\begin{proof}

Note that the estimates \eqref{Ha31} and \eqref{Ha41} can be verified rigorously for the approximated solutions $\Psi^{N,\sigma}=\sum_{j=0}^N A_j^{N,\sigma}(x_1) b_j(x')$ constructed in Lemma \ref{exist1}. As shown in Lemma \ref{exist}, up to a subsequence, $\{\Psi^{N,\sigma}\}$ converges strongly in $L^2(\Omega_2)$ and weakly in $H^1(\Omega_2)$ to a limit $\Psi^{\sigma}\in H^1(\Omega_2)$, and the sequence $\p_1 \Psi^{N,\sigma}$ will converge strongly in $L^2(\Om_2)$ and weakly in $H^1(\Om_2)$ to $\p_1\Psi^{\si}$ as $N$ tends to infinity. Since the estimates \eqref{Ha31} and \eqref{Ha41} are uniformly in $N$, the subsequence $\nabla\p_1^2\Psi^{N,\sigma}$ and $\nabla\p_1^3\Psi^{N,\sigma}$ will converge weakly in $L^2((\frac{3}{5}L_0, L_1+12\ell)\times \mb{E})$ to $\nabla\p_1^2 \Psi^{\sigma}$ and $\nabla\p_1^3\Psi^{\sigma}$, respectively. Thus besides the estimates \eqref{Hs1} and \eqref{Hs2} in Lemma \ref{exist}, $\Psi^{\si}$ also satisfies the following uniform estimate with respect to $\sigma$:
\be\label{Ha401}
\int_{\frac{3}{5}L_0}^{L_1+12\ell}\iint_{\mb{E}} |\nabla\p_{1}^2\Psi^{\sigma}|^2+ |\nabla\p_{1}^3\Psi^{\sigma}|^2 dx' dx_1\leq C_{\sharp}\|G\|_{H^3(\Omega_2)}^2\leq C_{\sharp}\|F\|_{H^3(\Omega)}^2.
\ee
By \eqref{Ha401}, there exists a subsequence $\{\Psi^{\sigma_j}\}_{j=1}^{\infty}$ such that $\nabla\p_1^2\Psi^{\sigma_j}$ and $\nabla\p_1^3\Psi^{\sigma_j}$ will converge weakly in $L^2((\frac{3}{5}L_0, L_1+12\ell)\times \mb{E})$ to $\nabla\p_1^2 \Psi$ and $\nabla\p_1^3\Psi$, respectively. This, together with Lemma \ref{exist}, yields that
\be\label{Ha402}
\|\Psi\|_{H^2(\Omega_2)}^2+\int_{\frac{3}{5}L_0}^{L_1+12\ell}\iint_{\mb{E}}|\nabla\p_{1}^2\Psi|^2+ |\nabla\p_{1}^3\Psi|^2 dx' dx_1 \leq C_{\sharp}\|F\|_{H^3(\Omega)}^2.
\ee

As shown in Lemma \ref{exist}, $\Psi\in H^2(\Om_2)$ can be regarded as a strong solution to the following boundary value problem for a second order elliptic equation
\be\label{p11}\begin{cases}
\displaystyle-\p_1^2 \Psi -\sum_{i,j=2}^3 \p_i(a_{ij}\p_j\Psi)= \tilde{G}\in L^2(\Om_2),\ \ &\text{in }\Omega_2,\\
\Psi(L_0,x')=\p_1\Psi(L_2,x')=0,\ \ \ & \forall x'\in \mb{E},\\
\sum_{i,j=2}^3 a_{ij}\p_j\Psi n_i=0,\ \ \ &\text{on }[L_0,L_2]\times \p\mb{E},
\end{cases}\ee
where
\be\no
\tilde{G}=(a_{11}-1)\p_1^2\Psi+ 2\sum_{i=2}^3 a_{1i}\p_{1i}^2 \Psi- \sum_{i,j=2}^3 \p_i a_{ij}\p_j\Psi+ \bar{a}_1 \p_1\Psi-G.
\ee
By \eqref{Ha402}, $\p_1\tilde{G}$ and $\p_1^2\tilde{G}$ belong to $L^2((\frac{3}{5}L_0,L_1+12\ell)\times \mb{E})$. Choosing the test function $\xi=\p_1\zeta$ in \eqref{p2} with $\zeta\in C^2(\Omega_2)$ and $\zeta(x_1,x')\equiv 0$ on $[L_0,\frac{3}{5}L_0]\times \mb{E}\cup [L_1+12\ell, L_2]\times \mb{E}$, one can integrate by parts to show that $W_1=\p_1\Psi\in H^1(\Omega_2)$ is a weak solution to the following problem on $(\frac{3}{5}L_0,L_1+12\ell)\times \mb{E}$:
\be\label{p12}\begin{cases}
-\p_1^2 W_1 -\sum_{i,j=2}^3 \p_i (a_{ij}\p_j W_1)=G_1:=\p_1 \tilde{G} + \sum_{i,j=2}^3 \p_i(\p_1 a_{ij}\p_j\Psi),\\
\sum_{i,j=2}^3 a_{ij}\p_j W_1 n_i=0.
\end{cases}\ee

An argument using the difference quotient (for instance, see \cite[Theorem 8.8]{gt}) yields that the weak second order derivatives $\nabla_{x'}^2 W_1$ exist and satisfy
\be\label{p14}
&&\|\nabla_{x'}^2 W_1\|_{L^2((\frac{1}{2}L_0, L_1+11\ell)\times \mb{E})}\\\no
&&\leq C_*(\|W_1\|_{H^1(\Om_2)}+\|G_1\|_{L^2 (((\frac{3}{5}L_0,L_1+12\ell)\times \mb{E}))})\leq C_{\sharp}\|G\|_{H^3(\Omega_2)},
\ee
which implies that $\tilde{G}\in H^1((\frac{1}{2}L_0, L_1+11\ell)\times \mb{E})$. Thus the interior and boundary $H^3$ estimates up to the cylinder wall for the function $\Psi$ in \eqref{p11} lead to
\be\label{p15}
&&\|\Psi\|_{H^3((\frac{2}{5}L_0, L_1+10\ell)\times \mb{E})}\\\no
&&\leq C_{\sharp}(\|\Psi\|_{H^1(\Omega_2)}+\|\tilde{G}\|_{H^1((\frac{1}{2}L_0, L_1+11\ell)\times \mb{E})})\leq C_{\sharp}\|G\|_{H^3(\Omega_2)}.
\ee
It follows from \eqref{p14} and \eqref{p15} that $G_1\in H^1((\frac{2}{5}L_0, L_1+11\ell)\times \mb{E})$. Together with \eqref{p12}, one gets
\be\label{p16}
&&\|W_1\|_{H^3((\frac{3}{10}L_0, L_1+9\ell)\times \mb{E})}\\\no
&&\leq C_{\sharp}(\|W_1\|_{H^1(\Omega_2)}+\|G_1\|_{H^1((\frac{2}{5}L_0, L_1+10\ell)\times \mb{E})})\leq C_{\sharp}\|G\|_{H^3(\Om_2)}.
\ee
Similarly, $W_2=\p_1^2\Psi\in H^1((\frac{3}{5}L_0,L_1+12\ell)\times \mb{E})$ can be regarded as a weak solution to
\be\no\begin{cases}
-\p_1^2 W_2 -\sum_{i,j=2}^3 \p_i (a_{ij}\p_j W_2)=G_2:=\p_1^2 \tilde{G} + \sum_{i,j=2}^3 \p_i(2\p_1 a_{ij}\p_j W_1 +\p_1^2 a_{ij}\p_j\Psi),\\
\sum_{i,j=2}^3 a_{ij}\p_j W_2 n_i=0,\ \ \ \text{on }[\frac{3}{5}L_0,L_2]\times \p\mb{E}.
\end{cases}\ee
Thanks to \eqref{Ha402} and \eqref{p14}, $G_2\in L^2((\frac{1}{2}L_0, L_1+11\ell)\times \mb{E})$, and the second order weak derivatives $\nabla_{x'}^2 W_2$ exist and obey the estimate
\be\label{p17}
&&\|\nabla_{x'}^2 W_2\|_{L^2((\frac{2}{5}L_0, L_1+9\ell)\times \mb{E})}\leq C_*(\|W_2\|_{H^1((\frac{3}{5}L_0,L_1+12\ell)\times \mb{E})}\\\no
&&\q\q+\|\p_1^2 \tilde{G} + \sum\limits_{i,j=2}^3 \p_i(\p_1 a_{ij}\p_j\Psi)\|_{L^2 (((\frac{1}{2}L_0,L_1+11\ell)\times \mb{E}))})\leq C_{\sharp}\|G\|_{H^3(\Omega_2)}.
\ee
The estimates \eqref{p16} and \eqref{p17} imply that $\tilde{G}$ belongs to $H^2((\frac{2}{5}L_0, L_1+9\ell)\times \mb{E})$ and the interior and boundary $H^4$ estimates up to the cylinder wall for the function $\Psi$ in \eqref{p11} yield that
\be\label{p18}
\|\Psi\|_{H^4((\frac{1}{4}L_0, L_1+8\ell)\times \mb{E})}\leq C_{\sharp}(\|\Psi\|_{H^1(\Omega_2)}+\|\tilde{G}\|_{H^2((\frac{2}{5}L_0, L_1+9\ell)\times \mb{E})})\leq C_{\sharp}\|G\|_{H^3(\Omega_2)}.
\ee
Note that $\psi=\Psi|_{\Omega}$. Combining the estimates \eqref{H31} and \eqref{p18}, one gets $\psi\in H^4(\Omega)$ with the estimate
\be\no
\|\psi\|_{H^4(\Omega)}\leq C_{\sharp}\|G\|_{H^3(\Omega_2)}\leq C_{\sharp}\|F\|_{H^3(\Omega)}.
\ee

\end{proof}

Finally, we prove that the constant $C_{\sharp}$ in \eqref{aH400} can be replaced by a constant $C_*$ which depends only on the $H^3(\Om)$ norms of $k_{ij}$, $k_1$. One can start with the $H^2(\Omega)$ estimate.

\begin{lemma}\label{gloH2}
Under the assumptions in Lemma \ref{H1e}, the $H^2(\Om)$ strong solution to \eqref{li1} satisfies
  \be\label{f2}
  \| \psi\|_{H^2(\Om)}  \le C_* \|F\|_{H^1(\Om)},
  \ee
  where $C_*>0$ depends only on the $H^3(\Om)$ norms of $k_{ij}, i,j=1,2,3$ and $\bar{k}_1$.

\end{lemma}

\begin{proof}

Denote $w_1 = \p_{1} \psi$. Due to $\psi\in H^4(\Omega)$, then $w_1=\p_{1}\psi$ satisfies the following equation almost everywhere
\be\label{f1}\begin{cases}
\sum_{i,j=1}^3 k_{ij}\p_{ij}^2 w_1 + (\bar{k}_1+\p_1 k_{11})\p_{1}w_1+ 2\sum_{i=2}^3 \p_1 k_{1i} \p_i w_1\\
\q\q\q\q\q\q=\p_1 F-\bar{k}_1' \p_1\psi-\sum_{i,j=2}^3 \p_1 k_{ij} \p_{ij}^2 \psi,\\
(n_2\p_2+n_3\p_3)w_1(x_1,x')=0,\ \ \ \forall (x_1,x')\in \Ga_w.
\end{cases}\ee

Let $\eta$ be a monotone increasing smooth cutoff function on $[L_0,L_1]$ such that $0\leq \eta\leq 1$ and
\be\no
\eta(x_1)=\begin{cases}
0, \ \ &L_0\leq x_1\leq \frac{3L_0}{4},\\
1, \ \ &\frac{L_0}{2}\leq x_1\leq L_1.
\end{cases}\ee
Then $\tilde{w}_1=\eta w_1$ solves
\be\label{f00}\begin{cases}
\sum_{i,j=1}^3 k_{ij}\p_{ij}^2 \tilde{w}_1 +(\bar{k}_1+\p_1 k_{11})\p_{1}\tilde{w}_1+2\sum_{i=2}^3 \p_1 k_{1i} \p_{i}\tilde{w}_1=F_1,\ \text{in }\ \ \Omega,\\
\tilde{w}_1(L_0,x')=0,\ \ \text{on }\ \ x'\in \mb{E},\\
(n_2\p_2+n_3\p_3)\tilde{w}_1(x_1,x')=0,\ \forall (x_1,x')\in \Ga_w,
\end{cases}\ee
where
\be\no
&&F_1=\eta(\p_1 F-\bar{k}_1' \p_1\psi-\sum_{i,j=2}^3 \p_1 k_{ij} \p_{ij}^2 \psi)+ 2 \eta'\sum_{i=1}^3 k_{1i}\p_i w_1\\\no
&&\q\q\q +\eta' (\bar{k}_1+\p_1 k_{11}) w_1 + \eta'' k_{11} w_1.
\ee

Note that if $0<\delta_0\leq \delta_*$ in Lemma \ref{coe1}, then there holds for any $(x_1,x¡¯)\in\Omega$
\be\no
&& 2(\bar{k}_1+\p_1 k_{11})-\p_{1} k_{11}\leq 2\bar{k}_1+\bar{k}_{11}'+\|\p_{1} k_{11}-\bar{k}_{11}'\|_{L^{\infty}}\leq -\kappa_*<0,\\\no
&& 2(\bar{k}_1+\p_1 k_{11})+\p_{1} k_{11}\leq 2\bar{k}_1+3\bar{k}_{11}'+3\|\p_{1} k_{11}-\bar{k}_{11}'\|_{L^{\infty}}\leq -\kappa_*<0.
\ee
Then as shown in Lemma \ref{exist2}, there exists a unique strong solution $v_1\in H^2(\Omega)$ to \eqref{f00} with
\be\no
\|v_1\|_{H^1(\Omega)}\leq C_*\|F_1\|_{L^2(\Omega)}.
\ee
By the uniqueness, $v_1=\tilde{w}_1$ holds a.e. in $\Omega$. Thus
\be\no
&&\bigg(\int_{\frac{L_0}{2}}^{L_1}\iint_{\mb{E}} |\nabla w_1|^2 dx' dx_1\bigg)^{\frac{1}{2}}\leq \|\tilde{w}_1\|_{H^1(\Omega)}\leq C_*\|F_1\|_{L^2(\Omega)}\\\no
&&\leq C_*\bigg(\|\p_1 F\|_{L^2(\Om)}+\|\bar{k}_1'\|_{L^{\infty}}\|\p_{1}\psi\|_{L^2(\Omega)}+\sum_{i,j=2}^3\|\p_1 k_{ij}\|_{L^{\infty}(\Omega)}\|\p_{ij}^2 \psi\|_{L^2(\Omega)}\\\no
&&\q\q\q\q+\sum_{i=1}^3\|k_{1i}\|_{L^{\infty}(\Omega)}\|\eta'\p_i w_1\|_{L^2(\Om)}\\\no
&&\q\q+\|\bar{k}_1+\p_1 k_{11}\|_{L^{\infty}(\Om)}\|\eta' w_1\|_{L^2(\Om)}+\|k_{11}\|_{L^{\infty}(\Om)}\|\eta'' w_1\|_{L^2(\Omega)}\bigg)\\\no
&&\leq C_*(\|F\|_{H^1(\Omega)}+ (\epsilon+\delta_0)\|\psi\|_{H^2(\Omega)}),
\ee
where one has used \eqref{H1}. This, together with \eqref{H1}, yields
\be\label{f142}
\|\psi\|_{H^1(\Omega)}+ \|\nabla \p_1\psi\|_{L^2(\Omega)}\leq C_*(\|F\|_{H^1(\Omega)}+ (\epsilon+\delta_0)\|\psi\|_{H^2(\Omega)}).
\ee

For each fixed $x_1\in [L_0, L_1]$, it follows from \eqref{li1} that
\be\label{f15}\begin{cases}
\sum_{i=2}^3 \p_i^2\psi(x_1,x')=H(x_1,x'),\ \ \ &\text{in }\mb{E} \\
(n_2\p_2+n_3\p_3)\psi(x_1,x')=0,\ \ \ &\forall x'\in \p\mb{E},
\end{cases}\ee
where
\be\no
H(x_1,x'):= F- k_{11}\p_1^2\psi- 2\sum_{i=2}^3 k_{1i}\p_{1i}^2\psi-\bar{k}_1\p_1\psi- \sum_{i,j=2}^3(k_{ij}-\delta_{ij})\p_{ij}^2\psi.
\ee

Then the $H^2$ estimate for the Poisson equation with the homogeneous Neumann boundary condition implies that
\be\label{f16}
\|\nabla_{x'}^2 \psi(x_1,\cdot)\|_{L^2(\mb{E})}^2\leq C(\mb{E})\|H(x_1,\cdot)\|_{L^2(\mb{E})}^2.
\ee
Integrating the above inequality with respect to $x_1$ over $[L_0,L_1]$ gives
\be\no
&&\|\nabla_{x'}^2 \psi\|_{L^2(\Om)}\leq C\bigg(\|F\|_{L^2(\Om)}+\|k_{11}\|_{L^{\infty}}\|\p_1^2\psi\|_{L^2(\Omega)}+\sum_{i=2}^3 \|k_{1i}\|_{L^{\infty}}\|\p_{1i}^2\psi\|_{L^2(\Omega)}\\\no
&&\q\q+\|\bar{k}_1\p_1\psi\|_{L^2(\Om)}+\sum_{i,j=2}^3 \|k_{ij}-\delta_{ij}\|_{L^{\infty}(\Omega)}\|\p_{ij}^2\psi\|_{L^2(\Om)}\bigg)\\\label{f17}
&&\leq C_*(\|F\|_{H^1(\Omega)}+ (\epsilon+\delta_0)\|\psi\|_{H^2(\Omega)})
\ee
where \eqref{f142} has been used. Combining this with \eqref{f142} leads to
\be\no
\|\psi\|_{H^2(\Omega)}\leq C_{*}(\|F\|_{H^1(\Omega)}+ (\epsilon+\delta_0)\|\psi\|_{H^2(\Omega)}).
\ee
Choosing $\epsilon+\delta_0$ small enough so that $C_*(\epsilon+\delta_0)\leq \frac12$, one then obtains \eqref{f2} immediately.

\end{proof}

\begin{lemma}\label{H4}
Under the assumptions in Lemma \ref{H1e}, the $H^2(\Om)$ strong solution to \eqref{li1} satisfies
  \be\label{f3}
  &&\|\psi\|_{H^3(\Om)}\le C_*\| F\|_{H^2(\Om)},\\\label{f4}
  &&\|\psi\|_{H^4(\Om)}\le C_*\| F\|_{H^3(\Om)}
  \ee
with $C_*>0$ depends only on the $H^3(\Om)$ norms of $k_{ij}, i,j=1,2,3$ and $\bar{k}_1$.

\end{lemma}

\begin{proof}

Note that $w_2=\p_{1}w_1$ solves
\be\label{f20}\begin{cases}
\sum_{i,j=1}^3 k_{ij} \p_{ij}^2 w_2 + (\bar{k}_1+2\p_1 k_{11})\p_1 w_2+ 4 \sum_{i=2}^3 \p_1 k_{1i}\p_i w_2= F_2, \ \text{in }\Om,\\
(n_2\p_2+n_3\p_3)w_2(x_1,x')=0,\ \ \forall (x_1,x')\in \Ga_w,
\end{cases}\ee
where
\be\no
F_2=\p_1^2 F-(2\bar{k}_1'+\p_1^2 k_{11}) w_2 -2\sum_{i=2}^3 \p_1^2 k_{1i} \p_i w_1 -\bar{k}_1'' w_1-\sum_{i,j=2}^3\big(2\p_1 k_{ij}\p_{ij}^2 w_1 + \p_1^2 k_{ij}\p_{ij}^2 \psi\big).
\ee

Then $\tilde{w}_2= \eta w_2$ solves
\be\label{f21}\begin{cases}
\sum_{i,j=1}^3 k_{ij}\p_{ij}^2 \tilde{w}_2 + (\bar{k}_1+2\p_1 k_{11}) \p_{1} \tilde{w}_2 +4 \sum_{i=2}^3 \p_1 k_{1i}\p_i\tilde{w}_2=\tilde{F}_2,\ \text{in }\ \Omega,\\
\tilde{w}_2(L_0,x')=0,\ \ \text{on }\ \ x'\in \mb{E},\\
(n_2\p_2+n_3\p_3)\tilde{w}_2(x_1,x')=0,\ \forall (x_1,x')\in \Ga_w
\end{cases}\ee
where
\be\no
\tilde{F}_2=\eta F_2 + \eta'\big(2\sum_{i=1}^3 k_{1i}\p_i w_2 + (\bar{k}_1+2\p_1 k_{11})w_2\big) + \eta'' k_{11} w_2.
\ee
Note that if $0<\delta_0\leq \delta_*$ in Lemma \ref{coe1}, then there holds that for any $(x_1,x')\in\Omega$
\be\no
&& 2(\bar{k}_1+2\p_1 k_{11})-\p_{1} k_{11}\leq 2\bar{k}_1+3\bar{k}_{11}'+3\|\p_{1} k_{11}-\bar{k}_{11}'\|_{L^{\infty}}\leq -\kappa_*<0,\\\no
&& 2(\bar{k}_1+2\p_1 k_{11})+\p_{1} k_{11}\leq 2\bar{k}_1+5\bar{k}_{11}'+5\|\p_{1} k_{11}-\bar{k}_{11}'\|_{L^{\infty}}\leq -\kappa_*<0.
\ee
As in Lemma \ref{exist2}, we know $\tilde{w}_2$ is the unique $H^2(\Omega)$ solution to \eqref{f21} with
\be\no
&&\bigg(\int_{\frac{L_0}{2}}^{L_1}\iint_{\mb{E}} |\nabla w_2|^2 dx' dx_1\bigg)^{\frac{1}{2}}\leq \|\tilde{w}_2\|_{H^2(\Omega)}\leq C_*\|\tilde{F}_2\|_{L^2(\Omega)}\\\no
&&\leq C_*\bigg(\|F\|_{H^2}+\|(2\bar{k}_1'+\bar{k}_{11}'')w_2\|_{L^2}+\|\p_{1}^2(k_{11}-\bar{k}_{11})\|_{H^1}\|w_2\|_{H^1}+\|\bar{k}_{1}''\|_{L^\infty}\|w_1\|_{L^2}\\\no
&&\quad+ \sum_{i=2}^3 \|\p_1^2 k_{1i}\|_{H^1}\|\p_i w_1\|_{H^1}+\sum_{i,j=2}^32\|\p_1k_{ij}\|_{H^2}\|\p_{ij}^2 w_1\|_{L^2}+\|\p_1^2 k_{ij}\|_{H^1}\|\p_{ij}^2\psi\|_{H^1}\\\no
&&\quad+ \sum_{i=1}^3 \|k_{1i}\|_{H^2}\|\eta'\p_i w_2\|_{L^2}+ \|\bar{k}_1+2 \p_1 k_{11}\|_{H^2}\|\eta' w_2\|_{L^2}+\|k_{11}\|_{H^2} \|\eta'' w_2\|_{L^2}\bigg)\\\label{f22}
&&\leq C_*(\|F\|_{H^2(\Omega)}+(\epsilon+\delta_0)\|\psi\|_{H^3(\Omega)}),
\ee
where one has used \eqref{H1} and the inequality $\|fg\|_{L^2(\Om)}\leq \|f\|_{H^1(\Om)}\|g\|_{H^1(\Om)}$.

For each fixed $x_1\in [L_0,L_1]$, it follows from \eqref{f1} that
\be\label{f23}\begin{cases}
\sum_{i=2}^3 \p_i^2 w_1(x_1,x') =H_1(x_1,x'), \ \ &\text{in }\mb{E},\\
(n_2\p_2+n_3\p_3)w_1(x_1,x')=0,\ \ \ &\forall (x_1,x')\in \Ga_w,
\end{cases}\ee
where
\be\no
&&H_1:= \p_1 F-\bar{k}_1' \p_1\psi-\sum_{i,j=2}^3 \p_1 k_{ij} \p_{ij}^2 \psi- (\bar{k}_1+\p_1 k_{11})\p_{1}w_1- 2\sum_{i=2}^3 \p_1 k_{1i} \p_i w_1\\\no
&&\q\q \q\q- k_{11}\p_1^2 w_1 -2 \sum_{i=2}^3 k_{1i}\p_{1i}^2 w_1 - \sum_{i,j=2}^3(k_{ij}-\delta_{ij}) \p_{ij}^2 w_1.
\ee
Then it follows from \eqref{f16}-\eqref{f17} that
\be\label{f24}
&&\|\nabla_{x'}^2 w_1\|_{L^2(\Om)}\leq C(\mb{E})\|H_1\|_{L^2(\Omega)}\\\no
&&\leq C(\mb{E})\bigg(\|F\|_{H^1}+\|\bar{k}_1'\p_1\psi\|_{L^2}+ \sum_{i,j=2}^3\|\p_1 k_{ij}\|_{L^{\infty}}\|\p_{ij}^2\psi\|_{L^2}+\|\bar{k}_1+\p_1 k_{11}\|_{L^{\infty}}\|\p_1 w_1\|_{L^2}\\\no
&&\q+ \sum_{i=2}^3\|\p_1 k_{1i}\|_{L^{\infty}}\|\p_i w_1\|_{L^2}+ \sum_{i=1}^3\|k_{1i}\|_{L^{\infty}}\|\p_{1i}^2 w_1\|_{L^2}+ \sum_{i,j=2}^3 \|k_{ij}-\delta_{ij}\|_{L^{\infty}}\|\p_{ij}^2 w_1\|_{L^2}\bigg)\\\no
&&\leq C_*(\|F\|_{H^2}+(\epsilon+\delta_0)\|\psi\|_{H^3}),
\ee
where \eqref{H31} and \eqref{f22} have been used. Furthermore, one can conclude from \eqref{H31},\eqref{f15}, \eqref{f22} and \eqref{f24} that
\be\no
&&\|\nabla_{x'}^3 \psi\|_{L^2(\Omega)}\\\no
&&\leq C(\mb{E})\bigg\|F- k_{11}\p_1^2\psi- 2\sum_{i=2}^3 k_{1i}\p_{1i}^2\psi-\bar{k}_1\p_1\psi- \sum_{i,j=2}^3(k_{ij}-\delta_{ij})\p_{ij}^2\psi\bigg\|_{H^1}\\\no
&&\leq C_*(\|F\|_{H^1}+ \sum_{i=1}^3 \|k_{1i}\|_{H^2}\|\p_{1i}^2\psi\|_{H^1}+\sum_{i,j=2}^3\|k_{ij}-\delta_{ij}\|_{H^2}\|\p_{ij}^2\psi\|_{H^1})\\\label{f25}
&&\leq C_*(\|F\|_{H^2(\Omega)}+(\epsilon+\delta_0)\|\psi\|_{H^3(\Omega)}).
\ee
Collecting the estimates \eqref{H31}, \eqref{f2}, \eqref{f22}, \eqref{f24} and \eqref{f25} leads to
\be\no
\|\psi\|_{H^3(\Omega)}\leq C_*(\|F\|_{H^2(\Omega)}+(\epsilon+\delta_0)\|\psi\|_{H^3(\Omega)}).
\ee
Choosing $\epsilon+\delta_0$ small enough so that $C_*(\epsilon+\delta_0)\leq \frac12$, one can get \eqref{f3} immediately.

Next, note that $w_3=\p_1 w_2$ solves
\be\no\begin{cases}
\sum_{i,j=1}^3 k_{ij}\p_{ij}^2 w_3 + (\bar{k}_1+3\p_1 k_{11})\p_1 w_3 + 6 \sum_{i=2}^3 \p_1k_{1i}\p_i w_3= F_3,\ \ \text{in }\ \Om,\\
(n_2\p_2+n_3\p_3) w_3(x_1,x')=0,\ \ \ (x_1,x')\in \Ga_w,
\end{cases}\ee
where
\be\no
&&F_3=\p_1^3 F-3(\bar{k}_1'+\p_1^2 k_{11})w_3-(3\bar{k}_{1}''+\p_1^3 k_{11})w_2-\bar{k}_1^{(3)} w_1\\\no
&&-\sum_{i=2}^3 (6\p_1^2 k_{1i}\p_i w_2+ 2\p_1^3 k_{1i}\p_i w_1)-\sum_{i,j=2}^3 (3\p_1 k_{ij}\p_{ij}^2 w_2+3\p_1^2 k_{ij}\p_{ij}^2 w_1+\p_1^3 k_{ij}\p_{ij}^2 \psi).
\ee
Thus $\tilde{w}_3= \eta w_3$ solves
\be\label{f32}\begin{cases}
\sum_{i,j=1}^3 k_{ij}\p_{ij}^2 \tilde{w}_3 + (\bar{k}_1+3\p_1 k_{11}) \p_{1} \tilde{w}_3 +6 \sum_{i=2}^3 \p_1 k_{1i}\p_i\tilde{w}_2=\tilde{F}_3,\ \text{on }\ \Omega,\\
\tilde{w}_3(L_0,x')=0,\ \ \text{on }\ \ x'\in \mb{E},\\
(n_2\p_2+n_3\p_3)\tilde{w}_3(x_1,x')=0,\ \forall (x_1,x')\in \Ga_w
\end{cases}\ee
where
\be\no
&&\tilde{F}_3=\eta F_3+ \eta'\big(2\sum_{i=1}^3 k_{1i}\p_i w_3 + (\bar{k}_1+2\p_1 k_{11})w_3\big) + \eta'' k_{11} w_3.
\ee
Note that if $0<\delta_0\leq \delta_*$ in Lemma \ref{coe1}, then there holds that for any $(x_1,x')\in\Omega$
\be\no
&& 2(\bar{k}_1+3\p_1 k_{11})-\p_{1} k_{11}\leq 2\bar{k}_1+3\bar{k}_{11}'+3\|\p_{1} k_{11}-\bar{k}_{11}'\|_{L^{\infty}}\leq -\kappa_*<0,\\\no
&& 2(\bar{k}_1+3\p_1 k_{11})+\p_{1} k_{11}\leq 2\bar{k}_1+5\bar{k}_{11}'+5\|\p_{1} k_{11}-\bar{k}_{11}'\|_{L^{\infty}}\leq -\kappa_*<0.
\ee
Then as in Lemma \ref{exist2}, one can show that $\tilde{w}_3$ is the unique $H^2(\Omega)$ strong solution to \eqref{f32} with the following estimate
\be\no
&&\bigg(\int_{\frac{L_0}{2}}^{L_1}\iint_{\mb{E}} |\nabla w_3|^2 dx' dx_1\bigg)^{\frac{1}{2}}\leq \|\tilde{w}_3\|_{H^2(\Omega)}\leq C_*\|\tilde{F}_3\|_{L^2(\Omega)}\\\no
&&\leq C_*\bigg(\|F\|_{H^3}+ \|(\bar{k}_1'+\bar{k}_{11}'') w_3\|_{L^2} + \|\p_1^2(k_{11}-\bar{k}_{11})\|_{H^1}\|w_3\|_{H^1}+\|(\bar{k}_1''+\bar{k}_{11}^{(3)}) w_2\|_{L^2}\\\no
&&\q +\|\p_1^3(k_{11}-\bar{k}_{11})\|_{L^2}\|w_2\|_{L^{\infty}}+ \sum_{i=2}^3\|\p_1^2 k_{1i}\|_{H^1}\|\p_i w_2\|_{H^1} + \|\p_1^3 k_{1i}\|_{L^2}\|\p_i w_1\|_{L^{\infty}} \\\no
&&\q + \sum_{i,j=2}^3 \|\p_1 k_{ij}\|_{L^{\infty}}\|\p_{ij}^2 w_2\|_{L^2}+ \|\p_1 k_{ij}\|_{H^1}\|\p_{ij}^2 w_1\|_{H^1}+ \|\p_1^3 k_{ij}\|_{L^2}\|\p_{ij}^2 \psi\|_{L^{\infty}}\bigg)\\\no
&&\leq C_*(\|F\|_{H^3(\Om)}+ (\epsilon+\delta_0)\|\psi\|_{H^4(\Om)}).
\ee
Combining this with \eqref{H31} gives
\be\label{f34}
\|\nabla \p_1^3\psi\|_{L^2(\Omega)}\leq C_*(\|F\|_{H^3(\Om)}+ (\epsilon+\delta_0)\|\psi\|_{H^4(\Om)}).
\ee
Rewrite \eqref{f20} for each fixed $x_1\in [L_0,L_1]$ as
\be\no\begin{cases}
\sum_{i=2}^3\p_i^2 w_2(x_1,x')= H_2(x_1,x'), \ \ &\text{in }\mb{E},\\
(n_2\p_2+n_3\p_3)w_2(x_1,x')=0,\ \ \ &\forall (x_1,x')\in \Ga_w,
\end{cases}\ee
where
\be\no
&&H_2=F_2-(\bar{k}_1+2\p_1 k_{11})\p_1 w_2-\sum_{i=2}^3 (4\p_1 k_{1i}\p_i w_2+2k_{1i}\p_{1i}^2 w_2)\\\no
&&\q\q\q\q -k_{11}\p_1^2 w_2-\sum_{i,j=2}^3 (k_{ij}-\delta_{ij})\p_{ij}^2 w_2.
\ee
Then arguing as above, one can get from \eqref{f2} and \eqref{f3} that
\be\no
&&\|\nabla_{x'}^2 w_2\|_{L^2(\Om)}\leq C(\mb{E})\|H_2\|_{L^2(\Omega)}\leq C\bigg(\|F_2\|_{L^2} +\|\bar{k}_1+2\p_1 k_{11}\|_{L^{\infty}}\|\p_1 w_2\|_{L^2}\\\no
&&\q + \sum_{i=2}^3\|\p_1 k_{1i}\|_{L^{\infty}}\|\p_i w_2\|_{L^2} + \sum_{i=1}^3 \|k_{1i}\|_{L^{\infty}}\|\p_{1i}^2 w_2\|_{L^2}+ \sum_{i,j=2}^3\|k_{ij}-\delta_{ij}\|_{L^{\infty}}\|\p_{ij}^2 w_2\|_{L^2}\bigg)\\\no
&&\leq C_*\bigg(\|F\|_{H^2}+\|2\bar{k}_1'+\p_1^2 k_{11}\|_{H^1}\|w_2\|_{H^1}+\sum_{i=2}^3\|\p_1^2 k_{1i}\|_{H^1}\|\p_i w_1\|_{H^1}\\\no
&&\q+\sum_{i,j=2}^3\|\p_1 k_{ij}\|_{L^{\infty}}\|\p_{ij}^2 w_1\|_{L^2}+\|\p_1^2 k_{ij}\|_{H^1}\|\p_{ij}^2 \psi\|_{H^1}\bigg)+C_*(\|F\|_{H^3}+ (\epsilon+\delta_0)\|\psi\|_{H^4})\\\label{f38}
&&\leq C_*(\|F\|_{H^3(\Om)}+ (\epsilon+\delta_0)\|\psi\|_{H^4(\Om)}).
\ee
On the other hand, \eqref{f23} implies that
\be\no
&&\|\nabla_{x'}^3 w_1\|_{L^2(\Omega)}\leq C(\mb{E})\|H_1\|_{H^1(\Omega)}\\\no
&&\leq C_*\bigg(\|F\|_{H^2}+\|\p_1\psi\|_{H^1}+\sum_{i,j=2}^3\|\p_1 k_{ij}\|_{H^2}\|\p_{ij}^2\psi\|_{H^1}+\|\bar{k}_1+\p_1 k_{11}\|_{H^2}\|\p_1 w_1\|_{H^1}\\\no
&&\q+ \sum_{i=2}^3\|\p_1 k_{1i}\|_{H^2}\|\p_i w_1\|_{H^1}+ \sum_{i=1}^3 \|k_{1i}\|_{H^2}\|\p_{1i}^2 w_1\|_{H^1}+ \sum_{i,j=2}^3 \|k_{ij}-\delta_{ij}\|_{H^2}\|\p_{ij}^2 w_1\|_{H^1}\bigg)\\\label{f39}
&&\leq C_*(\|F\|_{H^3(\Om)}+ (\epsilon+\delta_0)\|\psi\|_{H^4(\Om)}),
\ee
where one has used \eqref{f34} and \eqref{f38} to control the term $\|\nabla \p_1 w_1\|_{H^1}$.

Finally, one may conclude from \eqref{f15} that
\be\no
&&\|\nabla_{x'}^4 \psi\|_{L^2(\Om)}\leq C(\mb{E})\|H\|_{H^2(\Omega)}\\\no
&&\leq C_*\bigg(\|F\|_{H^2}+ \sum\limits_{i=1}^3 \|k_{1i}\|_{H^2}\|\p_{1i}^2\psi\|_{H^2}+ \|\bar{k}_1\|_{H^2}\|\p_1\psi\|_{H^2}+ \sum\limits_{i,j=2}^3 \|k_{ij}-\delta_{ij}\|_{H^2}\|\p_{ij}^2\psi\|_{H^2}\bigg)\\\label{f40}
&&\leq C_*(\|F\|_{H^3(\Om)}+ (\epsilon+\delta_0)\|\psi\|_{H^4(\Om)}),
\ee
where \eqref{f38} and \eqref{f39} are employed to control the terms $\|\nabla \p_1\psi\|_{H^2}$.

Collecting all the estimates \eqref{f3} and \eqref{f34}-\eqref{f40} leads to
\be\no
\|\psi\|_{H^4(\Omega)}\leq C_*(\|F\|_{H^3(\Om)}+ (\epsilon+\delta_0)\|\psi\|_{H^4(\Om)}),
\ee
from which \eqref{f4} follows if $\epsilon+\delta_0$ is small enough so that $C_*(\epsilon+\delta_0)\leq \frac 12$.

\end{proof}

\subsection{Proof of Theorem \ref{3dirro}.}\label{proof}

We are now ready to prove Theorem \ref{3dirro}. For any given $\hat{\psi} \in \Sigma_{\de_0}$, by Lemmas \ref{coe1}, \ref{exist2} and Lemma \ref{H4}, there exists a unique solution $\psi \in H^4(\Om)$ to \eqref{li0} with the estimate
\be\no
\| \psi \|_{H^4(\Om)} \le C_*\|F(\na \hat{\psi})\|_{H^3(\Om)} \le \mc{C} \|F(\na \hat{\psi})\|_{H^3(\Om)}.
\ee
Here the positive constant $C_*$ depends only on the $H^3(\Om)$ norms of the coefficients $k_{ij}, i,j=1,2,3$, $\bar{k}_1$, which can be bounded by a constant $\mc{C}$ depending on the $C^3 ([L_0, L_1])$ norm of $\bar{k}_{11}$, $\bar{k}_1$ and the boundary data. In the following, the constant $\mc{C}$ will always denote a generic constant depending only on the background solutions and the boundary data.

Recall the definition of $F(\na \hat{\psi})$ in \eqref{f0}. Since the support of $\eta_0 (x_1)$ defined in \eqref{eta} is contained in $[L_0, \f9{10} L_0]$, according to the $H^4$ estimate \eqref{H31} in Lemma \ref{H3} and the estimates in Lemma \ref{H4}, the following improved estimate holds
\be\no
&& \|\psi\|_{H^4(\Om)}\leq\mc{C}\bigg(\| F_0(\na \hat{\psi} + \epsilon\na \psi_0) \|_{H^3(\Om)}\\\no
&&\q\q\q\q\q+\epsilon\sum_{i,j=1}^3 \|k_{ij}(\na \hat{\psi}+\epsilon\nabla \psi_0)\p_{ij}^2 \psi_0 \|_{H^2(\Om)}+\epsilon\|\bar{k}_1 \p_1\psi_0\|_{H^2(\Om)}\bigg) \\\no
&& \le
\mc{C} ( \eps + \| \hat{\psi} \|_{H^4(\Om)}^2 + \epsilon \|h_0\|_{H^3(\mb{E})} ) \le \mc{C} (\eps + \de_0^2).
\ee

Let $\de_0 = \sqrt{\epsilon}$. Then if $0 < \epsilon \le \eps_0 = \min \{ \f 1{4\mc{C}^2},\epsilon_*, \de_*^2 \}$, one has
\be\no
\| \psi \|_{H^4(\Om)} \le \mc{C}(\eps + \de_0^2) = 2 \mc{C} \eps \le \de_0.
\ee
Thus $\psi \in \Sigma_{\de_0}$ and one can define an operator $\mc{T} \hat{\psi } = \psi $, which maps $\Sigma_{\de_0}$ to itself. It remains to show that the mapping $\mc{T}$ is contractive in a low order norm for a sufficiently small $\eps_0$. Suppose that $\psi^{(m)} = \mc{T} \hat{\psi}^{(m)}$ $(m= 1, 2)$ for any $\hat{\psi}^{(1)}$, $\hat{\psi}^{(2)} \in \Sigma_{\de_0}$. Then it holds that
\be\no\begin{cases}
\sum_{i,j=1}^3 (k_{ij}(\nabla \hat{\psi}^{(1)})\p_{ij}^2 + \bar{k}_1\p_1) (\psi^{(1)} - \psi^{(2)}) = F(\na \hat{\psi}^{(1)}) - F(\na \hat{\psi}^{(2)}) \\
\q\q\q\q\q  - \sum_{i,j=1}^3 (k_{ij}(\nabla \hat{\psi}^{(1)})-k_{ij}(\nabla \hat{\psi}^{(2)}))\p_{ij}^2 \psi^{(2)}, \\
(\psi^{(1)} - \psi^{(2)}) (L_0, x') = 0, \q \forall x'\in \mb{E},\\
(n_2\p_2+n_3\p_3)(\psi^{(1)} - \psi^{(2)})(x_1,x')= 0, \q \forall (x_1,x')\in \Ga_w.
\end{cases}\ee
Since $\psi^{(m)}$, $\hat{\psi}^{(m)} \in \Sigma_{\de_0}$, for $m= 1, 2$, the $H^1$ estimate in Lemma \ref{H1e} yields that
\be\no
&& \| \mc{T} \hat{\psi}^{(1)} - \mc{T} \hat{\psi}^{(2)} \|_{H^1(\Om)} \\\no
&&\leq C_* \| F(\na \hat{\psi}^{(1)}) - F(\na \hat{\psi}^{(2)}) - \sum_{i,j=1}^3 (k_{ij}(\nabla \hat{\psi}^{(1)})-k_{ij}(\nabla \hat{\psi}^{(2)})) \psi^{(2)} \|_{L^2(\Om)} \\\no
&& \le \mc{C} \de_0 \| \hat{\psi}^{(1)} -  \hat{\psi}^{(2)} \|_{H^1(\Om)}
\le \f 12 \|  \hat{\psi}^{(1)}-\hat{\psi}^{(2)} \|_{H^1(\Om)} .
\ee
Thus $\mc{T}$ is a contractive mapping in $H^1(\Om)$ norm and there exists a unique $\psi \in \Sigma_{\de_0}$ such that $\mc{T} \psi = \psi$.

Finally, one can determine the locations of all the sonic points as follows. Recall
\be\no
|{\bf M}|^2\co \b|\frac {{\bf u}}{c (\rho)}\b|^2=\frac{(\bar{u}+\p_1\psi+\epsilon \p_1 \psi_0)^2+|\nabla'(\psi+\epsilon \psi_0)|^2}{c^2(\bar{\rho})+(\gamma-1)(\epsilon \Phi_0-\bar{u}\p_1(\psi+\epsilon \psi_0)-\frac12 |\nabla(\psi +\epsilon \psi_0)|^2)}.
\ee
Let $\bar{M}(x_1)$ be the Mach number of the background flow at $x_1$ so that $\bar{M}^2(x_1)=\frac{\bar{u}^2(x_1)}{c^2(\bar{\rho}(x_1))}$. Then
\be\no
(\bar{M}^2)'(x_1)=-\frac{(\gamma+1)\bar{M}^2\bar{f}}{c^2(\bar{\rho})-\bar{u}^2}>0,\ \forall x_1\in [L_0,0)\cup (0,L_1]
\ee
and
\be\no
\lim_{x_1\to 0}(\bar{M}^2)'(x_1)=\gamma^{-\frac{1}{\gamma+1}} (\rho_0 u_0)^{-\frac{\gamma-1}{\gamma+1}}\sqrt{(\gamma+1) \bar{f}'(0)}>0.
\ee
Thus $|\bar{M}(L_0)|^2 <1,|\bar{M} (L_1)|^2>1,\min_{x_1 \in [L_0, L_1]}(\bar{M}^2)'(x_1)>0$. Since
\be\no
\| |{\bf M}|^2 - \bar{M}^2 \|_{C^{1, \f 12 } (\overline{\Om})} \le \| |{\bf M}|^2 - \bar{M}^2 \|_{H^3(\Om)}\leq \|\psi+\epsilon\psi_0\|_{H^4(\Omega)}\leq \mc{C}\epsilon.
\ee
then for sufficiently small $\eps$, there holds
\be\no
|{\bf M}(L_0,x') |^2 < 1, \q |{\bf M}(L_1,x')|^2 > 1, \q \forall x'\in \mb{E},
\ee
and
\be\no
\p_{x_1}(|{\bf M}(x_1, x')|^2)> 0, \q \forall (x_1,x') \in \Om.
\ee
Thus for any $x'\in \mb{E}$, there exists a unique $\xi (x') \in (L_0, L_1)$ such that $|{\bf M}(\xi(x'), x')|^2 = 1$. Also the implicit function theorem implies that $\xi \in C^1 (\ol{\mb{E}})$ satisfying \eqref{3dsonic}. The proof of Theorem \ref{3dirro} is completed.

\section{The existence and uniqueness of smooth transonic Beltrami flow}\label{belflow}

In this section, we prove Theorem \ref{beltrami}.
The first equation in \eqref{belt} can be rewritten as
\be\no
\sum_{i,j=1}^3( c^2(H)\delta_{ij}- u_i u_j) \p_i u_j + \sum_{i=1}^3 u_i\p_i\Phi=0,
\ee
where $c^2(H)=(\gamma-1)(B_0+\bar{\Phi}+ \epsilon \Phi_0-\frac{1}{2}|{\bf u}|^2)$. Set
\be\no
v_1=u_1-\bar{u}, \ \ v_2=u_2,\ \ v_3=u_3.
\ee

It follows from \eqref{belt} that
\be\label{belt1}\begin{cases}
\sum_{i,j=1}^3 k_{ij}({\bf v}) \p_i v_j+ \bar{k}_1(x_1) v_1= \mathbb{F}({\bf v}),\ \ &\text{in }\Omega,\\
\p_2 v_3-\p_3 v_2= \ka(x) \mc{H}({\bf v}) (\bar{u}+v_1),\ \ &\text{in }\Omega,\\
\p_3 v_1-\p_1 v_3= \ka(x) \mc{H}({\bf v}) v_2,\ \ &\text{in }\Omega,\\
\p_1 v_2-\p_2 v_1= \ka(x) \mc{H}({\bf v}) v_3,\ \ &\text{in }\Omega,\\
((\bar{u}+v_1)\p_1+ v_2\p_2 + v_3\p_3) \ka =0,\ \ &\text{in }\Omega,
\end{cases}\ee
where
\be\label{belt2}\begin{cases}
k_{11}({\bf v})=1-\bar{M}^2-\frac{1}{c^2(\bar{\rho})}\left((\gamma+1)\bar{u} v_1-(\gamma-1)\epsilon \Phi_0+\frac{\gamma+1}{2}v_1^2+\frac{\gamma-1}{2}(v_2^2+v_3^2)\right),\\
k_{ii}({\bf v})=1+\frac{\gamma-1}{c^2(\bar{\rho})}(\epsilon \Phi_0-\bar{u} v_1-\frac{1}{2}\sum_{j=1}^3 v_j^2)-\frac{v_i^2}{c^2(\bar{\rho})}, \ \ i=2,3,\\
k_{1i}({\bf v})=k_{i1}({\bf v})=-\frac{(\bar{u}+v_1)v_i}{c^2(\bar{\rho})}, \ \ i=2,3,\\
k_{23}({\bf v})=k_{23}({\bf v})=-\frac{v_2v_3}{c^2(\bar{\rho})},\ \ \ \ \bar{k}_1(x_1)=\frac{1}{c^2(\bar{\rho})}(\bar{f}-(\gamma+1)\bar{u}\bar{u}'),\\
\mathbb{F}({\bf v})=\frac{\bar{u}'}{c^2(\bar{\rho})}\bigg(-(\gamma-1)\epsilon \Phi_0+\frac{\gamma+1}{2}v_1^2+\frac{\gamma-1}{2}(v_2^2+v_3^2)\bigg)\\
\q\q\q\q\q\q-\frac{\epsilon}{c^2(\bar{\rho})}\bigg((\bar{u}+v_1)\p_1\Phi_0+\sum_{i=2}^3v_i\p_i\Phi_0\bigg),\\
\mc{H}({\bf v})=\left(\frac{\gamma-1}{\ga}\right)^{\frac{1}{\gamma-1}}\left(B_0+\bar{\Phi}+\epsilon \Phi_0-\frac12 ({\bar u}+v_1)^2-\frac12(v_2^2+v_3^2)\right)^{\frac{1}{\gamma-1}}.
\end{cases}\ee
The boundary conditions in \eqref{bbc} become
\be\no\begin{cases}
v_i(L_0,x')= \epsilon h_i(x'), \ \ &\forall x'\in \mb{E},\ \ i=2,3,\\
(n_2 v_2+n_3 v_3)(x_1,x')=0, \ \ &\forall (x_1,x')\in \Ga_w:=[L_0,L_1]\times \p \mathbb{E}.
\end{cases}\ee
Note that the function $\kappa$ is transported along the streamline and the boundary data for $\ka$ at the entrance is given by
\be\no
\ka(L_0,x')= \frac{\epsilon (\p_2 h_3-\p_3 h_2)(x')}{\{\mc{H}({\bf v})(\bar{u}+v_1)\}(L_0,x')}.
\ee
Theorem \ref{3dirro} suggests that one may look for the solution ${\bf v}$ in the class $(H^3(\Omega))^3$. In this class, by the trace theorem, the function $\kappa(L_0,\cdot)$ is in $H^{\frac{5}{2}}(\mb{E})$ and $\kappa$ belongs to $H^{\frac52}(\Omega)$ only. Thus there is a loss of $\frac12$ derivatives. However, we note that the flow near the entrance is purely subsonic and it is plausible to improve the regularity of the velocity to be $H^4$ near the entrance. This motivates the following definition of the class of solutions:
\be\no
&&\mc{S}=\bigg\{{\bf v}\in H^3(\Omega)\cap H^4(\Omega_{1/3}):\|{\bf v}\|_{H^3(\Om)}+\|{\bf v}\|_{H^4(\Om_{1/3})}\leq \delta_1, \\\no
&&\quad \q\q\q v_i(L_0,x')=\epsilon h_i(x'), i=2,3,\forall x'\in \mb{E},\\\no
&&\q\q\q\q (n_2\p_2+ n_3\p_3)v_1(x_1,x')=(n_2 v_2 + n_3 v_3)(x_1,x')=0 \ \text{on }\Ga_w\bigg\},
\ee
where $\delta_1>0$ is a small constant to be specified later and $\Omega_{1/3}=\{(x_1,x'): L_0<x_1<\frac13 L_0,x'\in\mb{E}\}$. Note that though $\mc{S}$ is not a linear space, but it can be represented as $\mc{S}=(0,\epsilon h_2, \epsilon h_3)+ \mc{S}_0$, where $\mc{S}_0$ is a Banach space. The condition $(n_2\p_2+ n_3\p_3)v_1(x_1,x')=0$ on $\Ga_w$ is included in $\mc{S}$ to guarantee the compatibility conditions for better regularity near $\{(L_0,x'):x'\in \p\mb{E}\}$. As seen later from \eqref{dc13}, this can be true due to the fact that the vorticity on the wall is zero since $(\p_2 h_3-\p_3 h_2)(x')=0$ on $\p\mb{E}$.

For any given $\hat{{\bf v}}\in \mc{S}$, we define an operator $\mc{T}$ mapping from $\mc{S}$ to itself by solving the problem
\be\label{beltl}\begin{cases}
\sum_{i,j=1}^3 k_{ij}(\hat{{\bf v}}) \p_i v_j+ \bar{k}_1(x_1) v_1= \mathbb{F}(\hat{{\bf v}}),\ \ &\text{in }\Omega,\\
\p_2 v_3-\p_3 v_2+ \p_1\Pi= \ka(x) \mc{H}(\hat{{\bf v}}) (\bar{u}+\hat{v}_1),\ \ &\text{in }\Omega,\\
\p_3 v_1-\p_1 v_3+ \p_2\Pi= \ka(x) \mc{H}(\hat{{\bf v}}) \hat{v}_2,\ \ &\text{in }\Omega,\\
\p_1 v_2-\p_2 v_1+ \p_3\Pi= \ka(x) \mc{H}(\hat{{\bf v}}) \hat{v}_3,\ \ &\text{in }\Omega,\\
((\bar{u}+\hat{v}_1)\p_1+ \hat{v}_2\p_2 + \hat{v}_3\p_3) \ka =0,\ \ &\text{in }\Omega,
\end{cases}\ee
with the boundary conditions
\be\label{beltbc}\begin{cases}
v_i(L_0,x')= \epsilon h_i(x'), \ \ &\forall x'\in \mb{E},\ i=2,3,\\
\p_1\Pi(L_0,x')=\p_1\Pi(L_1,x')=0,\ \ \ &\forall x'\in \mb{E},\\
(n_2 v_2+n_3 v_3)(x_1,x')=\Pi(x_1,x')=0, \ \ &\forall (x_1,x')\in \Ga_w:=[L_0,L_1]\times \p \mathbb{E}.
\end{cases}\ee
It should be noted that the boundary condition for the transport equation $\eqref{beltl}_5$ for $\kappa$ at $x_1=L_0$ is given in \eqref{ka21}, which follows from the equations \eqref{beltl} and \eqref{beltbc}.

Note that here we consider an enlarged system of the problem \eqref{belt1} since the source term in the curl system is not divergence free in general after the linearization. Such a technique is motivated by our previous work \cite{wx23a} on the structural stability of cylindrical transonic shock. The mixed boundary conditions for $\Pi$ are chosen such that one can easily verify the compatibility condition on $\{(L_0,x'):x'\in \p\mb{E}\}$. The problem \eqref{beltl} with boundary conditions \eqref{beltbc} is solved as follows:

{\bf Step 1.} Solvability of $\ka(x)$. Consider the problem
\be\label{ka21}\begin{cases}
\p_1 \ka + \sum_{i=2}^3 \frac{\hat{v}_i}{\bar{u}+ \hat{v}_1}\p_i \ka=0,\ \ &\text{in }\ \Om,\\
\ka(L_0,x')= \frac{\epsilon (\p_2 h_3-\p_3 h_2)(x')}{\{\mc{H}(\hat{{\bf v}})(\bar{u}+\hat{v}_1)\}(L_0,x')},\ \ \ &\forall x'\in \mb{E}.
\end{cases}\ee
Since ${\bf v}\in H^4(\Om_{1/3})$, $\kappa(L_0,x')\in H^{3}(\mb{E})$ due to the trace's theorem. By simple energy estimates and a density approximation argument as in \cite{seth16}, then there exists a unique classical solution $\kappa\in H^3(\Omega)$ to \eqref{ka21} with the estimate
\be\label{ka23}
\|\ka\|_{H^3(\Om)}\leq C\epsilon \bigg\|\frac{(\p_2 h_3-\p_3 h_2)(x')}{\{\mc{H}(\hat{{\bf v}})(\bar{u}+\hat{v}_1)\}(L_0,x')}\bigg\|_{H^{3}(\mb{E})}\leq \mc{C}\epsilon.
\ee

Since $(\p_2 h_3-\p_3 h_2)(x')=0$ for all $x'\in \p\mb{E}$, there holds
\be\label{ka22}
\kappa(x_1,x')=\p_1\kappa(x_1,x')=\p_1^2 \kappa(x_1,x')=\sum_{i=2}^3 (\hat{v}_i\p_i\kappa)(x_1,x')=0,\ \ \ \text{on }\Ga_w.
\ee

{\bf Step 2.} Solvability of $\Pi$. Consider the following problem
\be\label{pi1}\begin{cases}
\Delta \Pi= \p_1(\ka(x)\mc{H}(\hat{{\bf v}})(\bar{u}+\hat{v}_1))+ \sum_{i=2}^3 \p_i(\ka(x)\mc{H}(\hat{{\bf v}})\hat{v}_i),\ \ &\text{in }\Om,\\
\p_1\Pi(L_0,x')=\p_1\Pi(L_1,x')=0, \ \ &\forall x'\in\mb{E},\\
\Pi(x_1,x')=0,\ \ &\forall (x_1,x')\in \Ga_w.
\end{cases}\ee
The solvability of \eqref{pi1} follows from that for the problem \eqref{43} in the Appendix \S\ref{div-curl}. Indeed, by \eqref{ka22}, one has
\be\no
\bigg\{\p_1(\ka\mc{H}(\hat{{\bf v}})(\bar{u}+\hat{v}_1))+ \sum_{i=2}^3 \p_i(\ka\mc{H}(\hat{{\bf v}})\hat{v}_i)\bigg\}(x_1,x')=0\ \ \ \text{on }\ \Ga_w.
\ee
Thus corresponding to \eqref{431}, the compatibility condition for \eqref{pi1} holds on $\{(x_1,x'): x_1= L_0/L_1, x'\in\p\mb{E}\}$:
\be\no
\p_1\bigg(\p_1(\ka(x)\mc{H}(\hat{{\bf v}})(\bar{u}+\hat{v}_1))+ \sum_{i=2}^3 \p_i(\ka(x)\mc{H}(\hat{{\bf v}})\hat{v}_i)\bigg)(L_0/L_1,x')=0,\ \text{on }\p\mb{E}.
\ee
Then there exists a unique solution $\Pi\in H^4(\Omega)$ to \eqref{pi1} satisfying
\be\no
&&\|\Pi\|_{H^4(\Omega)}\leq C(\Omega)(\|\ka(x)\mc{H}(\hat{{\bf v}})(\bar{u}+\hat{v}_1)\|_{H^3(\Om)}+ \sum_{i=2}^3\|\ka(x)\mc{H}(\hat{{\bf v}})\hat{v}_i\|_{H^3(\Omega)})\\\no
&&\leq \mc{C}\|\ka\|_{H^3(\Omega)}\leq \mc{C} \epsilon,
\ee
where the estimate \eqref{ka23} is used. It follows from the last boundary condition in \eqref{pi1} that
\be\label{pi5}
\p_1^2\Pi(x_1,x')=(n_2\p_3\Pi-n_3\p_2\Pi)(x_1,x')=0,\ \ \text{on }\ \ \Ga_w.
\ee

{\bf Step 3.} Solve a divergence-curl system with mixed boundary conditions. Consider the following problem
\be\label{dc}\begin{cases}
\p_1 \dot{v}_1 + \p_2 \dot{v}_2 + \p_3 \dot{v}_3 =0,\ \ &\text{in }\ \Om,\\
\p_2 \dot{v}_3-\p_3 \dot{v}_2 =\ka(x) \mc{H}(\hat{{\bf v}}) (\bar{u}+\hat{v}_1)-\p_1\Pi,\ \ &\text{in }\ \Om,\\
\p_3 \dot{v}_1-\p_1 \dot{v}_3 =\ka(x) \mc{H}(\hat{{\bf v}}) \hat{v}_2-\p_2\Pi,\ \ &\text{in }\ \Om,\\
\p_1 \dot{v}_2-\p_2 \dot{v}_1 =\ka(x) \mc{H}(\hat{{\bf v}}) \hat{v}_3-\p_3\Pi,\ \ &\text{in }\ \Om,\\
\dot{v}_i(L_0,x')=\epsilon h_i(x'), \ \ i=2,3,\ \ &\forall x'\in \mb{E},\\
\dot{v}_1(L_1,x')=0,\ \ &\forall x'\in \mb{E},\\
(n_2\dot{v}_2 + n_3 \dot{v}_3)(x_1,x')=0,\ \ \ &\forall (x_1,x')\in \Ga_w.
\end{cases}\ee

The velocity field ${\bf \dot{v}}$ can be decomposed as ${\bf \dot{v}}=\tilde{{\bf {v}}} + \check{{\bf v}}$, where $\tilde{{\bf {v}}}=(\tilde{v}_1,\tilde{v}_2, \tilde{v}_3)^t$ and $\check{{\bf v}}=(\check{v}_1,\check{v}_2, \check{v}_3)^t$ solve
\be\label{dc11}\begin{cases}
\p_1 \tilde{v}_1 + \p_2 \tilde{v}_2 + \p_3 \tilde{v}_3 =0,\ \ &\text{in }\ \Om,\\
\p_2 \tilde{v}_3-\p_3 \tilde{v}_2 =\ka(x) \mc{H}(\hat{{\bf v}}) (\bar{u}+\hat{v}_1)-\p_1\Pi,\ \ &\text{in }\ \Om,\\
\p_3 \tilde{v}_1-\p_1 \tilde{v}_3 =\ka(x) \mc{H}(\hat{{\bf v}}) \hat{v}_2-\p_2\Pi,\ \ &\text{in }\ \Om,\\
\p_1 \tilde{v}_2-\p_2 \tilde{v}_1 =\ka(x) \mc{H}(\hat{{\bf v}}) \hat{v}_3-\p_3\Pi,\ \ &\text{in }\ \Om,\\
\tilde{v}_1(L_0,x')= \tilde{v}_1(L_1,x')=0,\ \ &\forall x'\in \mb{E},\\
(n_2\tilde{v}_2 + n_3 \tilde{v}_3)(x_1,x')=0,\ \ \ &\forall (x_1,x')\in \Ga_w
\end{cases}\ee
and  solves
\be\label{dc22}\begin{cases}
\p_1 \check{v}_1 + \p_2 \check{v}_2 + \p_3 \check{v}_3 =0,\ \ &\text{in }\ \Om,\\
\text{curl }\check{{\bf v}} =0,\ \ &\text{in }\ \Om,\\
\check{v}_i(L_0,x')=\epsilon h_i(x')-\tilde{v}_i(L_0,x'), \ \ i=2,3,\ \ &\forall x'\in \mb{E},\\
\check{v}_1(L_1,x')=0,\ \ &\forall x'\in \mb{E},\\
(n_2\check{v}_2 + n_3 \check{v}_3)(x_1,x')=0,\ \ \ &\forall (x_1,x')\in \Ga_w,
\end{cases}\ee
respectively. According to Proposition \ref{dcp}, the solvability condition for the curl system in \eqref{dc11} follows from \eqref{pi1}. It remains to verify the compatibility condition \eqref{410} for \eqref{dc11}, which holds due to
\be\no
&&\p_1(\ka\mc{H}(\hat{{\bf v}}) (\bar{u}+\hat{v}_1)-\p_1\Pi)(L_0,x')\\\no
&&=\ka \p_1(\mc{H}(\hat{{\bf v}}) (\bar{u}+\hat{v}_1))+\p_1\ka \mc{H}(\hat{{\bf v}}) (\bar{u}+\hat{v}_1)-\p_1^2\Pi)(L_0,x')=0,\ \ \text{on }\ \ \p\mb{E},\\\no
&& \{ n_2(\ka(x) \mc{H}(\hat{{\bf v}}) \hat{v}_3-\p_3\Pi)-n_3(\ka(x) \mc{H}(\hat{{\bf v}}) \hat{v}_3-\p_3\Pi)\}(L_0,x')\\\no
&&=\ka\mc{H}(\hat{{\bf v}}) (n_2\hat{v}_3- n_3 \hat{v}_2)(L_0,x')-(n_2\p_3\Pi-n_3\p_2\Pi)(L_0,x')=0,\ \ \text{on }\ \ \p\mb{E},
\ee
where \eqref{ka22} and \eqref{pi5} have been used here.

Thus there exists a unique solution ${\bf \tilde{v}}\in H^4(\Omega)$ to \eqref{dc11} such that
\be\no
&&\|{\bf \tilde{v}}\|_{H^4(\Omega)}\leq C(\Omega)(\|\ka\mc{H}(\hat{{\bf v}}) (\bar{u}+\hat{v}_1)-\p_1\Pi\|_{H^3(\Om)}+\sum_{i=2}^3\|\ka\mc{H}(\hat{{\bf v}}) \hat{v}_i-\p_i\Pi\|_{H^3(\Om)})\\\label{dc12}
&&\leq C(\|\ka\|_{H^3(\Om)}+\|\Pi\|_{H^3(\Omega)})\leq \mc{C}\epsilon.
\ee
Furthermore, on $\Ga_w$ there holds
\be\no
&&0= n_2 \p_1 \tilde{v}_2 + n_3 \p_1 \tilde{v}_3= n_2(\p_2 \tilde{v}_1+ \ka \mc{H}(\hat{{\bf v}})\hat{v}_3-\p_3 \Pi)+ n_3 (\p_3 \tilde{v}_1- \ka \mc{H}(\hat{{\bf v}})\hat{v}_2+\p_2 \Pi)\\\no
&&= (n_2\p_2 \tilde{v}_1 + n_3 \p_3 \tilde{v}_1)+ \ka \mc{H}(\hat{{\bf v}})(n_2\hat{v}_3-n_3 \hat{v}_2)- (n_2\p_3\Pi-n_3\p_2 \Pi)(x_1,x').
\ee
Note that \eqref{ka22} and \eqref{pi5} imply
\be\label{dc13}
n_2\p_2 \tilde{v}_1 + n_3 \p_3 \tilde{v}_1=0\ \ \text{on } \Ga_w.
\ee
This, together with the first equation in \eqref{dc11}, implies that
\be\label{dc14}
(n_2\p_2+ n_3\p_3)(\p_2 \tilde{v}_2 + \p_3 \tilde{v}_3)(x_1,x')=0\ \ \ \ \text{on }\ \Ga_w.
\ee

Next, we solve the problem \eqref{dc22}. Note first that
\be\no
&&\p_2(\epsilon h_3(x')-\tilde{v}_3(L_0,x'))-\p_3(\epsilon h_2(x')-\tilde{v}_2(L_0,x'))\\\no
&&= \epsilon (\p_2 h_3-\p_3 h_2)(x')-(\p_2 \tilde{v}_3-\p_3 \tilde{v}_2)(L_0,x')\\\no
&&=\epsilon (\p_2 h_3-\p_3 h_2)(x')-(\kappa \mc{H}(\hat{{\bf v}})(\bar{u}+\hat{v}_1)+\p_1\Pi)(L_0,x')\equiv 0,\ \ \text{in }\mb{E}.
\ee


Since $\mathbb{E}\subset \mathbb{R}^2$ is simply-connected, there exists a unique $h(x')\in H^{9/2}(\mb{E})$ with zero mean $\iint_{\mb{E}} h(x') dx'=0$ such that $\p_i h(x')=\epsilon h_i(x')-\tilde{v}_i(L_0,x')$ for all $x'\in \mb{E}$ and $i=2,3$. Moreover,
\be\label{dc21}
&&\|h\|_{H^{9/2}(\mb{E})}=\|h\|_{L^2(\mb{E})}+\|\nabla'h\|_{H^{7/2}(\mb{E})}\leq C(\mb{E})\|\nabla'h\|_{H^{7/2}(\mb{E})}\\\no
&&\leq C(\mb{E})\sum_{i=2}^3\|\epsilon h_i-\tilde{v}_i(L_0,\cdot)\|_{H^{7/2}(\mb{E})}\leq \mc{C} \epsilon.
\ee
Since $\text{curl }\check{{\bf v}}\equiv 0$ in $\Om$, then $\check{{\bf v}}=\nabla \phi$ for some potential function $\phi$ and the problem \eqref{dc22} becomes
\be\label{dc23}\begin{cases}
\Delta \phi(x_1,x')=0,\ \ &\text{in }\ \Om,\\
\phi(L_0,x')= h(x'),\ \ \ &\forall x'\in \mb{E},\\
\p_1\phi(L_1,x')=0,\ \ \  &\forall x'\in \mb{E},\\
(n_2\p_2 + n_3 \p_3)\phi(x_1,x')=0,\ \ &\text{on }\ \Ga_w.
\end{cases}\ee
Suppose that there exists a smooth solution $\phi\in H^5(\Omega)$ to \eqref{dc23}, then $(n_2\p_2 + n_3 \p_3) \p_1^2\phi(x_1,x')=0$ on $\Ga_w$ and thus
\be\no
&&0=(n_2\p_2 + n_3 \p_3)(\p_1^2+ \p_2^2+ \p_3^2)\phi(L_0,x')\\\no
&&\q=(n_2\p_2 + n_3 \p_3)(\p_2^2+ \p_3^2)\phi(L_0,x')\\\label{dc24}
&&\q=(n_2\p_2 + n_3 \p_3)(\p_2^2+ \p_3^2) h(x') \ \ \ \ \text{on }\ \ \p\mb{E}.
\ee
For $h$ defined in \eqref{dc21}, the compatibility condition \eqref{dc24} follows from \eqref{bcp1} and \eqref{dc14}:
\be\no
\epsilon (n_2\p_2 + n_3 \p_3)(\p_2 h_2 + \p_3 h_3)(x')- (n_2\p_2 + n_3 \p_3)(\p_2 \dot{v}_2^1+\p_3 \dot{v}_3^1)(L_0,x')=0\ \ \ \text{on }\ \ \p\mb{E}.
\ee
The solution $\phi$ to \eqref{dc23} can be obtained as $\phi(x_1,x')=\sum_{m=0}^{\infty} s_m(x_1) b_m(x')$, where $s_m(x_1)$ solves
\be\label{dc25}\begin{cases}
s_m''-\lambda_m s_m(x_1)=0,\ \ \ \forall x_1\in [L_0,L_1],\\
s_m(L_0)=r_m,\ \ \ s_m'(L_1)=0
\end{cases}\ee
and $r_m=\iint_{\mb{E}} h(x') b_m(x') dx'$. Indeed, for $m\geq 1$ the function $s_m(x_1)$ has the form
\be\no
s_m(x_1)=\frac{r_m}{1+e^{2\sqrt{\lambda_m}(L_0-L_1)}}\bigg(e^{-\sqrt{\lambda_m}(x_1-L_0)}+ e^{\sqrt{\lambda_m}(x_1+L_0-2L_1)}\bigg).
\ee
Using the equation in \eqref{dc25}, one can get by simple calculations that
\be\no
&&\int_{L_0}^{L_1} |s_m^{(5)}|^2+\la_m |s_m^{(4)}|^2+\la_m^2 |s_m^{(3)}|^2 + 2 \la_m^3 |s_m''|^2 + 2 \la_m^4 |s_m'|^2+ \la_m^5 |s_m|^2 dx_1\\\no
&&=\int_{L_0}^{L_1}\la_m^4 |s_m'|^2+\la_m^5 |s_m|^2+\la_m^4 |s_m'|^2 + 2 \la_m^3 |s_m''|^2 + 2 \la_m^4 |s_m'|^2+ \la_m^5 |s_m|^2 dx_1\\\no
&&=2\lambda_m^3\int_{L_0}^{L_1} |s_m''|^2+2\lambda_m |s_m'|^2+ \la_m^2 |s_m|^2 dx_1=-4\lambda_m^4 s_m(L_0) s_m'(L_0)\\\no
&&=\frac{4\la_m^{\frac92}r_m^2}{1+e^{2\sqrt{\lambda_m}(L_0-L_1)}}(1-e^{2\sqrt{\lambda_m}(L_0-L_1)})\leq 4\la_m^{\frac92}r_m^2.
\ee
The estimate for $m=0$ can be derived easily. Therefore
\be\label{dc26}
&&\|\phi\|_{H^5(\Omega)}\leq C(\Omega)\sum_{m=0}^{\infty}\int_{L_0}^{L_1}\sum_{i=0}^5 \lambda_m^i |s_m^{(5-i)}(x_1)|^2 dx_1\\\no
&&\leq C\sum_{m=0}^{\infty} \lambda_m^{\frac{9}{2}} r_m^2\leq C\|h\|_{H^{\frac{9}{2}}(\mb{E})}\leq \mc{C}\epsilon.
\ee

{\bf Step 4.} Solvability of the velocity field. Define $w_i=v_i-\dot{v}_i, i=1,2,3$. Then it follows from \eqref{beltl},\eqref{beltbc} and \eqref{dc} that
\be\label{dc31}\begin{cases}
\sum_{i,j=1}^3 k_{ij}(\hat{{\bf v}}) \p_i w_j + \bar{k}_1(x_1) w_1 =\mathbb{F}(\hat{{\bf v}})-\sum_{i,j=1}^3 k_{ij}(\hat{{\bf v}}) \p_i \dot{v}_j - \bar{k}_1(x_1) \dot{v}_1,\ \text{in }\ \Om,\\
\text{curl }{\bf w}=0,\ \text{in }\ \Om,\\
w_2(L_0,x')=w_3(L_0,x')= 0,\ \forall x'\in \mb{E},\\
(n_2w_2 + n_3 w_3)(x_1,x')=0,\ \ \forall (x_1,x')\in \Ga_w.
\end{cases}\ee
In terms of a potential function $\psi$ of ${\bf w}$, i.e. ${\bf w}=\nabla \psi$, \eqref{dc31} becomes
\be\label{dc32}\begin{cases}
\displaystyle\sum_{i,j=1}^3 k_{ij}(\hat{{\bf v}}) \p_{ij}^2 \psi + \bar{k}_1(x_1) \p_1 \psi=F(\hat{{\bf v}})-\sum_{i,j=1}^3 k_{ij}(\hat{{\bf v}}) \p_i \dot{v}_j - \bar{k}_1(x_1) \dot{v}_1,\ \text{in }\ \Om,\\
\psi(L_0,x')=0,\ \ \forall x'\in \mb{E},\\
(n_2\p_2 + n_3 \p_3 )\psi(x_1,x')=0,\ \text{on }\Ga_w.
\end{cases}\ee

Since $\hat{{\bf v}}\in \mc{S}$, the coefficients $k_{ij}(\hat{{\bf v}}), i,j=1,2,3$ defined in \eqref{belt2} satisfy the estimates \eqref{coe100}-\eqref{coe103} in Lemma \ref{coe1} with $\nabla \psi+ \epsilon \nabla \psi_0$ replaced by $\hat{{\bf v}}$. The identity \eqref{coe104} can also be verified as in Lemma \ref{coe1} with $\nabla \psi_1$ replaced by $\hat{{\bf v}}$. It follows from Lemma \ref{exist2} that there exists a unique strong solution $\psi\in H^2(\Omega)$ to the problem \eqref{dc32}. To improve the regularity of $\psi$ near the the corner circle $\{(L_0,x'): x'\in \p\mb{E}\}$, as in Lemma \ref{H3}, one may rewrite \eqref{dc32} as
\be\no\begin{cases}
\p_1(\bar{k}_{11}\p_1\psi)+ \p_2^2\psi + \p_3^2 \psi= \mc{G}(x_1,x'), \ \ &\text{in }\ \Om,\\
\psi(L_0,x')=0,\ \ \ &\forall x'\in \mb{E},\\
(n_2\p_2 + n_3 \p_3 )\psi(x_1,x')=0,\ \ &\text{on }\Ga_w.
\end{cases}\ee
with
\be\no
&&\mc{G}=F(\hat{{\bf v}})-\sum_{i,j=1}^3 k_{ij}(\hat{{\bf v}}) \p_i \dot{v}_j - \bar{k}_1(x_1) \dot{v}_1- (k_{11}(\hat{{\bf v}})-\bar{k}_{11})\p_1^2 \psi + (\bar{k}_{11}'-\bar{k}_1)\p_1 \psi\\\no
&&\q\q- 2 \sum_{i=2}^3 k_{1i}(\hat{{\bf v}})\p_{1i}^2\psi -\sum_{i,j=2}^3 (k_{ij}(\hat{{\bf v}})-\delta_{ij})\p_{ij}^2\psi.
\ee
It remains to verify that $(n_2\p_2+ n_3 \p_3 )\mc{G}(L_0,x')=0$ holds on $\p\mb{E}$. Since $\hat{{\bf v}}\in \mc{S}$, $(n_2\p_2+n_3\p_3)\hat{v}_1=0$ holds on $\Ga_w$ and $\hat{v}_i(L_0,x')=\epsilon h_i(x')$ for $i=2,3$. These, together with \eqref{Phi1} and \eqref{bcp1}, yield that
\be\label{dc35}\begin{cases}
(n_2\p_2 + n_3 \p_3)\{F(\hat{{\bf v}})(L_0,x')\}=0,\\
(n_2\p_2 + n_3 \p_3)\{k_{ij}(\hat{{\bf v}})(L_0,x')\}=0,\ \ \text{on }\p\mb{E} \ \ \text{for }i,j=1,2,3,\\
k_{12}(\hat{{\bf v}})(L_0,x')=k_{13}(\hat{{\bf v}})(L_0,x')=0,\ \ \text{on }\ \p\mb{E}.
\end{cases}\ee
Note that ${\bf \dot{v}}=\tilde{{\bf v}}+ \check{{\bf v}}= \tilde{{\bf v}}+ \nabla \phi$, then it follows from \eqref{dc13} and the last condition in \eqref{dc23} that
\be\no\begin{cases}
(n_2\p_2+n_3\p_3)\dot{v}_1(x_1,x')=0, \ \ \text{on }\Ga_w\\
\dot{v}_i(L_0,x')= \epsilon h_i(x'),\ \ \ \forall x'\in \mb{E}, \ i=2,3.
\end{cases}\ee
Note that $\psi(L_0,x')=0$ on $\mb{E}$, thus $\sum_{i,j=2}^3 (k_{ij}(\hat{{\bf v}})-\delta_{ij})\p_{ij}^2\psi(L_0,x')=0$ on $\mb{E}$. These, together with \eqref{dc35}, imply that $(n_2\p_2+ n_3 \p_3 )\mc{G}(L_0,x')=0$ holds on $\p\mb{E}$. Therefore
\be\no
&&\|\psi\|_{H^4(\Omega)}+\|\psi\|_{H^5(\Omega_{1/3})}\leq C_*\|F(\hat{{\bf v}})-\sum_{i,j=1}^3 k_{ij}(\hat{{\bf v}}) \p_i \dot{v}_j - \bar{k}_1(x_1) \dot{v}_1\|_{H^3(\Omega)}\\\label{dc36}
&&\leq \mc{C} (\epsilon+\|\hat{{\bf v}}\|_{H^3(\Omega)}^2 + \|\tilde{{\bf v}}\|_{H^4(\Omega)}+\|\nabla \phi\|_{H^4(\Omega)})\leq \mc{C}(\epsilon+\delta_1^2).
\ee
The estimate of $\|\psi\|_{H^5(\Omega_{1/3})}$ can be derived as in Lemma \ref{li50}. In summary, by \eqref{dc12}-\eqref{dc13},\eqref{dc26} and \eqref{dc36}, one could verify that ${\bf v}= \tilde{{\bf v}}+\nabla \phi + \nabla\psi$ belongs to $\mc{S}$ and solves \eqref{beltl}-\eqref{beltbc} such that
\be\no
\|{\bf v}\|_{H^3(\Om)}+\|{\bf v}\|_{H^4(\Om_{1/3})}\leq \mc{C}(\epsilon + \delta_0^2).
\ee
Set $\de_1 = \sqrt{\epsilon}$. If $0 < \epsilon \le \epsilon_1 = \min \{ \f 1{4 M^2},\epsilon_*, \de_*^2\}$, then
\be\no
\|{\bf v}\|_{H^3(\Om)}+\|{\bf v}\|_{H^4(\Om_{1/3})}\leq \mc{C}(\epsilon + \delta_1^2)\leq \delta_1.
\ee
Thus the operator $\mc{T}$ maps $\mc{S}$ to itself. It suffices to show that $\mc{T}$ is a contraction in the weak norm $\|{\bf Y}\|_{H^2(\Omega)}+\|{\bf Y}\|_{H^3(\Omega_{1/3})}$. For any $\hat{{\bf v}}^m\in \mc{S}$ for $m=1,2$, denote the corresponding solutions to \eqref{ka21}, \eqref{pi1},\eqref{dc} and \eqref{dc31} by $\kappa_m, \Pi_m, \dot{{\bf v}}^m$ and $\psi_m$ when $\hat{{\bf v}}$ is replaced by $\hat{{\bf v}}^m$ for $m=1,2$, respectively. Set
\be\no
\hat{{\bf Y}}= \hat{{\bf v}}^1-\hat{{\bf v}}^2, \ \ \dot{{\bf Y}}=\dot{{\bf v}}^1-\dot{{\bf v}}^2, \ \ \psi=\psi_1-\psi_2,\ {\bf Y}={\bf v}^1-{\bf v}^2.
\ee
Then
\be\no\begin{cases}
\p_1 (\ka_1-\ka_2) + \sum_{i=2}^3 \frac{\hat{v}_i^1}{\bar{u}+ \hat{v}_1^1}\p_i (\ka_1-\ka_2)=-\sum_{i=2}^3 \bigg(\frac{\hat{v}_i^1}{\bar{u}+ \hat{v}_1^1}-\frac{\hat{v}_i^2}{\bar{u}+ \hat{v}_1^2}\bigg)\p_i \ka_2,\ \text{in }\ \Om,\\
(\ka_1-\ka_2)(L_0,x')= \epsilon (\p_2 h_3-\p_3 h_2)(x')\bigg(\frac{1}{\mc{H}(\hat{{\bf v}}^1)(\bar{u}+\hat{v}_1^1)}-\frac{1}{\mc{H}(\hat{{\bf v}}^2)(\bar{u}+\hat{v}_1^2)}\bigg)(L_0,x'),\ \forall x'\in \mb{E}.
\end{cases}\ee
Standard estimates lead to
\be\no
&&\|(\ka_1-\ka_2)\|_{H^2(\Omega)}\leq C\bigg(\sum_{i=2}^3\bigg\|\frac{\hat{v}_i^1}{\bar{u}+ \hat{v}_1^1}-\frac{\hat{v}_i^2}{\bar{u}+ \hat{v}_1^2}\bigg\|_{H^2(\Om)}\|\p_i\ka_2\|_{H^{2}(\Omega)}\\\no
&&\q\q+\epsilon \bigg\|(\p_2 h_3-\p_3 h_2)(x')\bigg(\frac{1}{\mc{H}(\hat{{\bf v}}^1)(\bar{u}+\hat{v}_1^1)}-\frac{1}{\mc{H}(\hat{{\bf v}}^2)(\bar{u}+\hat{v}_1^2)}\bigg)(L_0,x')\bigg\|_{H^2(\mb{E})}\bigg)\\\no
&&\leq \mc{C}\epsilon (\|\hat{{\bf Y}}(L_0,\cdot)\|_{H^2(\mb{E})}+\|\hat{{\bf Y}}\|_{H^2(\Omega)}) \leq \mc{C}\epsilon(\|\hat{{\bf Y}}\|_{H^2(\Omega)}+ \|\hat{{\bf Y}}\|_{H^3(\Omega_{1/3})}.
\ee

It follows from \eqref{pi1} that
\be\no
&&\|\Pi_1-\Pi_2\|_{H^3(\Om)}\leq C\bigg(\|\ka_1(x)\mc{H}(\hat{{\bf v}}^1)(\bar{u}+\hat{v}_1^1)-\ka_2(x)\mc{H}(\hat{{\bf v}}^2)(\bar{u}+\hat{v}_1^2)\|_{H^2(\Om)}\\\no
&&\q\q+\sum_{i=2}^3\|\ka_1(x)\mc{H}(\hat{{\bf v}^1})\hat{v}_i^1-\ka_2(x)\mc{H}(\hat{{\bf v}^2})\hat{v}_i^2\|_{H^2(\Om)}\bigg)\\\no
&&\leq \mc{C}(\|\ka_1-\ka_2\|_{H^2(\Om)}+\|\hat{{\bf Y}}\|_{H^2(\Omega)})\leq \mc{C}\epsilon(\|\hat{{\bf Y}}\|_{H^2(\Omega)}+ \|\hat{{\bf Y}}\|_{H^3(\Omega_{1/3})}.
\ee
Furthermore, one can conclude from \eqref{dc} that
\be\no
&&\|\dot{{\bf Y}}\|_{H^3(\Om)}\leq C\bigg(\|\ka_1(x)\mc{H}(\hat{{\bf v}}^1)(\bar{u}+\hat{v}_1^1)-\ka_2(x)\mc{H}(\hat{{\bf v}}^2)(\bar{u}+\hat{v}_1^2)\|_{H^2(\Om)}\\\no
&&\q\q\q+\sum_{i=2}^3\|\ka_1(x)\mc{H}(\hat{{\bf v}^1})\hat{v}_i^1-\ka_2(x)\mc{H}(\hat{{\bf v}^2})\hat{v}_i^2\|_{H^2(\Om)}+\|\nabla (\Pi_1-\Pi_2)\|_{H^2(\Omega)}\bigg)\\\no
&&\leq \mc{C}(\|k_1-k_2\|_{H^2(\Om)}+\|\hat{{\bf Y}}\|_{H^2(\Omega)})\leq \mc{C}\epsilon(\|\hat{{\bf Y}}\|_{H^2(\Omega)}+ \|\hat{{\bf Y}}\|_{H^3(\Omega_{1/3})}).
\ee

Finally, note that $\psi$ solves
\be\label{c7}\begin{cases}
\sum_{i,j=1}^3 k_{ij}(\hat{{\bf v}}^1) \p_{ij}^2 \psi + \bar{k}_1(x_1) \p_1 \psi=F(\hat{{\bf v}}^1)-F(\hat{{\bf v}}^2)-\sum_{i,j=1}^3 (k_{ij}(\hat{{\bf v}^1})-k_{ij}(\hat{{\bf v}^2}))\p_{ij}^2 \psi_2\\\no
\q\q\q-\sum_{i,j=1}^3 (k_{ij}(\hat{{\bf v}^1}) \p_i \dot{v}_j^1-k_{ij}(\hat{{\bf v}^2}) \p_i \dot{v}_j^2) - \bar{k}_1(x_1) \dot{Y}_1,\q \text{in }\ \Om,\\
\psi(L_0,x')=0,\ \q \forall x'\in \mb{E},\\
(n_2\p_2 + n_3 \p_3 )\psi(x_1,x')=0,\ \text{on }\Ga_w.
\end{cases}\ee
As for \eqref{dc36} and above, one can derive
\be\no
&&\|\psi\|_{H^3(\Om)}+\|\psi\|_{H^4(\Om_{1/3})}\\\no
&&\leq C(\|F(\hat{{\bf v}}^1)-F(\hat{{\bf v}}^2)\|_{H^1(\Om)}+\sum_{i,j=1}^3\|k_{ij}(\hat{{\bf v}^1})-k_{ij}(\hat{{\bf v}^2})\|_{H^2(\Omega)}\|\psi_2\|_{H^4(\Omega)})\\\no
&&\q\q + C(\sum_{i,j=1}^3 \|k_{ij}(\hat{{\bf v}^1}) \p_i \dot{v}_j^1-k_{ij}(\hat{{\bf v}^2}) \p_i \dot{v}_j^2\|_{H^2(\Omega)}+ \|\dot{Y}_1\|_{H^2(\Omega)})\\\no
&&\leq \mc{C}\epsilon(\|\hat{{\bf Y}}\|_{H^2(\Omega)}+ \|\hat{{\bf Y}}\|_{H^3(\Omega_{1/3})}.
\ee
Collecting all the above estimates yields
\be\no
\|{\bf Y}\|_{H^2(\Omega)}+ \|{\bf Y}\|_{H^3(\Om_{1/3})}\leq \mc{C}\epsilon (\|\hat{{\bf Y}}\|_{H^2(\Omega)}+ \|\hat{{\bf Y}}\|_{H^3(\Omega_{1/3})})\leq \frac 12 (\|\hat{{\bf Y}}\|_{H^2(\Omega)}+ \|\hat{{\bf Y}}\|_{H^3(\Omega_{1/3})}).
\ee
Therefore $\mc{T}$ is a contraction mapping and there exists a unique ${\bf v}\in \mc{S}$ such that $\mc{T}{\bf v}={\bf v}$. This is the desired solution to the problem \eqref{belt}-\eqref{bbc}. The existence and properties of the sonic front can be proved similarly as for Theorem \ref{3dirro}. Thus the proof of Theorem \ref{beltrami} is finished.

\section{Appendix: the $H^4(\Omega)$ regularity of the solution to the divergence-curl system with normal boundary conditions in cylinders}\label{div-curl}\noindent

The goal here is to show the existence of a unique solution in $H^4(\Omega)$ to the following div-curl problem
\be\label{41}\begin{cases}
\text{div }{\bf v}=0,\ \ &\text{in }\Omega,\\
\text{curl }{\bf v}= {\bf f},\ \ \ \ &\text{in }\Omega,\\
{\bf v}\cdot {\bf n}=0,\ \ \ &\text{on }\partial\Omega,
\end{cases}\ee
where ${\bf n}=(n_1,n_2,n_3)^t$ is the unit outer normal to $\p\Omega$, ${\bf f}=(f_1,f_2,f_3)^t$ is a vector function in $(H^3(\Omega))^3$ satisfying $\text{div }{\bf f}=0$ and the compatibility conditions
\be\label{410}\begin{cases}
\p_1 f_1(L_0,x')=\p_1f_1(L_1,x')=0,\ \ &\forall x'\in \p\mb{E},\\
(n_2f_3-n_3f_2)(L_0,x')=(n_2f_3-n_3f_2)(L_1,x')=0, \ \ &\forall x'\in \p\mb{E}.
\end{cases}\ee
Seth in \cite{seth16} had proved the existence and uniqueness of ${\bf v}\in H^3(\Om)$ to \eqref{41} when ${\bf f}\in H^2(\Omega)$. Here one can extend the analysis in \cite{seth16} to  get the $H^4(\Omega)$ regularity.

\begin{proposition}\label{dcp}
{\it Assume that ${\bf f}\in H^3(\Omega;\mathbb{R}^3)$ satisfying $\text{div }{\bf f}=0$ in $\Omega$ and the compatibility conditions \eqref{410}. Then there exists a unique smooth solution ${\bf v}$ to \eqref{41} in $H^4(\Omega;\mathbb{R}^3)$ with
\be\label{411}
\|{\bf v}\|_{4,\Omega}\leq C(\Om)\|{\bf f}\|_{3,\Omega}.
\ee
}\end{proposition}

\begin{proof}
The unique solution to \eqref{41} can be given as ${\bf v}=\text{curl }{\bf u}$ with ${\bf u}$ solving
\be\label{42}\begin{cases}
\Delta {\bf u}= {\bf f},\ \ &\text{in }\Omega,\\
\text{div }{\bf u}=0,\ \ \ \ &\text{on }\partial\Omega,\\
{\bf u}\times {\bf n}=0,\ \ \ &\text{on }\partial\Omega.
\end{cases}\ee
Since $\text{div }{\bf f}=0$ in $\Om$, $\Delta \text{div }{\bf u}=0$ in $\Omega$, then the maximum principle implies that $\text{div }{\bf u}=0$ and thus $\text{curl }{\bf v}=\text{curl }\text{curl }{\bf u}=\Delta {\bf u}-\nabla \text{div }{\bf u}= {\bf f}$ in $\Omega$. The normal boundary condition ${\bf v}\cdot {\bf n}=(\text{curl }{\bf u})\cdot {\bf n}=0$ on $\p\Omega$ can be derived by the test function technique: for any $h\in C^1(\ol{\Omega})$
\be\no
&&\iint_{\p\Om}\{(\nabla\times {\bf u})\cdot {\bf n}\} h ds= \iiint_{\Omega}\text{div }(h (\nabla \times {\bf u})) dx \\\no
&&=\iiint_{\Om} \nabla h\cdot (\nabla\times {\bf u}) dx= \iint_{\p\Om} \nabla h \cdot ({\bf n}\times {\bf u}) ds=0.
\ee

The problem \eqref{42} can be decomposed as two independent subproblems:
\be\label{43}\begin{cases}
\Delta u_1= f_1,\ \ &\text{in }\Omega,\\
\p_1 u_1(L_0,x')=\p_1 u_1(L_1,x')=0,\ \ \ \ &\forall x'\in\mb{E},\\
u_1(x_1,x')=0,\ \ \ &\text{on }\Ga_w \ ,
\end{cases}\ee
and
\be\label{44}\begin{cases}
\Delta (u_2, u_3)= (f_2, f_3),\ \ &\text{in }\Omega,\\
(u_2,u_3)(L_0,x')=(u_2,u_3)(L_1,x')=0,\ \ \ \ &\forall x'\in \mb{E},\\
(n_2u_3-n_3 u_2)(x_1,x')=(\p_2 u_2+\p_3 u_3)(x_1,x')=0,\ \ \ &\text{on }\Ga_{w}.
\end{cases}\ee

Suppose that there exists a unique smooth solution $u_1\in H^5(\Omega)$ to \eqref{43}. Then
\be\no\begin{cases}
\p_1 (\p_2^2+\p_3^2) u_1(L_0,x')=\p_1 (\p_2^2+\p_3^2) u_1(L_1,x')=0,\ \ &\forall x'\in \mb{E},\\
\p_1^3 u_1(x_1,x')=0,\ \ \ &\text{on }\Ga_w,
\end{cases}\ee
from which one can deduce the compatibility condition
\be\label{431}
\p_1 f_1(L_0,x')=\p_1 f_1(L_1,x')=0,\ \ \text{on }\p\mb{E}.
\ee

Similarly, if there exists a unique smooth solution $(u_2,u_3)\in(H^5(\Omega))^2$ to \eqref{44}, then
\be\no\begin{cases}
(\p_2^3,\p_2 \p_3^2)u_2(L_0/L_1,x')=(\p_3\p_2^2,\p_3^3)u_3(L_0/L_1,x')=0,\ \ &\forall x'\in \mb{E},\\
(n_2\p_1^2 u_3- n_3 \p_1^2 u_2)(x_1,x')=\p_1^2 (\p_2 u_2 +\p_3 u_3)(x_1,x')=0,\ \ \ &\text{on }\Ga_w,
\end{cases}\ee
from which the following compatibility conditions hold
\be\no
(n_2 f_3-n_3 f_2)(L_0/L_1,x')=(\p_2 f_2+\p_3 f_3)(L_0/L_1,x')=0,\ \ \text{on }\p\mb{E}.
\ee

Consider the eigenvalue problem
\be\label{eig}\begin{cases}
-(\p_2^2+\p_3^2) q(x')= \al q(x'),\ \ &\forall x' \in \mb{E},\\
q(x')=0,\ \ &\text{on }\p\mb{E}.
\end{cases}\ee
It is well-known that \eqref{eig} has a family of eigenvalues $ 0<\al_1 \leq \al_2 \leq ... $  and $ \al_m \to +\infty$  as $m \to +\infty$ with associated eigenvectors $(q_m(x'))_{m=1}^{\infty}$, which constitute an orthonormal basis in $L^2(\mb{E})$ and an orthogonal basis in $H_0^1(\mb{E})$. Consider the eigenvalue problem
\be\label{2deige}\begin{cases}
-(\p_2^2+\p_3^2) {\bf e}(x')= \beta {\bf e}(x'),\ \ &\forall x' \in \mb{E},\\
(n_2e_3-n_3 e_2)(x')=(\p_2 e_2+\p_3 e_3)(x_1,x')=0,\ \ &\text{on }\p\mb{E},
\end{cases}\ee
where ${\bf e}(x')=(e_2(x'), e_3(x'))$. It was proved in \cite{seth16} that the eigenvalue problem \eqref{2deige} has a family of eigenvalues $ 0<\beta_1 \leq \beta_2 \leq ... $  and $ \beta_m \to +\infty$  as $m \to +\infty$ with associated eigenvectors $({\bf e}^m(x'))_{m=1}^{\infty}$, which constitute an orthonormal basis in $(L^2(\mb{E}))^2$.

Assume that the solutions to \eqref{43} and \eqref{44} respectively are given by
\be\no
u_1(x_1,x')=\sum_{m=1}^{\infty} p_m(x_1) q_m(x'),\ \ \ (u_2,u_3)(x_1,x')= \sum_{m=1}^{\infty} c_m(x_1) {\bf e}^m(x').
\ee
Substituting these into \eqref{43} and \eqref{44} respectively, one gets for every $m\geq 1$
\be\label{45}\begin{cases}
p_m''(x_1)=\al_m p_m(x_1) + f_{1m}(x_1),\ \forall x_1\in [L_0,L_1],\\
p_m'(L_0)=p_m'(L_1)=0
\end{cases}\ee
and
\be\label{46}\begin{cases}
c_m''(x_1)=\beta_m c_m(x_1) + h_{m}(x_1),\ \forall x_1\in [L_0,L_1],\\
c_m(L_0)=c_m(L_1)=0,
\end{cases}\ee
where
\be\no
&&f_{1m}(x_1)=\iint_{\mb{E}} f_1(x_1,x')q_m(x') dx',\\\no
&&h_m(x_1)=\iint_{\mb{E}} e_2^m(x')f_2(x_1,x')+ e_3^m(x') f_3(x_1,x') dx'.
\ee
It is easy to derive from \eqref{45} and \eqref{46} that
\be\label{47}
&&\int_{L_0}^{L_1} \sum_{i=0}^5 \alpha_m^{5-i} |p_m^{(i)}(x_1)|^2 dx_1\leq C_0\int_{L_0}^{L_1} \sum_{i=0}^3 \alpha_m^{3-i} |f_{1m}^{(i)}(x_1)|^2 dx_1,\\\label{48}
&&\int_{L_0}^{L_1} \sum_{i=0}^5 \beta_m^{5-i} |c_m^{(i)}(x_1)|^2 dx_1\leq C_0\int_{L_0}^{L_1} \sum_{i=0}^3 \beta_m^{3-i} |h_{m}^{(i)}(x_1)|^2 dx_1.
\ee
Using \eqref{47}, one may argue as in Lemma \ref{li50} to derive the estimate
\be\label{49}
\|u_1\|_{H^5(\Omega)}^2\leq C(\Omega)\int_{L_0}^{L_1}\sum_{m=1}^{\infty} \sum_{i=0}^5 \alpha_m^{5-i} |p_m^{(i)}(x_1)|^2 dx_1\leq C(\Omega)\|f_1\|_{H^3(\Omega)}^2.
\ee

Consider the problem
\be\label{55}\begin{cases}
(\p_2^2+\p_3^2) (w_2(x'), w_3(x'))= (g_2(x'), g_3(x')),\ \ &\text{in }\mb{E},\\
(n_2w_3-n_3 w_2)(x')=(\p_2 w_2+\p_3 w_3)(x')=0,\ \ \ &\text{on }\p\mb{E}.
\end{cases}\ee
It is proved in \cite{seth16} that \eqref{55} is an elliptic system with complementing boundary conditions in the sense defined in \cite{adn64}. Thus by the elliptic estimates in \cite{adn64} and the uniqueness for smooth solution to \eqref{55}, for any nonnegative integer $s\geq 0$ there holds
\be\no
\|(w_2,w_3)\|_{H^{s+2}(\mb{E})}\leq C(\mb{E})\|(g_2,g_3)\|_{H^{s}(\mb{E})}.
\ee
This, together with \eqref{48}, would imply that
\be\label{50}
\|(u_2,u_3)\|_{H^5(\Omega)}^2\leq C(\Omega)\int_{L_0}^{L_1}\sum_{m=1}^{\infty}  \sum_{i=0}^5 \beta_m^{5-i} |c_m^{(i)}(x_1)|^2 dx_1\leq C(\Omega)\|(f_2,f_3)\|_{H^3(\Omega)}^2.
\ee
The estimate \eqref{411} follows from \eqref{49} and \eqref{50} directly. The proof of Proposition \ref{dcp} is completed.
\end{proof}

{\bf Acknowledgement.} Weng is supported by National Natural Science Foundation of China (Grants No. 12071359, 12221001). Xin is supported by the Zheng Ge Ru Foundation, Hong Kong RGC Earmarked Research Grants CUHK-14301421, CUHK-14301023, CUHK-14302819 and CUHK-14300819, and the key projects of NSFC Grants No.12131010 and No. 11931013.

{\bf Data Availability Statement.} No data, models or code were generated or used during the study.

{\bf Conflict of interest.} On behalf of all authors, the corresponding author states that there is no conflict of interests.

\end{document}